\documentclass[11pt,a4paper]{article}

\usepackage[dvipsnames]{xcolor}
\usepackage[all,cmtip]{xy}
\entrymodifiers={+!!<0pt,\fontdimen22\textfont2>}

\usepackage{fouriernc}
\usepackage[english]{babel}         
\usepackage{verbatim}               
\usepackage[T1]{fontenc}

\usepackage{amsmath}
\usepackage{amsfonts}
\usepackage{amssymb}
\usepackage{mathrsfs}    
\usepackage{amsthm}

\usepackage{todonotes}

\usepackage[colorlinks=true,linkcolor=MidnightBlue,citecolor=OliveGreen]{hyperref}           

\swapnumbers

\theoremstyle{plain}
\newtheorem*{thrm}{Theorem}
\newtheorem{thm}{Theorem}[section]

\newtheorem{prop}[thm]{Proposition}
\newtheorem{cor}[thm]{Corollary}
\newtheorem{lemma}[thm]{Lemma}

\newtheorem{question}[thm]{Question}
\theoremstyle{definition}

\newtheorem{defn}[thm]{Definition}
\newtheorem{defns}[thm]{Definitions}

\theoremstyle{remark}
\newtheorem{remark}[thm]{Aside}
\newtheorem{eg}[thm]{Example}
\newtheorem{egs}[thm]{Examples}
\newcommand{\proofof}[1]{\end{#1}\begin{proof}}

\makeatletter
\renewcommand\section{\@startsection {section}{1}{\z@}%
  {-3.5ex \@plus -1ex \@minus -.2ex}{2.3ex \@plus.2ex}%
  {\normalfont\large\bfseries}}
\renewcommand\subsection{\@startsection{subsection}{2}{\z@}%
  {-3.25ex\@plus -1ex \@minus -.2ex}{1.5ex \@plus .2ex}%
  {\normalfont\bfseries}}

\newcommand{\lie}[1]{\mathfrak{#1}}
\newcommand{\sh}[1]{\mathcal{#1}}
\newcommand{\N}{{\mathbb N}}
\newcommand{\Z}{{\mathbb Z}}
\newcommand{\Q}{{\mathbb Q}}
\newcommand{\R}{{\mathbb R}}
\newcommand{\C}{{\mathbb C}}
\newcommand{\F}{{\mathbb F}}

\newcommand{\G}{{\mathbb G}}
\renewcommand{\P}{{\mathbb P}}
\newcommand{\A}{{\mathbb A}}

\newcommand{\tens}{\mathbin{\otimes}}

\newcommand{\Hom}{\mathrm{Hom}}

\DeclareMathOperator*\colim{colim}
\DeclareMathAlphabet{\mathrmsl}{OT1}{cmr}{m}{sl}

\newcommand{\rssymb}[2]{\newcommand{#1}{\mathrmsl{#2}} }
\newcommand{\oper}[3][n]{\newcommand{#2}{\mathop{\mathrm{#3}}%
\ifx n#1\nolimits\else\limits\fi} }
\newcommand{\rsoper}[3][n]{\newcommand{#2}{\mathop{\mathrmsl{#3}}%
\ifx n#1\nolimits\else\limits\fi} }

\oper\Ad{Ad}
\oper\ad{ad}

\oper\val{val}
\oper\coker{coker}
\oper\mult{mult}
\oper\Iso{Iso}
\oper\End{End}
\oper\Aut{Aut}
\oper\Sub{Sub}
\oper\Alt{Alt}
\oper\Ext{Ext}
\oper\Pic {Pic}
\oper\Sym{Sym}
\oper\Spec{Spec}
\oper\Spf{Spf}
\oper\Sp{Sp}
\oper\Spa{Spa}
\oper\Proj{Proj}
\rsoper\divg{div}
\rsoper{\sym}{sym}
\rsoper{\alt}{alt}
\rsoper\trace{tr}
\rssymb\id{id}

\newcommand{\thismonth}{\ifcase\month\or
  January\or February\or March\or April\or May\or June\or
  July\or August\or September\or October\or November\or December\fi
  \space\number\year}

\usepackage{stmaryrd}

\newcommand{\QC}{\mathrm{QC}}
\newcommand{\Alg}{\mathrm{Alg}}

\newcommand{\Pro}{\mathrm{Pro}}
\newcommand{\Sh}{\mathrm{Sh}}
\newcommand{\PSh}{\mathrm{PSh}}

\title{Affine manifolds are rigid analytic spaces in characteristic one \\ I: Toric formal geometry}
\author{Andrew W. Macpherson}
\begin{document}

\maketitle
\begin{abstract}I extend the definitions of schemes relative to monoids with zero - and therefore, toric geometry - to the world of formal schemes. This expands the usual framework to include, for instance, models for Mumford's degenerating Abelian varieties. 

Following the usual toric paradigm, normal formal $\F_1$-schemes can be classified in terms of certain cone complexes, and their properties understood in combinatorial terms. I use this to give a simple algebraisation criterion. 

I also reformulate the traditional notions of separated and proper morphism in a manner amenable to the context of relative formal geometry, and give characterisations in terms of the topology of cone complexes.\end{abstract}

\tableofcontents

\section{Introduction}

Toric geometry was introduced, in the wake of the pioneering work \cite{toroid}, as an organising principle for the ubiquitous appearance of monomial techniques in algebraic geometry. Abstractly, it can be viewed as an attempt to answer the question
\begin{align}\label{one}\text{\emph{Which algebraic varieties can be defined without addition?}}\end{align}
The absence of addition - or equivalently, the large torus symmetry - makes understanding such varieties substantially simpler as compared to more general varieties. Indeed, toric varieties famously can be captured entirely in terms of a combinatorial object: a \emph{fan} in a rational vector space. For this reason, they have provided a fertile testing ground for theories in algebraic geometry that are too difficult to tackle in more general situations.

In this sequence of papers, I address the following generalisation of question \ref{one}:
\begin{align}\label{two}\text{\emph{Which analytic spaces can be defined without addition?}}\end{align}
The analogue to the answer that was provided before by toric geometry is here given by analytic spaces with the structure of a \emph{maximal torus fibration} over an affine manifold. 
The affine manifold adopts the r\^ole of parametrising object that in algebraic geometry was played by the fan. 

While the hypothesis of being defined by monomials is still very restrictive locally, the parametrisation by affine manifolds shows that passing to the analytic setting makes available much more interesting \emph{global} geometry. For instance, part of Mumford's theory of completely degenerate Abelian varieties \cite{Mumford} can be described in this way; skip to \cite[\S5.7]{part2} for a discussion. Still, toric analytic spaces represent a substantial simplification of the theory of analytic spaces in general.

The correspondence between maximal torus fibrations over $\C$ and their parametrising affine manifolds has an elementary geometric description: by definition, an affine manifold $B$ carries a certain local system $\Lambda_B$ of \emph{integer 1-forms}, and the corresponding complex analytic space is defined to be the torus bundle
\[ \pi:TB/\epsilon\Lambda_B^\vee\rightarrow B \]
depending on a parameter $\epsilon$, whose monomial holomorphic co-ordinates $\exp{\epsilon(\pi^*F+2\pi idF)}$ are indexed by affine functions $F$ on $B$ with integer differential. 

The focus of these papers is rather in the rigid analytic, or `non-Archimedean', paradigm. In other words, the basic object of study is a \emph{formal degeneration} - which already appeared in the previous formula in the form of the parameter $\epsilon$. From this perspective, the geometric objects that become available are particularly degenerate formal families that not only have toroidal crossings in the sense studied in \cite{toroid}, but all of the components of whose central fibre are toric varieties. Such families arise as the $e^\epsilon$-adic completion of the corresponding complex-analytic objects.

Of course, these ideas have also found plenty of attention in the literature - indeed, the basic picture is already implicit in the work of Mumford. A particular source of inspiration for me has been the Gross-Siebert programme in mirror symmetry \cite{Grossbook} and the paper \cite{KoSo2}, in which the authors define a notion of `non-Archimedean torus fibration' which was the starting point for the present work. 

\subsection*{Characteristic one}

A more fundamental approach to question (\ref{one}) is to try to define \emph{algebraic geometry itself} without addition. This is the method of `relative' algebraic geometry (cf. \cite{Toen}). While toric geometry has the advantage of being very \emph{concrete}, relative geometry is heavily modelled on the abstract framework of traditional algebraic geometry and so is very \emph{robust}. It is possible to push through much of the basic structure of \cite{EGA} in a very general setting (see, for example, \cite{Durov}).

Under this approach, one is free to replace the defining algebraic objects of geometry over $\Z$ - the commutative rings - with a commutative algebraic object of one's choosing. Since we are interested in geometry without addition, the obvious choice here is to work with \emph{multiplicative monoids}. General principles then provide us with a category of `schemes' equipped with a structure sheaf of monoids.

The essential features of any object defined in such a geometry are necessarily independent of `coefficients' - for instance, considerations of characteristic. For this reason, algebraic geometry relative to monoids formed the basis of some early attempts to make sense of geometry relative to the mysterious object $\F_1$, the so-called `prime field of characteristic one' (\cite{Deitmartoric}, \cite{Connes}). It is this connection that gives the present sequence of papers - and our category of schemes relative to monoids - its name.\footnote{I should point out that arithmeticians quickly realised that monoids were completely inadequate for questions of arithmetic, and turned to consider more sophisticated objects \cite{Borger,Deitmarses,Durov,Lorscheid}. Since geometry, rather than arithmetic, is the primary motivation for this paper, beyond a few scattered comments I do not mount any attempt to study such ideas. See remark \ref{RIG_ADDITION}.}

Apart from being more structured than toric geometry, geometry in characteristic one is also more general - but not \emph{too much} more general, as the following basic result shows:
\begin{thrm}[\cite{Deitmartoric}]Let $X$ be a normal, connected $\F_1$-scheme, separated and of finite type over $\F_1$. The base change of $X$ to $\Z$ is a toric variety.\footnote{Strictly speaking, the notion of \emph{separated} $\F_1$-scheme has not, as far as I know, been defined in the literature. In this paper, I address this deficiency (\S\ref{SEP}).}\end{thrm}
In other words, under some natural - and rather light - hypotheses of an algebro-geometric nature, $\F_1$-geometry can be understood using the machinery of cones and fans. A large part of these papers is devoted to developing a more precise version of this statement and an analogue in the analytic setting.


\subsection*{Cone complexes}

In this paper, I address the case of formal geometry. Although we make no discussion of analytic spaces or affine manifolds until the sequel, the nature of the definition will be such that many of the salient features already appear at the formal level.

Before introducing the objects that `parametrise' toric formal schemes, I would like to present a more precise formulation of Deitmar's theorem:

\begin{thrm}[\ref{FAN_THM}]The category of connected, normal, separated, finite type $\F_1$-schemes and non-boundary morphisms is equivalent to the category of rational polyhedral fans.\end{thrm}

The hypotheses appearing in this statement can be divided into three groups. The first - normal and non-boundary - is essential to get any kind of description in terms of cones. Here `non-boundary' essentially means that the morphisms are restricted not to land in any closed subscheme. The closed subschemes are the part of an $\F_1$-scheme that becomes the toric boundary upon base change to $\Z$.

The second group consists of the `finite type' assumption. It is easy to alleviate this hypothesis, should we desire, by considering more general kinds of cone.

The third group of assumptions are of a topological nature, and are therefore almost completely invisible from the usual algebro-geometric standpoint. To help understand them, we will increase generality, and ask: what do \emph{completely general} normal, locally finite type $\F_1$-schemes look like? The traditional picture has already given us to understand \emph{affine} objects: they are the fans consisting of a single cone embedded in a vector space $N$. Since each normal, locally finite type $\F_1$-scheme $X$ is glued together from affine pieces, a global object can be understood in terms of a certain \emph{cone complex} $\Sigma_X$, which consists of a collection of embedded cones $\sigma\subset N$ glued together along faces in a way that respects the embedding in $N$.

Now we can ask what our topological hypotheses look like on cone complexes. The connectedness assumption is somewhat obvious, so we may restrict to the connected case. Now observe that since every connected, integral scheme is contractible onto its generic point, there is no topological obstruction to patching together the embeddings of constituent cones into $N$. We therefore get a globally defined, locally linear map
\[ \delta:\Sigma_X\rightarrow N_X \] 
which, borrowing terminology from the manifolds literature, we may call the \emph{developing map}. By unwinding the definitions, it is not difficult to see:
\[X\text{ separated} \qquad\Leftrightarrow \qquad \delta\text{ injective.}\]
In real life, one usually wishes to restrict attention to separated objects, and so this result renders the abstraction of cone complexes somewhat pointless for algebraic geometry. 

The analogue of the above statement is \emph{false} in formal geometry. Indeed, it is no longer true that integral formal schemes are $0$-types, and so $\delta$ need not even have a global extension. The theory of cone complexes is therefore unavoidable in formal geometry.

\subsection{Summary of results}
The main results of this paper give combinatorial characterisations of separated and proper morphisms of formal $\F_1$-schemes. However, the literature has apparently so far omitted to provide definitions of these terms, other than to remark that the definitions of separated and proper morphisms given in \cite{EGA} clearly do not work `out of the box'. Therefore, before stating the results, I must clarify what I am talking about.

My approach to defining propriety in this paper has been in the spirit of generalising the valuative criterion. Rather than giving the general definition here, I provide a list of equivalent characterisations in cases where the criteria can be couched in more familiar terms (but see \S\ref{STRUCT}):
\begin{thrm}Let $S$ be a formal scheme over $\Z$ or $\F_1$. Let $f:X\rightarrow S$ be of finite type. The following conditions are equivalent:
\begin{enumerate}\item ($S/\Z$) $f$ is separated and universally closed  (Thm. \ref{SEP_ME=EGA});
\item ($S$ qcqs and admits an ample bundle) for every open subset $U\subseteq X$ quasi-projective over $S$, there exists a Chow cover
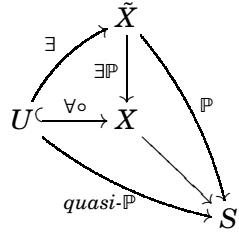
\[\xymatrix{ & \tilde X\ar[d]_{\exists \P}\ar@/^/[ddr]^{\P} \\ U\ar@{^{(}->}[r]^{\forall\circ}\ar@/^/[ur]^\exists\ar@/_/[drr]_-{\text{quasi-}\P} & X\ar[dr] \\ && S}\]
projective over $S$ and admitting a section over $U$ (Thm. \ref{SEP_CHOW});
\item ($S$ Noetherian) $f$ satisfies the valuative criterion for propriety (Cor. \ref{SEP_VAL}).\end{enumerate} For case \emph{iii)}, it is enough that every embedded subscheme of $X$ surjects onto the embedded closure of its image in $S$, and that this remains true after any base change (Thm. \ref{SEP_GROTHENDIECK_OVERCONVERGENT}).\footnote{Note that over $\F_1$, `embedded' means something more general than `closed embedded'; see \S\ref{MORP_EMBED}.}
\end{thrm}
Note that the Chow property \emph{ii)} stated here is actually much stronger than the usual statement of the Chow lemma, which only asserts the \emph{existence} of an open dense $U$. It is this property that follows immediately from our definition of propriety. The challenge lies in establishing sufficient conditions for checking it in practice. 

With rigid geometry in mind, it is actually natural to give more attention to \emph{overconvergent} morphisms. Intuitively, these are morphisms that satisfy the valuative criterion without necessarily being quasi-compact. This class of morphisms is not visible in ordinary toric geometry, since a connected integral scheme can be overconvergent only if it is proper. However, they are essential for understanding the rigid analytic world considered in the sequel \cite{part2}.

A large part of this paper (\S\ref{SEP}) is devoted to developing this machinery. See \S\ref{STRUCT} below for a summary of the approach.

\

An affine formal scheme is always obtained as the formal completion of a scheme. Topologically, the effect of formal completion is to remove certain open sets. On the combinatorial side, we must therefore respond by removing certain faces of the corresponding cone. To every affine, normal formal $\F_1$-scheme we may therefore associate an object in a category $\mathbf{Cone}_*^N$ of \emph{punctured cones} embedded in a rational vector space $N(\Q)$. 

Globalising, we obtain also a category $\mathrm{C}\mathbf{Cone}_*^N$ of \emph{punctured cone complexes} and a \emph{cone complex functor}
\[ \Sigma:\mathbf{FSch}_{\F_1}^\mathrm{n/nb}\rightarrow\mathrm{C}\mathbf{Cone}^N_* \]
from the category of normal formal $\F_1$-schemes with non-boundary morphisms. From general principles a classification - stated here, for simplicity, under additional finiteness assumptions - follows:

\begin{thrm}[\ref{PFAN_EQUIVALENCE}]The cone complex functor restricts to an equivalence of categories between the category of normal, locally Noetherian formal $\F_1$-schemes with non-boundary morphisms and the category of integral polyhedral punctured cone complexes.\footnote{A tweak of the definition of $\Sigma$ allows this statement and the two below to work for formal schemes locally of finite type over a possibly non-Noetherian rank one valuation $\F_1$-algebra.}\end{thrm}

As indicated above, any punctured cone complex $\Sigma_X$ comes equipped with a locally defined developing map $\delta:\Sigma_X\rightarrow N_X$. The primary results on cone complexes are stated in terms of this map:
\begin{thrm}[\ref{PFAN_SEPARATED}]Let $X$ be a locally Noetherian, integral formal $\F_1$-scheme. Then $X$ is separated if and only if the developing map of $\Sigma_X(\R)$ is a local immersion.\end{thrm}

\begin{thrm}[\ref{PFAN_PROPER}]Let $f:X\rightarrow S$ be a non-boundary morphism between locally Noetherian, integral formal $\F_1$-schemes. Suppose that $f$ is paracompact and locally of finite type.

Then $f$ is overconvergent if and only if $\Sigma_X(\R)\rightarrow\Sigma_S(\R)$ is a topological submersion. In particular, in this case its fibres are manifolds without boundary.\end{thrm}

If $f$ is overconvergent, the developing map equips the fibres of the induced map of cone complexes with a canonical structure of a \emph{radiant affine manifold} (cf. \cite{radiance}) with an action of $\R^\times_{>0}$. The affine manifold promised in the title and described in \cite{part2} will be the quotient by this action.

\

Finally, and somewhat incidentally to the rest of the development, the intuition provided by cone complexes permits us to prove an algebraisation criterion:

\begin{thrm}[\ref{PFAN_ALG}]Let $X$ be an integral formal scheme over $\F_1$. The multiplicative group $K_X^\times$ of the function field is a locally constant sheaf.

Suppose that $X$ is connected; then it is algebraisable if and only if $K_X^\times$ is constant. In this case, the algebraisation can be made functorial.\end{thrm}

Of course, algebraic $\F_1$-schemes being toric varieties, algebraisability over $\F_1$ is a much stronger hypothesis than algebraisability of a base change to $\Z$. Such examples are completely uninteresting from the perspective of affine manifold theory.

\subsection{Structural features}\label{STRUCT}

Schemes relative to monoids have already found attention from various perspectives in the literature. Much of sections \ref{FSCH}, \ref{MORP}, and \ref{INT} are devoted to further developing the basic properties of schemes and formal schemes to bring them more into line with the early parts of \cite{EGA}. Since these notions behave, with mild deviations, much like their counterparts over $\Z$, I relegate any summary of these sections to the body of the text.

\paragraph{Criteria for separation, propriety, overconvergence}The notions of separated and proper morphisms are traditionally defined in terms of \emph{closed maps}. As some authors \cite[\S6.5.20]{Durov} have already noted, in algebraic geometry relative to monoids there aren't enough closed sets for these definitions to produce something meaningful.

The generalisation from schemes to rigid analytic spaces, too, diminishes the relevance of closed subsets to geometry, though the rigid geometry community have nonetheless managed to push something through along classical lines \cite[II.7.5]{FujiKato}. 

In \cite[\S6.5.23]{Durov}, it is suggested that one could approach this problem by generalising the notion of closed subscheme to \emph{embedded} subschemes - that is, subschemes defined by equations (cf. \S\ref{MORP_EMBED}). Although we discuss that approach, I have chosen to focus instead on adapting the \emph{valuative criteria}. The only unsatisfactory aspect of the valuative criteria as they stand is their reliance on valuation rings.

So what we are looking for is a notion of valuative criterion \emph{without the valuation}; that is, the problem of extending a morphism $U\rightarrow X$ (over a base $S$) along a \emph{completely arbitrary} open immersion $U\subseteq V$:
\[\xymatrix{U\ar@{^{(}->}[r]\ar[d] & V\ar[d]\ar@{-->}[dl]_? \\ X\ar[r] & S}\]
Of course, it's unrealistic to expect that such an extension could exist without first modifying $V$ - the only reason we got away with it in the classical valuative criterion is that every finitely generated ideal in a valuation ring is principal, and hence cannot be blown up. 

The geometric input to the theory lies in specifying exactly what kinds of modifications are allowed. My choice of the word `modification' is no coincidence: for algebraic schemes I will allow any scheme \emph{projective} over $V$ admitting a section over $U$. Such a $V$-scheme is called an \emph{overconvergent neighbourhood} of $U/V$.

So for overconvergence of $X/S$, we are looking for an extension
\[\xymatrix{U\ar@{^{(}->}[r]\ar[d] & \tilde V\ar@{-->}[dl]_-{!}\ar[r]^\P & V\ar[d] \\ X\ar[rr] && S}\]
of $U\rightarrow X$ to an overconvergent neighbourhood $\tilde V$ of $U/V$, and this extension should be unique up to refinement - that is, further modification - of $\tilde V$. A morphism is proper, by definition, if it is overconvergent, quasi-compact, and quasi-separated.

The class of proper morphisms resulting from this definition turns out to be exactly those morphisms satisfying a strong version of the Chow lemma locally on the base (thm. \ref{SEP_CHOW}). The proof of this lemma for finite type, separated, and universally closed morphisms of schemes over $\Z$ - and hence that this theory recovers the usual one as presented in \cite{EGA} - relies on flattening stratification. An analogous statement for $\F_1$ seems quite distant at the moment, and so our definition is, at least \emph{a priori}, rather more powerful than any definition based around the notion of embedded subscheme - but see \S\ref{SEP_IMAGES}.

\

Another benefit to this kind of approach is that most of the basic structural results are `soft', meaning that they can be abstracted. We can therefore systematically define separated, proper, and overconvergent morphisms in any of the menagerie of topoi considered in this paper, the only input being a class of morphisms $\P$ satisfying a short list of conditions.

This allows us to mount a comparison of these notions in contexts related by morphisms of topoi. The comparison principles outlined in \S\ref{SEP_COMPARE} are of a somewhat technical nature; however, in practice they tend to rely on genuinely geometric input. The conclusions of the theory for formal $\F_1$-schemes are:
\begin{itemize}
\item overconvergence of morphisms locally of finite type between formal schemes can be checked using finite type reduced schemes as test spaces (\S\ref{SEP_FSCH});
\item if the the source is integral or the base is Noetherian, one can check using only \emph{normal} test spaces (\S\ref{INT});
\item in the latter case, one can even use only the specific test space $V=\A^1_{\F_1}=\Spec\F_1[t]$.\end{itemize}
The arguments for the first item are valid over $\Z$, and those for the second can most likely be extended with minor modifications; since the conclusions are anyway well-known in that setting, I didn't try to confirm this.

\paragraph{Capturing geometry in terms of combinatorial data} The constructions of combinatorial gadgets for $\F_1$-schemes, formal schemes, and rigid analytic spaces in \cite{part2} all follow the same general programme. I take a moment here to summarise the salient features. The initial steps are:
\begin{enumerate}\item assignment of combinatorial object to affine scheme;
\item description of open sets;
\item when does the combinatorial object determine the affine object?
\end{enumerate}
For example, in the case of finite type $\F_1$-schemes - approximately, toric varieties - these initial steps go:
\begin{itemize}
\item[-] an affine toric variety gives rise to a \emph{rational polyhedral cone} in a vector space with a lattice;
\item[-] open sets correspond to \emph{faces} of the cone;
\item[-] the cone determines the variety precisely when it is \emph{normal}.
\end{itemize}

This description becomes more useful when we have also:
\begin{itemize}
\item[iv)] description of points;
\item[v)] topological realisation.
\end{itemize}
In toric geometry, the integer points are usually interpreted as one-parameter subgroups of the embedded torus. More generally, the $H$-points, for general additive subgroups $H\subseteq\R$, have the interpretation of certain rank one valuations. Over $\F_1$, the notion of a valuation is essentially equivalent to that of a non-boundary \emph{jet}, that is, a point valued in the spectrum of a valuation $\F_1$-algebra with value group $H$.

In the special case $H=\R$, one can use the order topology to obtain a \emph{topological realisation} $\sigma(\R)$ of a combinatorial object $\sigma$. In general, this will be some kind of convex, semi-linear set inside a real vector (or later \cite{part2}, affine) space.

The affine picture well-understood, the next stage is globalisation:
\begin{itemize}
\item[vi)] glueing.\end{itemize}
This is an exercise in abstract nonsense: the language of locally representable sheaves (cf. \S\ref{TOPOS}) reduces the question of globalising a functorial construction to that of checking a couple of basic categorical properties, to whit,
\begin{itemize}
\item[vi\emph{a})] is it flat?\footnote{This is a fancy way of saying `would be left exact, if the category were finitely complete'. We need this for a stupid reason: our categories of combinatorial data do not contain an object representing the empty set!}
\item[vi\emph{b})] does it preserve open immersions?
\item[vi\emph{c})] does it preserve coverings?
\end{itemize}
The first item does not, for the most part, present a serious difficulty (but see the proof of theorem \ref{PFAN_ALG}). The second criterion will follow automatically from our understanding of item ii), and the third does not even appear until \cite{part2}.

The topological realisations of constituents glue together to give a compactly generated Hausdorff space $\Sigma_X(\R)$ whose topological properties reflect those of $X$. For instance, a locally Noetherian scheme $X$ is quasi-compact if and only if $\Sigma_X(\R)\setminus 0$ is conically compact, that is, compact up to the action of $\R^\times_{>0}$.

More interesting conditions on $X$ manifest in the study of the
\begin{itemize}\item[vii)] developing map, which we use to provide
\item[viii)] overconvergence criteria.\end{itemize}
We already concluded that overconvergence for normal schemes is detected by normal test spaces, and so it follows that it can equally be calculated in the category of cone complexes. 

Thus once we understand how \emph{modifications} look on the category of cone complexes, we can understand criteria for separation, propriety, and overconvergence directly from the definition via extension problems.

Returning to the example of classical toric geometry, we obtain a rephrasing of some well-known facts:
\begin{itemize}\item[-] to an algebraic $\F_1$-scheme we associate a locally representable presheaf on the category of cones - in other words, a complex of cones, glued together along faces, such that no two faces of the same cone are glued together;
\item[-] the developing map is defined globally on each connected component; 
\item[-] a birational modification is a subdivision;
\item[-] the separation criterion is that $\delta$ is an \emph{immersion}, and therefore gives rise to a \emph{fan} in a rational vector space;
\item[-] the overconvergence criterion is that $\delta$ is a \emph{homeomorphism}; that is, the fan has support the whole space.
\end{itemize}
The real crux is that for formal schemes, the developing map is not globally defined, and so the criteria for separation and overconvergence have to be replaced with \emph{local} conditions. This is the punchline of part I (\S\ref{PFAN_CRITERIUM}).

\subsection{Regarding topoi}\label{TOPOS}

In this paper, the notion of a geometry admitting an \emph{atlas} by some specified class of objects is formalised by the concept of \emph{locally representable sheaf}. The context in which this makes sense is captured by the following definition.

\begin{defns}\label{TOPOS_DEF}A \emph{spatial geometric context} (more briefly, \emph{spatial theory}) $(\Sh\mathbf C,\sh U)$ is a topos $\Sh\mathbf C$ together with a composable class $\sh U$ of monomorphisms, called \emph{open immersions}. The poset $\sh U_{/X}$ of open immersions into $X\in\Sh\mathbf C$ are required to satisfy also:
\begin{enumerate}
\item stability for base change, i.e. for each $X^\prime\rightarrow X$, base change induces a map $\sh U_{/X}\rightarrow\sh U_{/X^\prime}$;
\item for each $X$, $\sh U_{/X}$ is a complete lattice (i.e. being an open immersion is local);
\item every covering (epimorphism) of $X$ in $\Sh\mathbf C$ can be refined to a covering in $\sh U_{/X}$.
\end{enumerate}
Let $\mathbf C$ be a site for $\Sh\mathbf C$. An object $X\in\Sh\mathbf C$ is \emph{locally representable} in $\mathbf C$ if $\sh U_{/X}$ is generated by objects of $\mathbf C$. The site is \emph{spatial} if it generates $\sh U$ in the sense that every representable object is locally representable. 

Conversely, let $\mathbf C$ be a category with a Grothendieck topology generated by a specified class $\sh U^\mathrm{aff}$ of monomorphisms, called `affine' open immersions, stable for base change, composition, and descent. Then $\Sh\mathbf C$ carries a natural structure of a spatial theory for which $\mathbf C$ is a spatial site; the open immersions are those monomorphisms that are, after representable base change, exhausted by affine open immersions.

We will always implicitly consider spatial theories equipped with a spatial site, which we may as well assume includes all locally representable objects.

A \emph{spatial geometric morphism} between spatial theories is a geometric morphism of topoi whose pullback preserves open immersions. A flat functor between spatial sites that preserves open immersions and coverings extends to a spatial geometric morphism.\end{defns}

Every object $X$ of a spatial theory has a corresponding small site $\Sh(X):=\Sh(\sh U_{/X})$, which is by construction \emph{localic}. The locale associated to an object is denoted by the same letter. A morphism of objects induces a geometric morphism of small sites, and hence a continuous map of locales. 

If $\Sh\mathbf C$ has enough points, then $X$ is actually an honest sober topological space. By Deligne's theorem, this occurs if the Grothendieck topology on a spatial site $\mathbf C$ is generated by \emph{finite} coverings. In this series, we will always be working with topoi admitting enough points, and hence confuse locales with spaces.

\begin{lemma}\label{TOPOS_FUNCT}Let $f:\Sh\mathbf C\rightarrow\Sh\mathbf D$ be a spatial geometric morphism preserving spatial sites. For each locally representable $X\in\Sh\mathbf D$, $f^{-1}X$ is locally representable, and $f^{-1}|\sh U_{/X}$ is dual to a continuous map $f^{-1}X\rightarrow X$ of locales.\end{lemma}

I invite the reader to consult \cite{Toen} for more details of this approach to glueing; though written in a more restricted setting than the above, it is relevant for many examples of interest.

\paragraph{Coverings}
All coverings used in this paper are really \emph{hypercoverings}, that is, they include the data of the intersections. To be precise, a covering of an object $X$ of a (spatial) site $\mathbf C$ is a local isomorphism $X_\bullet\rightarrow X$ in $\PSh\mathbf C$; that is, a morphism that becomes invertible in $\Sh\mathbf C$. A \emph{member} of a covering is a representable object $X_i\rightarrow X_\bullet$ such that $X_i\rightarrow X$ is in $\sh U$.

\paragraph{Algebraic spaces}
It is also possible to define a broader class of `algebraic spaces' consisting of any object that can be represented as a colimit of representable objects and open immersions. Many of the things we prove about locally representable objects hold equally well for these algebraic spaces.

\section{$\F_1$-schemes and formal schemes}\label{FSCH}

\subsection{Review: $\F_1$-algebras}

In \cite{Deitmartoric}, it was proposed to define $\F_1$-geometry in terms of monoids. In this paper, we use a slight modification: in order to make sense of \emph{vanishing loci}, we want our monoids to have a zero element (as in \cite{Connes}). Thus our $\F_1$-algebras will be commutative algebra objects in the closed symmetric monoidal category of pointed sets.

To put a commutative monoid $(Q,+)$ over $\F_1$, you first write its elements as exponents so that the monoid law can be written multiplicatively:
\[ z^X\cdot z^Y = z^{X+Y}. \]
You then adjoin a \emph{sink} or \emph{absorbing element} $0$:
\[ 0\cdot z^X=0,\quad \forall X\in Q. \]
The resulting multiplicative monoid can be written $\F_1[z^Q]$ (or simply $\F_1[Q]$ if $Q$ is already written multiplicatively or we are being lazy). Note that even if $Q$ already had a sink, you need to add a new one.

This defines a faithful left adjoint
\[  -\tens\F_1:\mathbf{Mon} \rightarrow \mathrm{Alg}_{\F_1} \]
to the inclusion of the category $\mathrm{Alg}_{\F_1}$ of monoids with zero and homomorphisms that respect zero into the category $\mathbf{Mon}$ of all monoids. See \cite[\S3.1]{Connes} (in which the category $\mathrm{Alg}_{\F_1}$ is denoted $\mathfrak{Mon}$) for a more detailed discussion .

\begin{egs}[Fields]\label{RIG_FIELDS}An $\F_1$-algebra in which every non-zero element is invertible is called an \emph{$\F_1$-field}. Taking the multiplicative group $K\mapsto K^\times$ of an $\F_1$-field establishes an equivalence between the category of $\F_1$-fields and the category of Abelian groups. In particular, in contrast to the situation over $\Z$, there are non-injective homomorphisms between $\F_1$-fields. Any discrete $\F_1$-field admits a unique homomorphism to $\F_1$, the only `true' $\F_1$-field.

Any $\F_1$-algebra $A$ has a unique maximal subfield, the \emph{unit} or \emph{coefficient field} $\F_1[A^\times]$, which is just the group of units together with zero. The coefficient field is functorial.
\end{egs}

Taking the monoid algebra $-\tens_{\F_1}\Z$ defines an adjunction
\[ \Alg_{\F_1}\leftrightarrows\Alg_\Z \]
that allows us to base change $\F_1$-algebras to rings. Both free and forgetful functors commute with localisation.

\paragraph{Finiteness}
Being a category of commutative monoids in a symmetric tensor category, it is straightforward to make sense of the usual finiteness conditions - most importantly, finite type and finite presentation - for homomorphisms and modules. The Noetherian property also makes sense, and can be defined either by the ascending net condition on ideals or by requiring ideals to be finitely generated. We gather a few basic results here for reference.

\begin{prop}[Hilbert basis]\label{FIN_HILB}Let $A$ be a Noetherian $\F_1$-algebra. Then $A[x]$ is Noetherian.\end{prop}
\begin{proof}Since $A[x]=\bigvee_{n\in\N}Ax^n$ is a wedge sum of sets, any ideal $I$ is automatically homogeneous. Writing $I=\bigvee_{n\in\N}I_nx^n$, the ascending chain condition implies that $I_n\trianglelefteq A$ stablises for large $n$. Thus $I$ is finitely generated by $I_0,I_1x,\ldots,I_nx^n$ for large $n$.\end{proof}

Note that in the $\F_1$ context, being Noetherian by no means implies that the set of \emph{quotients} satisfies the ascending chain condition - only quotients by \emph{ideals}.

\begin{cor}\label{FIN_NOETHER}An $\F_1$-algebra is Noetherian if and only if it is of finite type over its coefficient field.\end{cor}
\begin{proof}After the Hilbert basis theorem, we only need to show that a Noetherian $\F_1$-algebra is finitely generated over its unit field - this follows because in particular, the maximal ideal $A\setminus A^\times$ is finitely generated.\end{proof}

In the sequel, we will be led to consider rings that are integral extensions of Noetherian rings, or \emph{integral/Noetherian}.

\begin{lemma}\label{FIN_CRUX}Suppose that $A$ is integral/Noetherian. The set of radical ideals of $A$ satisfies the ascending net condition.\end{lemma}

\subsection{Review: $\F_1$-schemes}

There is a simple way to associate a topological space to an $\F_1$-algebra $A$ (cf. for example \cite{Connes}, \cite{Deitmar}) - just as in algebraic geometry, you take the set $\Spec A$ of prime ideals, topologised with basic open sets given by localisations.

The prime spectrum of $A$ has a unique closed point corresponding to the unique maximal ideal $A\setminus A^\times$ (\cite[\S1.2]{Deitmar}). In other words, affine $\F_1$-schemes are always \emph{local}. This implies that the various other ways one might try to define a covering condition on the category $\mathbf{Sch}^\mathrm{aff}_{\F_1}=\Alg_{\F_1}^\mathrm{op}$ of affine $\F_1$-schemes are all equivalent - indeed, trivial.

More precisely, for a family $A\rightarrow A[f^{-1}_i]$ of localisations, the following are equivalent:
\begin{enumerate}\item $\Spec A=\bigcup_i\Spec A[f^{-1}_i]$;
\item $f_i$ is invertible for some $i$;
\item $\Spec A\cong\Spec A[f^{-1}_i]$ for some $i$;
\item $A\rightarrow \prod_iA[f_i^{-1}]$ is a universally effective monomorphism in $\Alg_{\F_1}$;
\item $A\rightarrow \prod_iA[f_i^{-1}]$ is effective for descent of modules.\end{enumerate}
We therefore define the \emph{Zariski topos} $\Sh\mathbf{Sch}_{\F_1}$ of $\F_1$-schemes to be the presheaf category on $\Alg_{\F_1}^\mathrm{op}$. It is a spatial geometric context in the sense of definition \ref{TOPOS_DEF}, and the lattice $\sh U_{/X}$ of open subobjects of an affine object dual to an algebra $A$ is exactly the lattice of open subsets of its prime spectrum $\Spec A$.

The category $\mathbf{Sch}_{\F_1}$ of $\F_1$-schemes is defined in \cite{Connes} as a category of spaces equipped with a sheaf of monoids, locally modelled by affine $\F_1$-schemes. The functor of points
\[ \mathbf{Sch}_{\F_1}\rightarrow \Sh\mathbf{Sch}_{\F_1}  \]
embeds it as the full subcategory of locally representable sheaves in the Zariski topos. This perspective is treated explicitly in \cite{Toen}.

One easily makes sense of the usual finiteness conditions (quasi-compact, finite type, locally of finite type, locally Noetherian, etc.). The spectrum of an $\F_1$ algebra $A$ is: \begin{itemize}\item a point if and only if $A$ is an $\F_1$-field (e.g. \ref{RIG_FIELDS});
\item a Noetherian topological space if $A$ is integral/Noetherian (lemma \ref{FIN_CRUX}).\end{itemize}

\begin{remark}[Partial addition]\label{RIG_ADDITION}It seems likely that to get a true arithmetic over fields of characteristic one, plain monoids are really good enough. Various authors \cite{Deitmarses,Durov,Lorscheid} have introduced categories of objects with a kind of \emph{partial addition} in an attempt to address this.

The deficiency of the plain monoid theory is already visible at the level of geometry, and we will see this crop up a few times throughout this series. A key example is the following: unlike the case of ordinary schemes, our Spec functor does not take finite products to disjoint unions. Indeed, the product $A_1\times A_2$ of two $\F_1$-algebras $A_i$ has maximal ideal generated by the idempotents $(1,0)$ and $(0,1)$, which corresponds to a point not in either $\Spec A_i$. 

The problem here is that we do not have the relation
\[ (1,0)+(0,1)=1 \]
which in ordinary commutative algebra forces this ideal to equal the whole ring.
This problem can be rectified by allowing the addition of $(1,0)$ to $(0,1)$ (but no other additions), which happens automatically if you take the product in the category of monads \cite{Durov} (and probably, blueprints \cite{Lorscheid} or sesquiads \cite{Deitmarses}, but I didn't check).

I don't wish to pursue such a generalisation in this paper, though I expect that much of the theory is sufficiently abstract that it runs in a more general setting.\end{remark}

\paragraph{Base change to $\Z$}Base change to $\Z$ commutes with localisation, and there is no descent condition to check, so there is a spatial geometric morphism
\[ \Sh\mathbf{Sch}_\Z\rightarrow\Sh\mathbf{Sch}_{\F_1} \qquad \mathbf{Sch}_{\F_1}\rightarrow\mathbf{Sch}_\Z. \]
The forgetful functor does not preserve coverings, so there cannot be a forgetful functor from $\Z$-schemes to $\F_1$-schemes that preserves the affine objects.

It is possible, following \cite{Connes}, to `glue' the categories of $\F_1$ and $\Z$-schemes together by this morphism, but we won't dwell on that perspective here.

\subsection{Pro-discrete}\label{RIG_PROD}

Let $A$ be an $\F_1$-algebra. In the theory of formal schemes, we will want to consider $A$-modules $M$ equipped with an \emph{$A$-linear topology}. Such a topology is defined by a filtration of $M$ by $A$-submodules, which are declared \emph{open}. Indeed, after enlarging such a filtration so that
\begin{itemize}\item the intersection of two open discs is open;
\item every disc containing an open disc is open,\end{itemize}
these submodules (together with $\emptyset$) are the open sets of a topology on $M$ in the usual sense.

We will \emph{always} assume the following condition:
\begin{itemize}\item the filtration is separated. \hfill \emph{Hausdorff}\end{itemize}
If, at any point, we find ourselves with a non-separated filtration, we must take the quotient by the intersection of all open submodules, on which the induced topology is Hausdorff.

The feature of linearly topologised modules over $\F_1$-algebras that simplifies the theory considerably, as compared with its counterpart over $\Z$, is the following statement:

\begin{lemma}\label{RIG_PRO}Let $M$ be a Hausdorff, linearly topologised $A$-module. Then $M\widetilde\rightarrow\lim M/U$, where $U$ ranges over all open submodules.\end{lemma}

In other words, Hausdorff topological modules are automatically \emph{pro-discrete}, or \emph{complete}, to adapt the terminology of Bourbaki. Ignoring size issues, the argument even allows us to identify the category of Hausdorff linearly topologised $A$-modules with the category of pro-objects in the category of discrete $A$-modules whose transition maps are quotients by ideals.\footnote{Note that this is not possible even for \emph{complete} modules over $\Z$.}

\begin{proof}Let $x\in\lim M/U$, and denote by $x_U$ its image in $M/U$. If all $x_U$ are $0$, then write $\phi(x)=0$. Otherwise, there exists a $U$ such that $x_U\neq 0$; then the fibre over $x_U$ of
\[ M \twoheadrightarrow M/U \]
has a unique element, which we call $\phi(x)$. The unicity implies that it must be a lift of $x$.

This defines an inverse $\phi:\lim M/U\rightarrow M$ to the canonical map.
\end{proof}

The tensor (or smash) product $M_1\tens_A M_2$ of linearly topologised modules $M_1,M_2$ is topologised strongly with respect to the maps
\[ e_{x_2}:M_1\rightarrow M_1\tens_AM_2,\quad x \mapsto x\tens x_2 \qquad e_{x_1}:M_2\rightarrow M_1\tens_AM_2,\quad x \mapsto x_1\tens x\]
for $x_i\in M_i$. In other words, its open submodules are those of the form $M_1\tens U_2\cup U_1\tens M_2$, with $U_i\subseteq M_i$ open. The resulting filtration may fail to be Hausdorff, and so the true (`completed') tensor product may be a quotient of the discrete tensor product. It is the limit of the tensor products of the discrete quotients of $M_1$ and $M_2$. Since tensor products for us will \emph{always} be completed, we will not introduce a new notation for this construction.

\

A \emph{linearly topologised} or \emph{pro-discrete $\F_1$-algebra} is an $\F_1$-algebra $A$ equipped with a (Hausdorff) linear topology as a module over itself. The multiplication $A\tens_{\F_1}A\rightarrow A$ is automatically continuous and open. A pro-discrete $\F_1$-algebra is, up to size issues, the same as a Mittag-Leffler pro-$\F_1$-algebra.

We will also want to make the assumption that
\begin{itemize}
\item the product of two open ideals is open. \hfill \emph{adicity}\end{itemize}
The essential consequence of this condition is that the Rees algebras of $A$ are Banach $A$-modules, and hence that blow-ups are representable by schemes; see \S\ref{MORP_REES}.

In reality, we usually also assume the existence (at least locally) of a finite type \emph{ideal of definition}
\begin{itemize}
\item there exists an open, finitely generated ideal whose powers generate the topology. \hfill \emph{admissibility}\end{itemize}
which is linguistically reasonable, given that the word `formal' comes from the formal power series that appear in admissibly topologised rings. This property is not stable under limits. 

In fact, apart from the classification by cone complexes, the only place the existence of an ideal of definition is used in this paper is in the proof that separation and propriety of formal schemes can be checked on a reduction (lemma \ref{SEP_NILPOTENTS}).

\begin{defn}\label{RIG_DEF_BAN}A \emph{pro-discrete $A$-module} is a linearly topologised module over the discrete monoid underlying $A$ such that the action $M\tens_{\F_1} A\rightarrow A$ is continuous. The category of pro-discrete $A$-modules with continuous $A$-equivariant homomorphisms is denoted $\mathrm{lct}_A$ (the letters standing for `locally convex topological'). 

A module $M$ is said to be \emph{adic}, or \emph{Banach}, if its topology is generated by the submodules $MU$ with $U\subseteq A$ an open ideal. Banach modules are stable for tensor product and pullback.\end{defn}

\begin{egs}The basic examples are those coming from totally ordered Abelian groups $H$.; one topologises the $\F_1$-algebra $\F_1[\![t^{-H}]\!]$ associated to $H^\circ=H_{\leq 0}$ by declaring open the ideals associated to \emph{lower sets} in $H^\circ$. The sign convention is such that $t$ is topologically nilpotent. Note that, in contrast to the situation over $\Z$, this monoid has the same underlying set as its polynomial counterpart $\F_1[t^{H^\circ}]$.

Such `formal power series' monoids and arbitary products of discrete monoids satisfy the adicity condition, but typical infinite-dimensional examples like $\F_1[x_1,\overrightarrow\cdots]$, an infinite limit of polynomial rings on finite index sets do not.\end{egs}

\paragraph{Base change to $\Z$}In ordinary commutative algebra, the theory of complete linearly topologised rings is \emph{not} equivalent to the theory of pro-discrete rings. When discussing formal schemes over $\Z$, we will use the latter as our definition; that is, a pro-discrete ring is a pro-object of $\Alg_\Z$ with surjective transition maps.

The base change functor on discrete algebras has a natural pro-extension
\[ \Pro^\mathrm{ML}\Alg_{\F_1}\rightarrow\Pro^\mathrm{ML}\Alg_\Z \]
that commutes with tensor products and manifestly preserves `admissibility'. However, since the forgetful functor $\Alg_\Z\rightarrow\Alg_{\F_1}$ does not preserve quotients by ideals, this functor does not have a right adjoint, $\mathrm{i.e.}$ there is no forgetful functor from pro-discrete rings to pro-discrete $\F_1$-algebras.

\subsection{Some definitions of formal $\F_1$-schemes}\label{RIG_FSCH}

One extends the definition of spectrum to linearly topologised $\F_1$-algebras $A$ by analogy with the traditional setting \cite[\textbf{I}.10]{EGA}, that is, by taking the set of \emph{open} prime ideals. The principal open sets correspond to \emph{completed} localisations of $A$, which are defined exactly as over $\Z$:
\[ \lim_JA/J =A \rightarrow A\{f^{-1}\}:= \lim_JA/J[f^{-1}]. \]
The formal prime spectrum is a `pro-discretely monoidal space', which I will break with tradition by denoting $\Spec A$. This will not cause ambiguity, as we will never take the spectrum of \emph{all} prime ideals of a non-discrete ring. 

Since the maximal ideal is always open, the formal spectrum of a linearly topologised ring continues to be local. In other words, there are no non-trivial coverings of affine formal schemes over $\F_1$.

The equivalent definitions of the (trivial) covering condition on the opposite $\mathbf{FSch}^\mathrm{aff}_{\F_1}$ to the category of admissible $\F_1$-algebras still hold good in the pro-discrete regime - indeed, we can even add 
\begin{enumerate}
\item[vi)] $A\rightarrow \prod_iA[f_i^{-1}]$ is effective for descent of pro-discrete modules.
\item[vii)] $A\rightarrow \prod_iA[f_i^{-1}]$ is effective for descent of Banach modules.
\end{enumerate}
Thus any of these - now seven - conditions define the same spatial theory $\Sh\mathbf{FSch}_{\F_1}$, which is the presheaf category on $\mathbf{FSch}^\mathrm{aff}_{\F_1}$ with open immersions dual to completed localisations. It carries a tautological sheaf $\sh O$ of adic $\F_1$-algebras. 

We may therefore give two equivalent definitions of formal schemes:

\begin{defns}[Formal schemes]A formal $\F_1$-scheme is 
\begin{enumerate}\item a space with a sheaf $\sh O$ of adic $\F_1$-algebras, locally isomorphic to the spectrum of an adic (admissible) $\F_1$-algebra. 
\item a locally representable presheaf on $\mathbf{FSch}^\mathrm{aff}_{\F_1}$.
\end{enumerate}
The category of formal $\F_1$-schemes is denoted $\mathbf{FSch}_{\F_1}$.\end{defns}

\begin{remark}There are actually some subtleties involved in making sense of a `sheaf of adic $\F_1$-algebras'. The theory of pro-discrete $\F_1$-algebras is not a finite products theory, and so the category of pro-discrete $\F_1$-algebras internal to the sheaf topos of a space $X$ is \emph{not} the same as the category of colimit-preserving functors from $\Sh(X)$ into the category of pro-discrete $\F_1$-algebras. This fact will be familiar to students of the $\ell$-adic cohomology.

Spaces defined in terms of affine pieces like formal schemes do not suffer from this ambiguity, so either approach would be sufficient for the purposes of this definition. However, note that even for formal schemes, since admissibility is not stable for limits, general section spaces of $\sh O$ need not admit ideals of definition.\end{remark}

The inclusion of the subcategory of discrete algebras gives a spatial geometric morphism \[\Sh\mathbf{FSch}_{\F_1}\rightarrow\Sh\mathbf{Sch}_{\F_1} \qquad \mathbf{Sch}_{\F_1}\rightarrow \mathbf{FSch}_{\F_1}. \]
Via the fully faithful composite of the Yoneda embedding and the pushforward functor \[ \mathbf{FSch}_{\F_1}\hookrightarrow\Sh\mathbf{FSch}_{\F_1}\rightarrow\Sh\mathbf{Sch}_{\F_1},\] it is also possible to consider formal schemes as certain locally \emph{ind-representable} objects of $\Sh\mathbf{Sch}_{\F_1}$. Indeed, by definition the opposite category to that of admissibly topologised $\F_1$-algebras is a full subcategory of $\mathrm{Ind}\mathbf{Sch}_{\F_1}^\mathrm{aff}$ whose objects, among other conditions, have transition maps finitely presented nilpotent embeddings. For simplicity we also take this perspective on formal schemes over $\Z$.

The definition of blow-ups, and hence of rigid analytic spaces, is not local for $\Sh\mathbf{Sch}$, so we will still need to use the larger topos $\Sh\mathbf{FSch}$ in the sequel.

The categories $\mathrm{lct}_-$ extend to a stack $\QC=\QC_{\mathbf{FSch}}$ on the formal topos; the categories $\mathrm{Ban}_-$ define a full substack (def. \ref{RIG_DEF_BAN}).
If $f$ is a qcqs morphism of formal schemes, the corresponding pullback morphism $f^*$ on $\QC$ has a right adjoint $f_*$. A qcqs morphism is representable by schemes in $\Sh\mathbf{Sch}_{\F_1}$ - a property the literature usually calls \emph{adicity} - if and only if this pushforward functor preserves the substack of Banach modules.

\paragraph{Base change to $\Z$}The category of affine formal schemes over $\Z$ is, by our definition, opposite to the full subcategory of Mittag-Leffler pro-rings
\[\mathbf{FSch}_\Z^\mathrm{aff}\hookrightarrow\Pro\Alg_\Z^\mathrm{op}\]
defined by the criteria of adicity and admissibility. Note that this is slightly different from the traditional theory expounded in terms of topological algebra in \cite[\textbf{1}.1.10]{EGA}, though under the axiom of dependent choice the two approaches can be shown to be equivalent at least for first countably topologised algebras.

The base change functor $\mathbf{FSch}^\mathrm{aff}_{\F_1}\rightarrow\mathbf{FSch}_\Z^\mathrm{aff}$ is the restriction of the ind-extension of the base change for schemes, and it gives rise, as in the former case, to a spatial geometric morphism
\[ \Sh\mathbf{FSch}_\Z\rightarrow\Sh\mathbf{FSch}_{\F_1}\qquad \mathbf{FSch}_{\F_1}\rightarrow \mathbf{FSch}_\Z. \]

\subsection{Formal completion}\label{FSCH_COMPLETION}
The initial remarks of this section are valid for $\F_1$ and for $\Z$. Note that in the latter case our definitions deviate a little (see immediately above) from the traditional ones, under which some of the assertions of this section are false.

Let ${}^Z\mathbf{FSch}$ denote the category whose objects are pairs $(X,Z)$ consisting of a formal scheme over $\F_1$ or $\Z$ together with a closed, finitely presented algebraic subscheme $Z$, and whose morphisms $f:X_1\rightarrow X_2$ satisfy $Z_1\subseteq (f^{-1}Z_2)^\mathrm{red}$. The isomorphism class of an object $(X,Z)$ of ${}^Z\mathbf{FSch}$ depends only on $X$ and the underlying set of $Z$.

The forgetful functor
\[ {}^Z\mathbf{FSch}\rightarrow\mathbf{FSch} \]
has adjoints on both sides. The right adjoint, which marks an entire formal scheme, has a further right adjoint: \emph{formal completion.}


The \emph{formal completion of $X$ along $Z$} is, as an ind-scheme, the inductive limit
\[ \hat X_Z := \colim_{Z^\prime\subseteq Z} Z^\prime \]
of closed subschemes with set-theoretic support in $Z$. Formal completion is idempotent. If $X$ is a scheme, then the formal completion along $Z$ may be computed as the double complement of $Z$ in the Heyting algebra of subobjects of $X$ in the Zariski topos.

To see that $\hat X_Z$ is a formal scheme according to our definitions (\S\ref{RIG_PROD},\ref{RIG_FSCH}), we will need to show that its co-ordinate algebra is adic. We declare open any quasi-coherent ideal sheaf $I\trianglelefteq\sh O_X$ cosupported on $Z$, and write $\hat{\sh{O}}_{X,Z}$ for the Hausdorff quotient of the linear topology on $\sh O_X$ thus obtained. It is a pro-discrete quasi-coherent sheaf on $X$, the limit of all quotients of $\sh O_X$ supported on $Z$. 

Since cosupports remain unchanged under taking powers of an ideal, $\hat{\sh{O}}_{X,Z}$ is adic, and hence its spectrum $\hat X_Z$ is a (marked) formal scheme. Its ideal of definition given by the ideal defining $Z$.

\begin{prop}[Formal completion commutes with base change] Let $(Y,Z)$ be a marked formal scheme, $f:X\rightarrow Y$ any morphism (including along $\Spec\Z\rightarrow\Spec\F_1$). The square
\[\xymatrix{ \hat X_{f^*Z} \ar[r]\ar[d] & X \ar[d] \\ \hat Y_Z\ar[r] & Y }\]
is Cartesian.\end{prop}
\begin{proof}Follows from the existence of an ideal of definition.\end{proof}

\begin{eg}In the extreme example where $Z=X$, the formal completion is $X$ itself. The `next most faithful' formal completion of $X$ comes from setting $Z$ the \emph{boundary} of $X$, that is, the union of all divisors. If $X$ is a locally Noetherian (resp. Noetherian) $\F_1$-scheme, the result is a locally Noetherian (resp. Noetherian) formal $\F_1$-scheme. This is in marked contrast to the same operation in algebraic geometry, which generally yields a huge mess.\end{eg}

\paragraph{Functorial algebraisation}
If $A$ is a topological $\F_1$-algebra, let us write $A^?$ for the `forgetful' underlying discrete $\F_1$-algebra. It defines a functor
\[ ?:\mathbf{FSch}^\mathrm{aff}\rightarrow {}^Z\mathbf{Sch}^\mathrm{aff} \]
which marks all open ideals. It is left adjoint, and right inverse, to formal completion.

Of course, for ordinary schemes one hits a wall as soon as one tries to globalise this functor. But, stupidly, since over monoids there are no non-trivial coverings of formal schemes, it actually extends to a colimit-preserving (but not left exact) functor
\[?:\Sh\mathbf{FSch}_{\F_1}\rightarrow \Sh^Z\mathbf{Sch}_{\F_1}:= \PSh^Z\mathbf{Sch}_{\F_1}^\mathrm{aff}\]
left adjoint and right inverse to formal completion.

\begin{defn}A formal scheme $X$ is said to be \emph{algebraisable} if there exists a marked scheme $(Y,Z)$ and an isomorphism of (unmarked) formal schemes $\hat Y_Z\cong X$; these data are called an \emph{algebraisation} of $X$.\end{defn}

The forgetful sheaf ${}^?X$ is a natural candidate for an algebraisation of $X$: if ${}^?X$ is a scheme, then it is a \emph{functorial} algebraisation of $X$. We are led naturally to the following strengthening of the algebraisability question:

\begin{question}When does $?$ take formal schemes to schemes?\end{question}
We give a complete answer to this question for integral formal schemes in \S\ref{PFAN_MERO}.

\subsection{Markings}\label{FSCH_MARKING}
In the sequel, it will be useful to have a category parametrising \emph{marked} formal schemes, that is, formal scheme marked along a family of closed formal subschemes. This is a variation on the definition of the category ${}^Z\mathbf{FSch}$ (\S\ref{FSCH_COMPLETION}).

\begin{defn}A \emph{marked formal scheme} is a pair $(X,Z)$ consisting of a formal scheme $X$ and a family of finitely presented closed formal subschemes $Z$. A morphism of pairs is a morphism $f:X_1\rightarrow X_2$ of formal schemes such that $(f^{-1}Z_2)^\mathrm{red}\subseteq Z_1$.

The category of marked schemes, resp. formal schemes is denoted ${}_Z\mathbf{Sch}\hookrightarrow{}_Z\mathbf{FSch}$. \end{defn}

The isomorphism class of $(X,Z)$ depends only on the underlying reduced formal schemes $Z^\mathrm{red}$ of $Z$. A marked formal scheme $(X,Z)$ has a canonical `maximal' representative in its isomorphism class, in which $Z$ is replaced with the family of all finitely presented closed subschemes with set-theoretic support in a finite union of members of $Z$.

If, in fact, $f^{-1}Z_2=Z_1$, we say that $f$ is \emph{represented by the formal scheme} $X_1$ over $(X_2,Z_2)$; the slice category of represented morphisms is simply $\mathbf{FSch}_{X_2}$.

Let us call a family $(X_\bullet,Z_\bullet)\rightarrow (X,Z)$ a \emph{covering} if it is represented by formal schemes and $X_\bullet\rightarrow X$ is a Zariski covering. If the formal schemes involved are affine, it is equivalent to ask simply that one of the maps is an isomorphism of pairs. This system of coverings defines a topos $\Sh_Z\mathbf{FSch}$, the \emph{marked formal topos}. It is a spatial theory (def. \ref{TOPOS_DEF}) with represented Zariski-open immersions, and the locally representable objects are precisely the marked formal schemes.

\paragraph{Adjoints}
The forgetful functor \[{}_Z\mathbf{FSch}\rightarrow \mathbf{FSch}\] has a fully faithful right adjoint, which associates to the formal scheme $X$ the pair $(X,\emptyset)$, and left adjoint, which associates $(X,X)$. The right adjoint has a further right adjoint $(X,Z)\mapsto X\setminus Z$, which commutes with coverings. These four adjoints, from left to right:
\begin{align}\nonumber (X,X) &\mapsfrom  X \\ 
\nonumber (X,Z) &\mapsto  X \\
\nonumber (X,\emptyset) & \mapsfrom  X \\
\nonumber (X,Z) & \mapsto  X\setminus Z\end{align}
induce three essential geometric morphisms
\[\xymatrix{ \Sh\mathbf{FSch}\ar@<3pt>[r]\ar@<-3pt>[r] & \Sh_Z\mathbf{FSch};\ar[l] }\]
the one going from right to left is surjective, and the other two are sections.

More generally, if $(Y,Z)\in{}_Z\mathbf{FSch}$ is any object, the right adjoint to \[{}_Z\mathbf{FSch}_{(Y,Z)}\rightarrow \mathbf{FSch}_{Y}\] takes $X$ to $(X,Z\times_{Y}X)$. The other adjoints have the same formulas.



\section{Elementary properties of morphisms}\label{MORP}

Here I review some easy properties of morphisms of formal schemes, in generalisation of the concepts presented in \cite[\textbf{II}]{EGA} for schemes over $\Z$. In particular, we make a special study of projective morphisms \S\ref{MORP_PROJ} and blow-ups \S\ref{MORP_REES}, a good understanding of which is critical to the definitions of both overconvergence and rigid analytic geometry \cite{part2}.

Unless otherwise noted, all definitions are valid over $\F_1$ and $\Z$ (and are mostly standard in the latter case); I therefore suppress the subscripts denoting the base. Unless otherwise noted, `stable for base change' includes the base change $\F_1\rightarrow\Z$.





\subsection{Embeddings and immersions}\label{MORP_EMBED}

The passage from $\Z$ to $\F_1$ presents only two complications, both already visible in the class of \emph{embeddings}.

The first issue is that surjective morphisms of monoids need not be determined by their kernel. Dually, we find that embeddings - inclusions of subschemes cut out by equations - need not be \emph{closed}. This is a manifestation of the same issue that necessitates an alternative approach to separation and propriety in \S\ref{SEP}.

The second issue is of a more pathological nature (i.e. it should probably fixed by any `true' theory of $\F_1$, should such a thing exist): affine monoid schemes are necessarily connected, so the inclusion of a disconnected closed subscheme of a connected scheme cannot be affine. This makes it somewhat difficult to guarantee the existence of embedded closures for arbitrary immersions.

For the purposes of this paper, luckily, we are able to completely sidestep having to address the second issue; I only formulate a couple of unanswered questions \ref{Q_AFF_EMBED}, \ref{Q_PROJ_EMBED}.

\begin{defns}[Embeddings]\label{MORP_EMBEDDING}An affine morphism $X\rightarrow Y$ is said to be a \emph{formal embedding} if for each closed algebraic subscheme $X_0\hookrightarrow X$, $\sh O_Y\rightarrow \sh O_Y(X_0)$ is surjective. Over $\F_1$, this condition is the same as asking that $\sh O_Y\rightarrow\sh O_Y(X)$ be surjective. It is moreover an \emph{embedding} if it is representable by schemes. Affine formal embeddings are monic.


A quasi-compact monomorphism $X\rightarrow Y$ is a (formal) embedding if it can be covered by affine (formal) embeddings; that is, if locally on $Y$ there are coverings \[U_\bullet\rightarrow X_\bullet\twoheadrightarrow X  \] in $\Sh\mathbf{FSch}$ with $U_\bullet\rightarrow X_\bullet$ an open immersion, $U_\bullet\twoheadrightarrow X$ an open cover and each $X_i\rightarrow Y$ an affine (formal) embedding.

The partially ordered set of embedded (resp. formally embedded) subschemes of $Y$ is denoted $\lie Z(Y)$ (resp. $\hat{\lie Z}(Y)$). 
Since any formal embedding is ind-representable by embeddings, $\hat{\lie Z}$ is generated under filtered suprema by ${\lie Z}$.\end{defns}

Over $\Z$, every embedding is a closed immersion associated to some ideal, and in particular, affine. One can therefore define unions of embeddings by intersecting ideals, and $\lie Z$ has finite joins. The same follows for $\hat{\lie Z}$, since it is generated by $\lie Z$.

Over $\F_1$, embeddings can fail to be closed, and finite joins in $\lie Z$, even when they exist, can fail to be set-theoretic unions. They can also fail to be affine; see example \ref{MORP_A2}.

\begin{defn}[Embedded image]A \emph{formally embedded image} for a morphism $X\rightarrow Y$ is an embedded subscheme $\mathrm{cl}(X/Y)$ of $Y$ initial among those through which $X$ factors.

An open immersion $X\hookrightarrow Y$ is said to be \emph{dense}, or more precisely \emph{scheme-theoretically dense}, if $Y$ is a formally embedded closure of $X$ in $Y$. A dense open immersion is always dense in the sense of point-set topology, but not vice versa.\end{defn}

One can construct a formally embedded image of an affine morphism $f:X\rightarrow Y$ via the formula
\begin{align}\label{prous}\mathrm{cl}(X/Y):=\Spec\left(\lim_{X_0\subseteq X}\mathrm{Im}[\sh O_Y\rightarrow \sh O_Y(X_0)]\right)\end{align}
where the limit is over closed algebraic subschemes of $X$. Over $\F_1$, this limit is the same as the image of $\sh O_Y$ in $\sh O_Y(X)$ with the \emph{subspace} topology.

\begin{lemma}\label{AGNU}Let $X\rightarrow Y$ be an affine morphism. Then $\mathrm{cl}(X/Y)$ as defined by formula (\ref{prous}) is a formally embedded image of $X/Y$.\end{lemma}
\begin{proof}Both the definition of embedding and the stated formula (\ref{prous}) are local on the target, so we may assume $Y$ is affine.

Let $Z\hookrightarrow Y$ be an embedded subscheme through which $X$ factors. If we are over $\Z$, $Z$ is affine. Otherwise, let $\bigcup_iZ_i$ be an affine covering of $Z$. Since affine schemes over $\F_1$ have no non-trivial coverings, we must have $Z_i\times_YX=X$ for a single index $i$. 

Either way, the embedded image of $X$ in $Y$ is affine, and hence computed by (\ref{prous}).\end{proof}

If $X$ and $Y$ are both schemes, the formally embedded image is in fact an embedded subscheme. However, for formal schemes $\mathrm{cl}(X/Y)\rightarrow Y$ is rarely representable by schemes, even when $X\hookrightarrow Y$ is an open immersion.

We would like be able to construct formally embedded images for more general quasi-compact morphisms $X\rightarrow Y$, that is, \emph{left adjoints}
\[ \mathrm{cl}(-/Y):\mathbf{FSch}_Y\rightarrow\hat{\lie Z}(Y)  \] to the inclusion. By lemma \ref{AGNU}, this adjoint exists on the subcategory $\mathbf{FSch}_Y^\mathrm{aff}$ and for any formal scheme over $\Z$.

For more general formal schemes over $Y$, one can at least define a pro-adjoint 
\[ \mathbf{FSch}_Y\rightarrow\Pro\hat{\lie Z}(Y);\]
that is, the formally embedded closure always makes sense as a pro-object of $\mathbf{FSch}_Y$. The pro-adjoint would be the extension of an ordinary adjoint if one could take arbitrary limits of embeddings. However, since embeddings can fail to be affine it is not clear that this is possible.

The matter would be settled by a positive answer to the following question, inspired by the example \ref{MORP_A2}:

\begin{question}\label{Q_AFF_EMBED}Is every embedding affine in the category of schemes relative to monads (or blueprints, or sesquiads...)?\end{question}

\begin{eg}[$\hat{\lie Z}(\A^1)$]\label{MORP_A2}The formally embedded subobject poset of the affine line $\A^1_{\F_1}$ over $\F_1$ has the form, increasing from left to right:
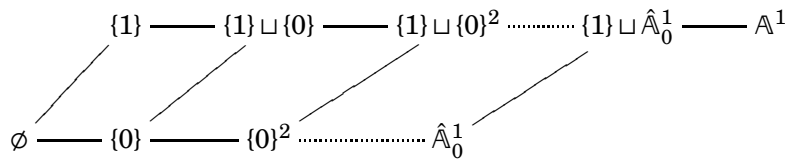
\[\xymatrix{&\{1\}\ar@{-}[r]& \{1\}\sqcup\{0\}\ar@{-}[r] & \{1\}\sqcup\{0\}^2\ar@{.}[r] &\{1\}\sqcup \hat\A^1_0\ar@{-}[r] & \A^1 \\
 \emptyset\ar@{-}[ur]\ar@{-}[r] & \{0\}\ar@{-}[ur]  \ar@{-}[r]  & \{0\}^2 \ar@{-}[ur]\ar@{.}[r] & \hat\A^1_0\ar@{-}[ur] }\]
The disjoint union of the origin and the non-closed embedded point  $1$ in is embedded, but being disconnected, is not affine. Note that its inclusion is a bijection, but not a topological immersion. It is an embedded open subset of its affinisation $\Spec(\F_1\times\F_1)$, which itself is no longer immersed. 

This counterexample vanishes after enlarging the category of monoids to include certain monads with non-trivial addition as in e.g. \ref{RIG_ADDITION}. After base change to $\Z$, the embedding of course becomes closed. 
\end{eg}

\begin{defn}[Immersions]\label{MORP_IMMERSION}A morphism is a \emph{(formal) immersion} if it can be written as an open immersion followed by a (formal) embedding.\end{defn}

An immersion of $\F_1$-schemes need not be a topological immersion; see example \ref{MORP_A2}.

In the case that $X\hookrightarrow Y$ is an immersion, we also call $\mathrm{cl}(X/Y)$ the \emph{formally embedded closure} of $X$ in $Y$. A (formal) immersion $X\hookrightarrow Y$ is a (formal) embedding if and only if for all affine (formal) immersions $U\hookrightarrow Y$ factoring through $X$, the formally embedded closure $\mathrm{cl}(U/Y)$ is also contained in $X$. 

\begin{prop}\label{MORP_OPEN}An immersion is open in the embedded image of its affinisation. In particular, an affine immersion is open in its embedded closure.\end{prop}
\begin{proof}Same proof as \cite[\href{http://stacks.math.columbia.edu/tag/01P9}{01P9}]{stacksproject}.\end{proof}

\begin{prop}[Stability]The class of (formal) immersions is stable for base change and descent.
The class of (formal) embeddings is stable for composition, base change, and descent.\end{prop}


\subsection{Integral morphisms and relative normalisation}\label{MORP_NORM}

Let $A$ be an $\F_1$-algebra, and let $B$ be a \emph{finite} $A$-algebra, that is, an algebra that is finitely generated as an $A$-module. It follows that for every $f\in B$, either $f^n\in A$ for some $n>0$, or the set $\{f,f^2,\ldots\}$ is finite. In other words, $f$ satisfies a \emph{monic equation}
\[ f^n=c_if^i,\quad i<n,c_i\in A \]over $A$. One can therefore define \emph{integral closure} of pairs of $\F_1$-algebras in much the same way as for commutative rings.

\begin{defn}An affine morphism $X\rightarrow Y$ in $\Sh\mathbf{FSch}$ is said to be \emph{finite} if $\sh O_Y(X)$ is a finite Banach $\sh O_Y$-algebra; that is, locally on $Y$ there is a topological quotient $p:\sh O_Y^{\oplus n}\twoheadrightarrow\sh O_Y(X)$ for some $n\in\N$. We say that a morphism is \emph{integral} if it is a limit of finite morphisms in the category of formal schemes adic over $Y$.

An affine morphism is said to be \emph{formally finite} if locally on $Y$ there is a module homomorphism $p:\sh O_Y^{\oplus n}\rightarrow\sh O_Y(X)$ that surjects onto every discrete quotient of $\sh O_Y(X)$. Over $\F_1$, this is equivalent to surjectivity of $p$ itself. 
A morphism is \emph{formally integral} if it is integral over a formally finite morphism.

A \emph{relative normalisation} of an affine morphism $X\stackrel{f}{\rightarrow}Y$ is an integral morphism $\nu_fY\rightarrow Y$ initial among those factoring $f$. A relative normalisation can be constructed by taking the integral closure of the image of $\sh O_Y$ inside $f_*\sh O_X$.
\end{defn}

\begin{remark}Since embeddings can fail to be affine, any reasonable definition would also allow finite morphisms to be non-affine. Following the template of definition \ref{MORP_EMBEDDING}, we might say that a general quasi-compact morphism is finite if it can be covered by affine finite morphisms. 

For the purposes of this paper, we can survive without treating non-affine finite morphisms.\end{remark}

A finite morphism $X\rightarrow Z$ followed by a formal embedding $Z\hookrightarrow Y$ is formally finite: the surjection $\sh O_Y(Z)^{\oplus n}\twoheadrightarrow\sh O_Y(X)$ induces a module homomorphism $\sh O_Y^{\oplus n}\rightarrow\sh O_Y(X)$ satisfying the requirement of the definition. Of course, by definition an integral morphism followed by a formal embedding is formally integral.

Conversely, by taking the formally embedded image in $Y$, any formally finite (resp. integral) affine morphism $X\rightarrow Y$ can be written as a finite (resp. integral) morphism followed by an affine formal embedding. By the same logic, a formal embedding followed by a finite (resp. integral) morphism can be rewritten in this way, and is hence, in particular, formally finite (resp. integral).

\begin{lemma}\label{MORP_FINITE}Any formally finite affine morphism can be written as a finite morphism followed by an affine formal embedding. Any composite of finite morphisms and affine formal embeddings is formally finite.

Any formally integral affine morphism can be written as an integral morphism followed by aa affine formal embedding. Any composite of integral morphisms and affine formal embeddings is formally integral.\end{lemma}

\begin{prop}[Stability]The classes of finite, integral, formally finite, and formally integral morphisms are stable for composition, base change, and descent. If $gf$ is finite, resp. integral, resp. formally finite, resp. formally integral, then so is $f$.\end{prop}

\paragraph{Divisorial markings}
For rigid analytic geometry, we will need a slightly refined version of the relative normalisation, adapted to the full subcategory ${}_Z\mathbf{FSch}^\mathrm{div}\subset {}_Z\mathbf{FSch}$ of \emph{divisorially marked} formal schemes, that is, marked formal schemes for which $Z$ is a collection of locally principal subschemes. The isomorphism class of an object $(Y,Z)$ of ${}_Z\mathbf{FSch}^\mathrm{div}$ is determined by the formal scheme $Y$ together with the multiplicative subset $S_Z\subseteq\sh O_Y$ comprised of sections whose vanishing locus has radical contained in $Z$. 

We may therefore write this object of ${}_Z\mathbf{FSch}^\mathrm{div}$ as $(Y;\sh O_Y[S_Z^{-1}])$, where $\sh O_Y[S_Z^{-1}]$ denotes the localisation in the category of all (not necessarily pro-discrete) topological modules.\footnote{Strictly speaking, over $\Z$ one must be a little careful in making sense of this localisation. It is good enough to consider it as a certain ind-object in the category of Banach $\sh O_Y$-modules.}


\begin{defns}\label{MORP_ADMIT_INT}An integral morphism $X\rightarrow Y$ is an \emph{isomorphism away from $Z$}, or $Z$-\emph{admissible}, if $\sh O_Y[S_Z^{-1}]\rightarrow\sh O_X[S_Z^{-1}]$ is an isomorphism.

We say that the pair $(Y,\sh O_Y[S_Z^{-1}])$ is \emph{relatively normal}, or that $Y$ is \emph{normal along} $Z$, if $Z$ are Cartier divisors and $\sh O_Y$ is integrally closed in $\sh O_Y[S_Z^{-1}]$. Equivalently, $(Y,Z)$ is relatively normal if and only if $Z$ is Cartier and any $Z$-admissible finite morphism into $Y$ is an isomorphism.\end{defns}

Let us denote the full subcategories of ${}_Z\mathbf{FSch}^\mathrm{div}$ whose objects are marked along Cartier divisors, resp. are relatively normal, by ${}_Z\mathbf{FSch}^\mathrm{inv}$, resp. ${}_Z\mathbf{FSch}^\nu$.

Given a pair $(Y,\sh O_Y[S_Z^{-1}])\in{}_Z\mathbf{FSch}^\mathrm{div}$, we can construct a \emph{relative normalisation}, or \emph{normalisation of $Y$ along $Z$}, in two stages:
\begin{enumerate}\item replace $\sh O_Y$ with its image in $\sh O_Y[S_Z^{-1}]$; \hfill\emph{invert $Z$}
 \item pass to the integral closure inside $\sh O_Y[S_Z^{-1}]$. \hfill\emph{separate crossings along $Z$} \end{enumerate}
The first stage yields an embedded subscheme of $Y$, and the second yields an integral morphism $\nu_ZY\rightarrow Y$ whose embedded image is that produced by the first. These two stages constitute right adjoints to the inclusions
\[{}_Z\mathbf{FSch}^\nu\leftrightarrows{}_Z\mathbf{FSch}^\mathrm{inv}\leftrightarrows{}_Z\mathbf{FSch}^\mathrm{div}\]
It satisfies the following universal property: any finite morphism $X\rightarrow Y$ that is an isomorphism away from $Z$ factorises uniquely
\[ \nu_ZY\rightarrow X\rightarrow Y \]
the relative normalisation.


\subsection{Projective morphisms}\label{MORP_PROJ}

One can define, as for usual schemes over $\Z$, the Proj of a positively graded $\F_1$-algebra $A$ - that is, an algebra object in the category of $\N$-indexed families of pointed sets - in terms of homogeneous prime ideals. Even without this description, one can easily define the principal affine subsets of $\Proj A$, in the usual way, as the spectra of the degree zero piece $A[f^{-1}]_0$ of the $\Z$-graded localisation. 

We lift the usual definitions of quasi-projective, projective, ample and very ample invertible sheaf from \cite[\textbf{II}]{EGA}.

In particular, for any quasi-coherent Banach module $V$, we can define as usual a projective bundle
\[ \P(V):=\Proj(\Sym^\bullet V), \]
qcqs when $V$ is of finite type, and this `generates' the definition of projectivity. By choosing a finitely presented covering $V^\mathrm{fp}\twoheadrightarrow V$, one can always write a projective bundle as a finitely presented projective bundle $\P(V^\mathrm{fp})$ composed with an affine embedding $\P(V)\hookrightarrow\P(V^\mathrm{fp})$.

\begin{lemma}Let $V$ be a quasi-coherent Banach module on $Y$, $f:X\rightarrow Y$ a morphism. Then $\P(f^*V)\cong\P(V)\times_YX$.\end{lemma}

\begin{lemma}\label{MORP_PROJ_DIAG}The diagonal of a projective bundle is an affine embedding.\end{lemma}
\begin{proof}By direct calculation, which applies equally over $\Z$ and over $\F_1$.\end{proof}


\begin{defn}\label{MORP_PROJ_DEF}A morphism is said to be (formally) \emph{projective}, resp. \emph{integral/projective} if it is (formally) finite, resp. integral, over a projective bundle. We may always assume that the projective bundle is finitely presented. 

A projective morphism $\Proj A\rightarrow Y$ to a divisorially marked formal scheme $Y\in{}_Z\mathbf{FSch}^\mathrm{div}$ is an \emph{isomorphism away from $Z$}, or $Z$-\emph{admissible}, if $A_k[S_Z^{-1}]$ is an invertible $\sh O_Y[S_Z^{-1}]$-module for $k\gg 0$. The definition extends to integral/projective morphisms via def. \ref{MORP_ADMIT_INT}.\end{defn}

\begin{prop}[Stability]\label{MORP_PROJ_STABILITY}Projective and integral/projective morphisms, as well as their formal variants, are stable for composition and base change.

The diagonal of a formally projective morphism is an affine embedding.\end{prop}
\begin{proof}To show stability under composition for projective morphisms, we must show that a finitely presented projective bundle $\P_X(V)\rightarrow X$ followed by a finite morphism $X\rightarrow Y$ may be written the other way around, and hence is projective. Let $V_Y$ be a model for $V$ over $Y$. Then the square
\[\xymatrix{\P_X(V)\ar[r]\ar[d] & \P_Y(V_Y)\ar[d] \\ X\ar[r] & Y}\]
is Cartesian, whence the result. The proof for the more general classes proceeds in the same manner, \emph{mutatis mutandi}.\end{proof}

By definition, a morphism is projective if and only if it is the Proj of a finitely generated graded $\sh O_Y$-algebra whose piece in degree zero is finite.

If, more generally, $A$ is a graded quasi-coherent $\sh O_Y$-algebra finitely generated over its degree zero piece $A_0$, and $A_0$ is integral (resp. formally finite, resp. formally integral) over $\sh O_Y$, then $\Proj A\rightarrow Y$ is integral/projective (resp. formally projective, resp. formally integral/projective).

To attempt to prove a converse, one might consider the affinisation
\[ X\rightarrow \Spec f_*\sh O_X\rightarrow Y \]
of an integral/projective morphism $f:X\rightarrow Y$. By proposition \ref{MORP_PROJ_STABILITY}, $X\rightarrow \Spec f_*\sh O_X$ is integral/projective. We would like to know that $f_*\sh O_X$ is an integral $\sh O_Y$-algebra. Since $f_*$ commutes with filtered colimits, it would be enough to know that when $f$ is \emph{projective}, $f_*\sh O_X$ is \emph{finite}.

This brings us to the question of finiteness of global sections over projective morphisms:

\begin{question}\label{Q_PROJ_FINITE}Let $f:X\rightarrow Y$ be projective. Is $f_*\sh O_X$ a finite $\sh O_Y$-module?\end{question}

A positive answer to this question in general would imply:

\vspace{9pt}\hspace{2.3cm}\begin{tabular}{ rcl }
 projective \& affine &$\Rightarrow$ &  finite \\
 integral/projective \& affine & $\Rightarrow$ & integral 
\end{tabular}
\\\\
Indeed, the second statement would follow from the first and the fact that any affine morphism is a limit of finite type affine morphisms. Of course, there cannot be any analogue of these implications for formally projective morphisms.

An attack on this question for $\F_1$ would take us far beyond the scope of this paper. For $\Z$, if $f$ is pseudo-coherent (for example, if $Y$ is locally Noetherian), it is a consequence of the much more general projective pushforward theorem \cite[III.2.2]{SGA6}. In the special case of pushing forward the structure sheaf, we can easily make do with a little less technology:

\begin{lemma}\label{A_PROJ_FINITE}Let $f:X\rightarrow Y$ be a projective morphism of formal schemes over $\Z$. Then $f_*\sh O_X$ is an integral $\sh O_Y$-algebra.\footnote{This argument uses that $Y$ has finite ideal type.}\end{lemma}
\begin{proof}The question being local on $Y$, suppose $Y=\Spec\sh O(Y)$ is affine. Then $X$ is a closed subscheme of a projective space $\P^r_Y$. Let $\{A_i\}_{i\in I}$ be the filtered set of pairs consisting of a Noetherian subring $A$ of $\sh O(Y)$ and a closed subscheme $X_i\subseteq\P^r_i$ such that $X\subseteq X_i\times_i\P^r_Y$. Then $X\rightarrow X_i$ is affine, and $X=\lim_iX_i$.


Pick $Y_0=\Spec A_0$. For every $i\rightarrow 0$, we have commuting squares
\[\xymatrix{ X_i\ar[r]^{g_i}\ar[d]_f & X_0\ar[d]^f \\ Y_i\ar[r]^{g_i} & Y_0 }\]
We will calculate global sections on $Y_0$.

Since the transition maps are affine, \[ \colim_ig_{i*}\sh O_{X_i} \tilde\rightarrow g_*\sh O_X\]
in the category of quasi-coherent sheaves on $\P^r_0$. Since $f_*$ commutes with filtered colimits, we may push forward this isomorphism to $Y_0$ and calculate
\[ \colim_ig_{i*}\sh O_{Y_i}\cong g_*\sh O_Y\longrightarrow f_*g_*\sh O_{X}\cong \colim_if_*g_{i*}\sh O_{X_i} \]
as a filtered colimit of homomorphisms $ g_*\sh O_{Y_i}\rightarrow g_* f_{i*}\sh O_{X_i}$. 

In order for this homomorphism to be integral, it will be enough for each term $\sh O_{Y_i}\rightarrow f_*\sh O_{X_i}$ to be finite. But this follows from the finiteness of cohomology over projective morphisms in the Noetherian case.\end{proof}

\begin{cor}An affine and integral/projective morphism between formal schemes over $\Z$ is integral.\end{cor}

\begin{eg}Any affine embedding $Z\hookrightarrow Y$ is projective - in fact, a projective bundle $\P(\iota_*\sh O_Z)$. However, this fails for our standard example of a non-affine embedding - a finite union $Z=\sqcup_{i=1}^kZ_k$ of disjoint affine embeddings $Z_i\hookrightarrow Y$ - which is actually only an embedded open subset of \[\P(\iota_*\sh O_Z)=\P\left(\prod_{i=1}^k\iota_*\sh O_{Z_i}\right)=\Spec\left(\prod_{i=1}^k\iota_*\sh O_{Z_i}\right).\]
On the other hand, in this case $Z$ is the projectivisation of the finite Banach module $\bigoplus_{i=1}^k\sh O_{Z_i}$. In particular, this gives an example where a projective morphism may fail to be isomorphic to the Proj of its homogeneous co-ordinate ring.\end{eg}

\begin{question}\label{Q_PROJ_EMBED}Is every embedding the projectivisation of a module?\end{question}

\subsection{Blowing up and modification}\label{MORP_REES}

Let $X\in\mathbf{FSch}$, and let $T\trianglelefteq\sh O_X$ be a quasi-coherent ideal sheaf. The Rees algebra 
\[ R_T=\bigoplus_{n\in\N}T^n=\bigvee_{n\in\N}T^n \]
is an $\N$-graded $\sh O_X$-algebra generated in degree one. The Proj of the Rees algebra has principal affine charts of the form
\[ \sh O_X\{T/s\}=\colim\left[\sh O_X\stackrel{s}{\rightarrow} T^k\stackrel{s}{\rightarrow} T^{2k}\rightarrow\cdots\right] \]
for $s\in T^k$, with the colimit taken in the category of Banach $\sh O_X$-modules. It is the universal way to make $T$ an invertible module. It is of finite type if and only if $T$ is finitely generated, in which case it is projective, being an affinely embedded subscheme of the projective bundle $\P(T)$.

\begin{defn}[Admissible modification]\label{RIG_ADMIT_DEF}A \emph{blow-up} of $Y$ \emph{with centre $Z_0\subseteq Y$} is a morphism $p:\widetilde Y\rightarrow Y$ final among those for which $p^{-1}Z_0$ is an invertible divisor. Such is computed by taking Proj of the Rees algebra of the ideal of $Z_0$.

A $Z$-\emph{admissible blow-up}, or blow-up \emph{of the pair} $(Y,Z)\in{}_Z\mathbf{FSch}$ is a finite type blow-up of $Y$ along a centre whose underlying reduced formal scheme is supported in $Z$. In other words, it is a Cartesian square
\[\xymatrix{ (\widetilde Y,p^{-1}Z)\ar[r]\ar[d] & (\widetilde Y,p^{-1}Z_0)\ar[d]^{p=\mathrm{Bl}} \\ (Y,Z)\ar[r] & (Y,Z_0) }\]
in ${}_Z\mathbf{FSch}$ with $p$ the blow up along $Z_0$.

An \emph{admissible modification} of $(Y,Z)$ is a represented morphism $p:(X,p^{-1}Z)\rightarrow (Y,Z)$ such that $p^{-1}Z$ is a Cartier divisor, and which admits a factorisation
\[ X\rightarrow Y^\prime \rightarrow Y  \]
with $q:Y^\prime\rightarrow Y$ a $Z$-admissible blow-up along centre $Z_0\subseteq Z$ and $X\rightarrow Y^\prime$ a $q^{-1}Z_0$-admissible integral/projective morphism.\end{defn}

Blowing up is right adjoint
\[ \mathrm{Bl}:{}_Z\mathbf{FSch}\rightarrow{}_Z\mathbf{FSch}^\mathrm{inv} \]
to the inclusion of formal schemes with invertible marking into all marked formal schemes. It is a generalisation of the right adjoint ${}_Z\mathbf{FSch}^\mathrm{div}\rightarrow{}_Z\mathbf{FSch}^\mathrm{inv}$ defined for divisorial markings in \S\ref{MORP_NORM}.


\paragraph{Stability}
If $f:X\rightarrow Y$, then $f^*T\twoheadrightarrow f^{-1}T\sh O_X$ and so $f^*R_T\twoheadrightarrow R_{f^{-1}T\sh O_X}$; thus the Rees algebra is functorial for morphisms of formal schemes, and moreover no relevant ideal of $R_{f^{-1}T\sh O_X}$ can pull back to the irrelevant ideal of $R_T$. Thus we always get a commutative square
\[\xymatrix{\widetilde{X}\ar[d]_{\mathrm{Bl}_{X\times_Y Z}} \ar[r] & \widetilde{Y} \ar[d]^{\mathrm{Bl}_Z} \\ X \ar[r] & Y}\]
where $Z$ denotes the closed subscheme of $Y$ cut out by $T$.

If the morphism $f$ is flat, meaning that $f^*$ is exact on modules, then in fact the natural map of Rees algebras is an isomorphism and hence the square is Cartesian - so in particular, the restriction to an open set of a blow-up is a blow-up.

\begin{prop}\label{MORP_BLOWUP_STABLE}Admissible blow-ups are stable for flat base change.\end{prop}

The rest of this section addresses the failure of admissible blow-ups to be stable for general pullbacks or descent.

The universal property is enough to reproduce the arguments leading up to \cite[Prop. 3.1.17]{Abbes}:

\begin{lemma}\label{MORP_ADMISSIBLE}\begin{enumerate}\item A composite of admissible blow-ups is an admissible blow-up.

\item Let $\widetilde Y\rightarrow Y$ be a blow-up along a closed formal subscheme $Z\subseteq Y$, and let $X\rightarrow Y$. The blow-up of $X$ along $X\times_YZ$ is naturally isomorphic to the blow-up of $X\times_Y\widetilde Y$ along $X\times_Y\widetilde Y\times_YZ$.
\[\xymatrix{
\widetilde{X}\ar[rr]^-{\mathrm{Bl}_{X\times_Y\widetilde Y\times_YZ}}\ar[rrd]_{\mathrm{Bl}_{X\times_Y Z}} && X\times_Y\widetilde{Y} \ar[r]\ar[d] & \widetilde{Y} \ar[d]^{\mathrm{Bl}_Z} \\
&& X \ar[r] & Y
}\]
In particular, the saturation of the class of admissible blow-ups is stable for base change.

\item Let $Y_\bullet\twoheadrightarrow Y$ be an open cover, $p:\tilde Y\rightarrow Y$ a morphism whose restriction to $Y_\bullet$ is an admissible blow-up along ideal $T_\bullet$. Then the blow-up of $Y$ along $\prod_iT_i$ is naturally isomorphic to the blow-up of $\tilde Y$ along $p^*\prod_i T_i$.
\[ \xymatrix{ \tilde{\tilde Y}\ar[rdr]_{\mathrm{Bl}_{\prod T}}\ar[rr]^{\mathrm{Bl}_{p^*\prod T}} && \tilde{Y}\ar[d]^p &  \tilde{Y}_\bullet\ar[d]^{\mathrm{Bl}_{T_\bullet}}\ar@{->>}[l] \\
&& Y & Y_\bullet\ar@{->>}[l]  } \]
In particular, the saturation of the class of admissible blow-ups is local on the base.\end{enumerate}\end{lemma}
\begin{proof}Parts \emph{i)} and \emph{ii)} follow as in \cite[3.1.14, 17]{Abbes}, respectively; the essential part of the argument for part \emph{i)} can be found in \cite[Lemme 5.1.4]{RayGru}. 

For part \emph{iii)}, let $\tilde{\tilde Y}$ be the blow-up of $\tilde Y$ along $p^*\prod_iT_i$, where we extend $T_i$ to $Y$ by taking the closure. By proposition \ref{MORP_BLOWUP_STABLE}, the restriction of $\tilde{\tilde Y}\rightarrow\tilde Y$ to each $\tilde Y_i$ is the blow-up along $p^*\prod_iT_i$, which is the blow-up of $Y_i$ along $\prod_iT_i$.\end{proof}

\paragraph{}
Finite type blow-ups generate all projective morphisms under pullback. Indeed, let ]$A$ be a finitely generated $A_0$-algebra with $A_0$ integral over $Y$, and suppose, without loss of generality, that $A$ is generated in degree one (so that $\sh O(1)$ is very ample on $\Proj A$). Then the blow-up of $\Spec A$ along the irrelevant ideal $A_+=\bigoplus_{n>0}A_n$ is the total space $\mathbb{V}(\sh O(1))$ of the tautological line bundle on $\Proj A$.

Its pullback along the augmentation $A\rightarrow A_0$ is Proj of the graded algebra \[\bigoplus_{n\in\N} A_+^n/A_+^{n-1}\cong\bigoplus_{n\in\N}A_n=A\]
that is, the zero section of $\mathbb{V}(\sh O(1))$:
\[\xymatrix{ \Proj A\ar[r]\ar[d] & \mathbb{V}(\sh O(1)) \ar[d] \\ \Spec A_0 \ar[r]^0 & \Spec A }\]
If $\Proj A\rightarrow Y$ is $Z$-admissible, then $\mathbb V(\sh O(1))\rightarrow\Spec A$ can also be written as a blow-up along $p^{-1}Z\cap0$, where $p:\Spec A\rightarrow\Spec A_0$ is the projection.

It follows from this and part \emph{ii)} of lemma \ref{MORP_ADMISSIBLE} that the saturation of the class of $Z$-admissible blow-ups contains the class of finite type $Z$-admissible modifications.

In fact, it is possible to obtain a much more precise statement in many cases:

\begin{thm}[Projective birational $\Rightarrow$ blow-up]\label{MORP_PROJECTIVE=BLOWUP}Let $Y$ be a qcqs formal scheme with an ample invertible sheaf and $Z\hookrightarrow Y$ a Cartier divisor. Let $f:\Proj A\rightarrow Y$ be a $Z$-admissible projective morphism such that $f^{-1}Z$ is invertible. Then $f$ is a $Z$-admissible blow-up.\end{thm}
\begin{proof}Same proof as in \cite[II.7.17]{Hartshorne}.\end{proof}


\begin{cor}The class of admissible modifications is stable under composition.\end{cor}
\begin{proof}In light of lemma \ref{MORP_ADMISSIBLE}, part \emph{i)}, we only need show that a finite type blow-up $\mathrm{Bl}_Z:\widetilde X\rightarrow X$ of an admissible integral modification $X\rightarrow Y$ is a modification. Since $Z$ is finitely presented, it has a model $Z_0\hookrightarrow X_0$ under $X$ with $X_0$ finite over $Y$. Let $\widetilde X_0$ denote the corresponding blow-up. By part \emph{ii)} of lemma \ref{MORP_ADMISSIBLE}, $\widetilde X\rightarrow\widetilde X_0$ is the composite of the blow-up of a Cartier divisor, which is an affine embedding, and an integral morphism.\end{proof}

\begin{cor}\label{MORP_PROJECTIVE_LOCAL}Let $(Y,Z)$ be a qcqs marked formal scheme. Let $Y_\bullet\twoheadrightarrow Y$ be an open covering, $\tilde Y_\bullet\rightarrow Y_\bullet$ an admissible projective modification. After possibly refining $Y_\bullet$, there exist admissible blow-ups $\tilde{\tilde Y}\rightarrow Y$ and $\tilde{\tilde Y}\times_YY_\bullet\rightarrow\tilde Y_\bullet$ making the diagram
\[\xymatrix{ \tilde{\tilde Y}\times_YY_i\ar@{^{(}->}[r]\ar[d] & \tilde{\tilde Y}\ar[dd] \\ \tilde Y_i\ar[d] \\ Y_i\ar@{^{(}->}[r] & Y}\] commute.\end{cor}
\begin{proof}By blowing up we may assume that $Z$ is a Cartier divisor (cf. lemma \ref{MORP_ADMISSIBLE}). After passing to a finite, affine refinement of $Y_\bullet$, we are in the situation of theorem \ref{MORP_PROJECTIVE=BLOWUP}, so that each member $\tilde Y_i$ is a blow up of $Y_i$ along a centre $Z_i\subseteq Z\cap Y_i$. Let $\tilde{\tilde Y}$ be the scheme obtained by blowing up $Y$ along the closures $\bar Z_i$ of the $Z_i$ in any order.\end{proof}




\subsection{Finiteness for blow-ups}\label{RIG_LEMMA}

The purpose of this technical section is to establish a very weak form of finiteness in sufficient generality to understand the \emph{affine} picture of rigid analytic geometry.

\begin{lemma}\label{RIG_THELEMMA}Let $(Y,Z)$ be a divisorially marked formal $\F_1$-scheme, $f:\widetilde Y\rightarrow Y$ an admissible blow-up. Then $\widetilde Y$ may be supered by an admissible blow-up $f:\widetilde Y^\prime\rightarrow Y$ such that $f_*\sh O_{\widetilde Y^\prime}$ is a finite $\sh O_Y$-algebra.
\end{lemma}
\begin{proof}Any finite type admissible blow-up can be supered by a composite of admissible blow-ups along ideals generated by \emph{two} elements. So we may assume $T=(t_1,t_2)\trianglelefteq \sh O_{Y}$.

Suppose we have a global function $f$ on $\widetilde Y$. Its restriction $f_i$ to the affine open subset $(t_i\neq 0)$ is a rational function $\tilde f_i/t_i^{n_i}$, with $\tilde f_i\in T^{n_i}$ an element of the Rees algebra in degree $n_i$. We may choose the representing quotient so that $\tilde f_i=c_it_j^{n_i}$ with $j\neq i$. The existence of the global function $f$ gives us the relation
\[ \tilde f_1t_2^{n_2}= \tilde f_2t_1^{n_1} \]
in $T^{n_1+n_2}$.

If either $n_1$ or $n_2$ is zero, we are done. Otherwise,
 \[\tilde f_1^{(n_1+n_2)}=c_1^{n_2}(\tilde f_1t_2^{n_2})^{n_1}=c_1^{n_2}(\tilde f_2t_1^{n_1})^{n_1}=c_1^{n_2}c_2^{n_1}t_1^{n_1(n_1+n_2)}\]
holds in degree $n_1(n_1+n_2)$ of the Rees algebra. That is, $c_1^{n_2}c_2^{n_1}=f_1^{n_1(n_1+n_2)}$ is in the image of $\sh O_{Y}$ inside $f_*\sh O_{\widetilde Y}$.
\end{proof}

\section{Separation and overconvergence}\label{SEP}

We would like our definitions of separated, proper, and overconvergent morphisms to be valid over both $\F_1$ and $\Z$, and operate in a fairly uniform manner not only for schemes, formal schemes, and rigid analytic spaces, but also for the panoply of additional topoi that crop up throughout this work. It will be convenient, therefore, to allow ourselves the flexibility of working in an arbitrary coherent spatial geometric context, with a specified class $\P$ of morphisms satisfying the key properties obeyed by \emph{projective} morphisms in the topos of schemes.

All of the arguments of sections \ref{SEP_NHOOD}, \ref{SEP_EXTENS}, and \ref{SEP_COMPARE}, which are essentially just some games you can play with commuting diagrams, are valid at this level of abstraction.

\subsection{Overconvergent neighbourhoods}\label{SEP_NHOOD}
Let $\mathbf S$ be a spatial theory, $\mathbf C\rightarrow\mathbf S$ a spatial site closed under fibre products (def. \ref{TOPOS_DEF}). For the purposes of this paper, we may as well assume that $\mathbf S$ is coherent and that all representable objects are compact; however, this will not be necessary for what follows.

In this section and from now on, we fix a class of qcqs \emph{generating overconvergent morphisms} $\P$ in $\mathbf C$, satisfying:
\begin{itemize}\item[(P1)]A composite of $\P$ morphisms is $\P$;
\item[(P2)]A base change of a $\P$ morphism is $\P$;
\item[(P3)]The diagonal of a $\P$ morphism is $\P$.\hfill\emph{$\P$ morphisms are separated}\end{itemize}

\begin{defn}\label{SEP_SUR}Let $U\hookrightarrow V$ be a quasi-compact open immersion. A $\P$-\emph{overconvergent neighbourhood} of $U/V$ is a factorisation $U\rightarrow \tilde U\rightarrow V$ such that $\tilde U\rightarrow V$ is $\P$. We will suppress the $\P$ from the notation if it is to be understood from the context (which it will not always be).\end{defn}

Let us also suppose
\begin{itemize}\item[(P4)]every overconvergent neighbourhood of $U/V$ may be supered by one such that the inclusion of $U$ is an open immersion;\end{itemize}
so that after refinement, it makes sense to talk about overconvergent neighbourhoods of an overconvergent neighbourhood. In our applications, it will be possible to make this replacement \emph{functorial}.

\paragraph{Overconvergent germ}
Since $\P$ is stable under composition and base change, the set of $\P$ morphisms with fixed target $V$ is always cofiltered. The \emph{overconvergent germ} of $U/V$ is the pro-object
\[ \mathrm{Sur}_{U/V}:= \lim_{U\rightarrow\tilde V\rightarrow V}\tilde V \in \Pro(\mathbf C_V) \]
given as the formal limit of overconvergent neighbourhoods of of $U$. Since $\P$ is stable under base change, it extends to a presheaf
\[\mathrm{Sur}^\mathrm{pre}_{U/V}:\sh U_{/V}\rightarrow\Pro(\mathbf C_V)\]
of pro-objects on the small site of $V$. 

If $U\rightarrow\tilde V\rightarrow V$ is an overconvergent neighbourhood of $U/V$, then using axiom (P4) we may define the overconvergent germ $\mathrm{Sur}_{U/\tilde V}$. The corresponding map
\begin{align}\label{crunt} \mathrm{Sur}^\mathrm{pre}_{U/\tilde V}\rightarrow\mathrm{Sur}^\mathrm{pre}_{U/V} \end{align}
is an isomorphism of presheaves.

\paragraph{Localising overconvergent neighbourhoods}
We have not required any compatibility between $\P$ and coverings in $\mathbf C$. Indeed, in our examples $\P$, and hence the class of $\P$-overconvergent neighbourhoods, will not be local on the base. In other words, the presheaf $\mathrm{Sur}^\mathrm{pre}_{U/V}$ is usually not a sheaf.

We may apply the plus construction to turn $\mathrm{Sur}^\mathrm{pre}_{U/V}$ into a sheaf $\mathrm{Sur}_{U/V}:=(\mathrm{Sur}_{U/V}^\mathrm{pre})^+$ on $V$. By definition, the sections of $\mathrm{Sur}_{U/V}$ over $V$ are the covariant functor
\begin{align}\label{pleen} \Hom_S(\mathrm{Sur}_{U/V},-):\mathbf S_S\rightarrow\mathbf{Set},\quad X\mapsto \colim_{V_\bullet\twoheadrightarrow V}\colim_{\tilde V_\bullet\stackrel{\P}{\rightarrow} V_\bullet}\Hom_S(\tilde V_\bullet, X). \end{align}
The caveat is that after passing to an overconvergent neighbourhood $\tilde V$ of $U/V$, there may be new coverings available that are not pulled back from $V$. That is, the morphism $\mathrm{Sur}_{U/\tilde V}\rightarrow\mathrm{Sur}_{U/V}$ coming from (\ref{crunt}) may no longer be an isomorphism.

This leads us to formulate the axiom of compatibility between $\P$ and coverings in $\mathbf C$:

\begin{itemize}\item[(SC)] If $f:V\rightarrow S$ is $\P$, then $(f_*\mathrm{Sur}^\mathrm{pre}_{U/V})^+\tilde\rightarrow f_*\mathrm{Sur}_{U/V}$.\end{itemize}

Thus for all $V$ sufficiently small, the outer colimit drops out of \ref{pleen}, and we may assume overconvergent neighbourhoods are defined globally.

Unravelling this condition, we obtain an explicit, if slightly unwieldy, criterion:
\begin{itemize}\item[(SC')] Let $V\rightarrow S$ be $\P$, $U\hookrightarrow V$ a quasi-compact open immersion. Let $V_\bullet\twoheadrightarrow V$ be an open cover of $V$, $\tilde V_\bullet\rightarrow V_\bullet$ an overconvergent neighbourhood of $U_\bullet/V_\bullet$:
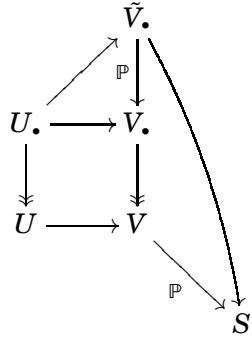
\[\xymatrix{ & \tilde V_\bullet\ar[d]_\P\ar@/^/[dddr] \\
U_\bullet\ar[r]\ar[ur]\ar@{->>}[d] & V_\bullet\ar@{->>}[d] \\
U\ar[r] & V\ar[dr]_\P \\ && S
}\]
Then, after possibly passing to a cover of $S$, there exists an overconvergent neighbourhood of $U/V$ whose restriction to $V_\bullet$ factors through $\tilde V_\bullet$.\end{itemize}

This axiom is the only subtle part of the theory. For projective morphisms of schemes, it is a consequence of corollary \ref{MORP_PROJECTIVE_LOCAL}.

\begin{lemma}\label{SEP_BLARG}The system $\P$ satisfies axiom $\mathrm{(SC)}$ if and only if for all overconvergent neighbourhoods $\tilde V$ of $U/V$,
\[ \mathrm{Sur}^+_{U/\tilde V}\tilde\rightarrow\mathrm{Sur}^+_{U/V} \]
is an isomorphism of pro-objects.\end{lemma}

Axiom (SC) makes sense even in the absence of (P4). One can formulate a local version of the latter
\begin{itemize}\item[(P4')]every overconvergent neighbourhood of $U/V$ may \emph{locally} be supered by one such that the inclusion of $U$ is an open immersion;\end{itemize}
which, under (SC), is equivalent to (P4).

\

In the sequel, we will supress the superscript $+$ from the notation of the overconvergent germ, and confuse $\mathrm{Sur}_{U/V}$ with its sections over $V$. Finally, we will use the term \emph{overconvergent neighbourhood of $U/V$} more generally to mean any object between $\mathrm{Sur}_{U/V}$ and $V$.

\subsection{Extension problems}\label{SEP_EXTENS}

Suppose we have fixed a system $\P$ of generating overconvergent morphisms satisfying the axioms (P1-4) and (SC). By localising, it makes sense to speak of overconvergent neighbourhoods in $\mathbf S$, and not merely in the site $\mathbf C$.

Suppose we are given a diagram
\[\xymatrix{ U\ar@{^{(}->}[r]\ar[d] & V\ar[d] \\ X \ar[r] & S }\]
in $\mathbf S$, with $U\hookrightarrow V$ a quasi-compact open immersion. The \emph{extension problem} is to find, locally on $V$, an overconvergent neighbourhood $\tilde V$ of $U$ in $V$ and an extension $\tilde V\rightarrow X$:
\[\xymatrix{U\ar@{^{(}->}[r]\ar[d] & \tilde V\ar@{-->}[dl]\ar[r] & V\ar[d] \\
X \ar[rr] && S}\]
Assuming such an extension exists, we would also like to know that it is unique up to passing to a further modification of $\tilde V$. 

Note that an extension problem
\[\xymatrix{U\ar[r]\ar[d] & V\ar[d] \\
X \ar[r] & X\times_S X}\]
is the same as a pair of extensions
\[\xymatrix{U\ar[r]\ar[d] &  V \ar@<2pt>[dl]\ar@<-2pt>[dl] \ar[d] \\
X \ar[r] & S }\]
in $\mathbf S$, and a solution to the first problem is a proof of equality of the two solutions to the second. Thus uniqueness of extensions for a morphism is the same as existence for its diagonal.

We may also disappear $S$ from these diagrams by working in $\mathbf S_S$; though note that the latter is only coherent if $S$ is qcqs.

\begin{defns}[Extensional properties]\label{SEP_EXTENS_DEF}A morphism $X\rightarrow S$ in $\mathbf S$ is $\P$-\emph{overconvergent near $U\rightarrow X$} if for any quasi-compact open immersion $U\hookrightarrow V$ and morphism $U\rightarrow X$ over $S$, there is a unique extension $\mathrm{Sur}_{U/V}\rightarrow X$ making the diagram
\[\xymatrix{U\ar@{^{(}->}[r]\ar[d] & \mathrm{Sur}_{U/V}\ar@{-->}[dl]\ar[r] & V\ar[d] \\
X\ar[rr] && S}\]
commute. It is said to be $\P$-\emph{overconvergent} if it is $\P$-overconvergent near every object $U$ over $X$, and $\P$-\emph{proper} if it is $\P$-overconvergent and qcqs.

We say $X/S$ is \emph{locally $\P$-separated} if its diagonal is $\P$-overconvergent. It is $\P$-\emph{separated} if it is locally separated and quasi-separated, that is, if its diagonal is proper. Every $\P$-overconvergent morphism is locally $\P$-separated.\end{defns}

In this and the following section, we will suppress the prefix $\P$; however, reader beware that in later sections, this abuse of notation \emph{will} cause confusion (def. \ref{SEP_DEF_PROPER}).

\paragraph{Canonical extensions}

If $U/V$ is an extension problem for $X/S$, then any solution uniquely factorises through $V\times_SX$. The latter therefore has the character of a `canonical solution' \[\xymatrix{U\ar[r]\ar[d] & V\times_S X\ar[r]\ar[dl] & V\ar[d] \\ X \ar[rr] && S}\]Indeed, a solution exists if and only if $\mathrm{Sur}_{U/V}\rightarrow V\times_S X$ over $V$.
Similarly, two extensions $V\rightrightarrows X$ agree if and only if $\mathrm{Sur}_{U/V}\rightarrow  V\times_{X\times X}X$ over $V$.

In particular:

\begin{prop}Every $\P$ morphism is proper.\end{prop}

In the situation that the replacement of (P4) can be made functorial, combining it with this construction gives a canonical solution $\tilde V$ with $U\hookrightarrow\tilde V$ a quasi-compact open immersion. Later, we will construct such solutions for formal schemes (def. \ref{SEP_ELEMENTARY}).

\paragraph{Proper neighbourhoods}
Let $U\hookrightarrow V$ be a quasi-compact open immersion, and suppose that $U\hookrightarrow \tilde V\rightarrow V$ is a factorisation with $\tilde V/V$ \emph{proper} with respect to $\P$. Then the definition of propriety, applied to the problem
\[\xymatrix{
U\ar[r]\ar[d] & V\ar@{=}[d] \\ \tilde V\ar[r]& V
}\]
implies that locally on $V$ there is a $\P$ cover of $\tilde V$ containing $U$. In other words, every proper neighbourhood is an overconvergent neighbourhood.

\begin{lemma}[Enlarging $\P$]\label{SEP_ENLARGE}Let $\P^\prime$ be the class of all $\P$-proper morphisms. Then $\P^\prime$ obeys $\mathrm{(P1-3)}$, \emph{(P4')}, and $\mathrm{(SC)}$. The theory of overconvergence for $\P^\prime$ is equivalent to that for $\P$.\end{lemma}
\begin{proof}Axioms (P1-3) follow from prop. \ref{SEP_STABILITY}; the remaining statements - (P4'), (SC) and that $\P^\prime$-overconvergent morphisms are $\P$-overconvergent - follow from the fact that by definition, any $\P^\prime$-overconvergent neighbourhood may locally be supered by a $\P$-overconvergent nehood.\end{proof}

\paragraph{Stability}
A straightforward unravelling of the definitions leads to the typical stability properties (cf. \cite[\textbf{I}.5.5.1 \& \textbf{II}.5.4.2-3]{EGA}):

\begin{prop}[Stability properties]\label{SEP_STABILITY}Extensional properties of morphisms are stable under the following constructions:
\begin{enumerate}\item composition;
\item base change;
\item descent.\end{enumerate}Every monomorphism is separated. 

Moreover, for any composable morphisms $f,g$:
\begin{enumerate}\item[iv)] if $fg$ is locally separated (resp. separated), then $g$ is locally separated (resp. separated);
\item[v)] if $fg$ is overconvergent (resp. proper) and $f$ is locally separated (resp. separated), then $g$ is overconvergent (resp. proper).\end{enumerate}\end{prop}
\begin{proof}I provide a proof only for the first part, by way of illustration. (The proofs of \emph{ii)} and \emph{iii)} are anyway addressed in example \ref{SEP_BASE_CHANGE}.)

Let $X\rightarrow Y\rightarrow Z$ be two overconvergent morphisms, and let $U/V$ be an extension problem for $X/Z$. Overconvergence of $Y/Z$ gives us a unique diagram
\[\xymatrix{ U\ar[r]\ar[d] & \mathrm{Sur}_{U/V} \ar[r]\ar@{-->}[d] & V\ar[d] \\ X\ar[r] & Y\ar[r]&Z }\]
which, by composability of $\mathrm{Sur}_{U/V}$, produces an extension problem $U/\tilde V$ for $X/Y$. Overconvergence of $X/Y$ then gives a unique solution $\mathrm{Sur}_{U/\tilde V}\rightarrow X$. By lemma \ref{SEP_BLARG}, $\mathrm{Sur}_{U/V}\cong\mathrm{Sur}_{U/\tilde V}$, whence the result.\end{proof}

\

Finally, though this is not critical for much of what follows, we will often also have recourse to another axiom:
\begin{itemize}\item[(P5)]For all quasi-compact open immersions $U\hookrightarrow V$, $\mathrm{Sur}_{U/V}$ is a sheaf in the $U$-variable.\end{itemize}
By quasi-compactness and locality on $V$, it is enough to check for quasi-compact $U$ and hence \emph{finite} coverings.

\begin{prop}\label{SEP_COLIMIT}Suppose that $\P$ satisfies $\mathrm{(P5)}$. Let $\{X_i\rightarrow S\}_{i\in I}$ be any family of overconvergent morphisms. Then $\colim_iX_i\rightarrow S$ is overconvergent.\end{prop}
\begin{proof}Let $U/V$ be an elementary extension problem for $\colim_iX_i/S$. Since colimits in a topos are universal, $U_i:=U\times_XX_i$ is a covering of $U$.\end{proof}

\begin{cor}Let $X_i\rightarrow S$ be a finite family of proper morphisms. Then $\coprod_iX_i\rightarrow S$ is proper. In particular, $\emptyset\rightarrow S$ is proper.\end{cor}

\subsection{Comparison principle}\label{SEP_COMPARE}

Suppose we have two spatial theories $\mathbf S_1,\mathbf S_2$ and a qcqs spatial geometric morphism
\[ \phi:\mathbf S_1\rightarrow\mathbf S_2. \]
It will be useful to have an understanding of how extensional properties are preserved or detected by various functors associated to $\phi$. For this, we will certainly need $\phi^*$ to also preserve $\P$, so that we always get a map
\[\mathrm{Sur}_{\phi^*U/\phi^*V}\rightarrow\phi^*\mathrm{Sur}_{U/V};\] in all our examples, this will be true by definition.

For example, although we will in \cite{part2} be mainly concerned with the rigid topos $\Sh\mathbf{Rig}$, we will also want to know that for formal schemes, separation and propriety can be \emph{detected} in the \emph{a priori} more tractable topos $\Sh\mathbf{FSch}$.

\begin{defn}Let $\mathbf S_1,\mathbf S_2$ be spatial theories on finitely complete sites $\mathbf C_1,\mathbf C_2$. We say that a left exact functor $F:\mathbf C_1\rightarrow\mathbf C_2$, or an extension thereof $\mathbf S_1\rightarrow \mathbf S_2$, \emph{preserves} (resp. \emph{detects}) \emph{overconvergence} if it preserves $\P$ and \[ X/S\text{ overconvergent} \quad \Rightarrow\text{ (resp. }\Leftarrow) \quad FX/FS\text{ overconvergent} \]
for any morphism $X\rightarrow S\in\mathbf C_1$. By left exactness of $F$, the same implication will then hold with `overconvergent' replaced by `locally separated'.\end{defn}

In our examples, $\mathbf C_i$ will be the site of all locally representable objects of $\mathbf S_i$, and $F$ will be either the pullback or pushforward functor associated to a spatial geometric morphism between $\mathbf S_1$ and $\mathbf S_2$. It will be important to distinguish whether $F$ preserves/detects overconvergence on all of $\mathbf S_1$ or merely on $\mathbf C_1$.

One can immediately make an elementary observation in the case $F$ is fully faithful:

\begin{lemma}\label{cf}If $F$ has a left exact left inverse $GF\cong 1$, and $G$ preserves (resp. detects) overconvergence, then $F$ detects (resp. preserves) overconvergence.\end{lemma}

More importantly, there is a farrago of other criteria that we will use throughout this paper and its sequel.

\begin{lemma}[Comparison criteria]\label{SEP_THELEMMA}Let $\phi:\mathbf S_1\rightarrow\mathbf S_2$ be an (essential) spatial geometric morphism whose pullback preserves $\P$.
\begin{enumerate}
\item Suppose that for any $U/V$ in $\mathbf C_2$, any overconvergent neighbourhood of $\phi^*U/\phi^*V$ can be dominated by $\phi^*$ applied to an overconvergent neighbourhood of $U/V$. That is,  \[\phi^*\mathrm{Sur}_{U/V}\rightarrow\mathrm{Sur}_{\phi^*U/\phi^*V} \rightarrow \phi^*V\] in $\mathbf S_1$. Then $\phi_*$ preserves overconvergence.
\item[i')] Suppose that for any $U/V$ in $\mathbf C_1$, any overconvergent neighbourhood of $\phi_!U/\phi_!V$ can be dominated by $\phi_!$ applied to an overconvergent neighbourhood of $U/V$. That is, \[\phi_!\mathrm{Sur}_{U/V}\rightarrow \mathrm{Sur}_{\phi_!U/\phi_!V} \rightarrow \phi_!V \] in $\mathbf S_2$. Then $\phi^*$ preserves overconvergence.
\item Suppose that for any $S$ and quasi-compact open immersion $U\hookrightarrow V$ in $\mathbf C_2$, the square
\[\xymatrix{ \Hom_S(\mathrm{Sur}_{U/V},-)\ar[r]\ar[d] & \Hom_{\phi^*S}(\mathrm{Sur}_{\phi^*(U/V)},\phi^*(-))\ar[d] \\ \Hom_S(U,-)\ar[r] & \Hom_{\phi^*S}(\phi^*U,\phi^*(-)) }\]
is Cartesian. Then $\phi^*$ detects overconvergence.
\item[ii')] In the situation of ii), it is enough that $\Hom_S(\mathrm{Sur}_{U/V},-)$ surject onto the fibre product.
\item Suppose that $\phi_!$ preserves $\P$, and that for any $S$ in $\mathbf C_2$ and quasi-compact open immersion $U/V$ in $\mathbf C_1$, the square
\[\xymatrix{\Hom_S(\mathrm{Sur}_{\phi_!U/\phi_!V},-)\ar[r]\ar[d] & \Hom_{\phi^*S}(\mathrm{Sur}_{U/V},\phi^*(-))\ar[d] \\ \Hom_S(\phi_!U,-)\ar[r] & \Hom_{\phi^*S}(U,\phi^*(-))} \]
is Cartesian. Then $\phi^*$ preserves overconvergence.
\item[iii')] In the situation of iii), it is enough that $\Hom_S(\mathrm{Sur}_{\phi_!U/\phi_!V},-)$ surject onto the fibre product.
\item Suppose that for any morphism $X\rightarrow S$ in $\mathbf S_2$ and quasi-compact open immersion $U/V$ in $\mathbf S_1$, the natural map
\[ \Hom_S(\mathrm{Sur}_{U^\prime/V^\prime},-)\rightarrow \Hom_{\phi^*S}(\mathrm{Sur}_{U/V},\phi^*(-))\]
is a colimit over the category of quasi-compact open immersions $U^\prime/V^\prime$ in $\mathbf S_2$ equipped with a map $U/V\rightarrow\phi^*(U^\prime/V^\prime)$. Then $\phi^*$ preserves overconvergence.
\end{enumerate}\end{lemma}
\begin{proof}\emph{i)} Using the adjunction property, an extension problem
\[\xymatrix{ U\ar@{^{(}->}[r]\ar[d] & V \ar[d] \\ \phi_*X \ar[r] & \phi_*S}\]
in $\mathbf S_2$ transforms into a problem
\[ \xymatrix{ \phi^*U\ar@{^{(}->}[r]\ar[d] & \phi^*V\ar[d]  \\ X \ar[r] & S}\]
in $\mathbf S_1$. If $X$ is overconvergent, there is a unique extension $\mathrm{Sur}_{\phi^*U/\phi^*V}\rightarrow X$. The condition ensures that $\mathrm{Sur}_{\phi^*U/\phi^*V}\rightarrow\phi^*\mathrm{Sur}_{U/V}$ is an isomorphism of pro-objects. Therefore by adjunction again, $\mathrm{Sur}_{U/V}\rightarrow \phi_*X$ uniquely and $\phi_*X$ is overconvergent. Condition \emph{i')} follows the from the same argument.

The remaining criteria are clear from the definition, which in general concerns the bijectivity of arrows \[\Hom_S(\mathrm{Sur}_{U/V},-)\rightarrow\Hom_S(U,-).\] The variants \emph{ii'), iii')} follow by considering the diagonal.\end{proof}

\begin{eg}[Base change]\label{SEP_BASE_CHANGE}The above principles apply to the pullback along the essential spatial geometric morphism $\phi:\mathbf S_{S^\prime}\rightarrow\mathbf S_S$ to prove parts \emph{ii)} and \emph{iii)} of prop. \ref{SEP_STABILITY}.

\emph{$\phi^*$ preserves overconvergence.} Apply criterion \emph{i')} of lemma \ref{SEP_THELEMMA}; $\phi_!$ is the functor that forgets the base morphism to $S^\prime$, and hence does not affect the definition of $\mathrm{Sur}_{U/V}$.

\emph{$\phi^*$ detects overconvergence:} When $\phi:S^\prime\twoheadrightarrow S$, then $\phi^*$ is comonadic, and one can show that this implies that the square in part \emph{ii)} of \ref{SEP_THELEMMA} is always Cartesian.\end{eg}

\subsection{Overconvergence in formal geometry}\label{SEP_FSCH}

In the category of formal schemes, we distinguish the following four classes of morphisms.
\begin{center}
\begin{tabular}{ r l r l }
$\P$ & projective  & 
$f\P$ & formally projective \\
$i/\P$ & integral/projective &
$f\!i/\P$ & formally integral/projective
\end{tabular}\end{center}
The intersections of the classes in the second column with the Zariski site $\mathbf{Sch}$ are exactly the classes in the first column. The classes $i/\P$ and $f\!i/\P$ are only of passing interest, and it is possible, if a little unnatural, to entirely avoid mentioning them. The intersection of $i\P$ with the site of schemes of finite type over some base $S$ is $\P$.

\subsubsection{Elementary extension problems}
Let $X\rightarrow S$ be quasi-separated, $U/V$ quasi-compact. Then $U\rightarrow X\times_SV$ is quasi-compact, and we may pass to the embedded image $\tilde V$ of the affinisation of this morphism. By proposition \ref{MORP_OPEN}, $U\rightarrow \tilde V$ is a dense, quasi-compact open immersion. The inclusion $U\hookrightarrow \tilde V$ can now be made affine by blowing up the reduced formal scheme with support $V\setminus U$. The resulting morphism 
\[\tilde V=\mathrm{Bl}_{V\setminus U}\left(\mathrm{cl}(\Spec\sh O_V(U)/V)\right) \rightarrow V\]
is projective; thus $\tilde V$ is an overconvergent neighbourhood of $U/V$ with respect to $f\P$ (or $\P$ if all players are schemes). 

Even if $X$ is not quasi-separated, by writing it as a filtered colimit $\colim_iX_i$ of quasi-separated objects in $\Sh\mathbf{FSch}$, one obtains by compactness of $U$ a factorisation $U\rightarrow V\times_SX_i$ from which one can construct a canonical extension (depending, of course, on $i$).

In the special case $X=V$, the canonical solution to the extension problem is simply obtained by blowing up $\mathrm{cl}(U/V)$ along $V\setminus U$; in particular, it is $f\P$ over $V$ (even $\P$ if $V$ is a scheme). This naturally transforms any extension problem into one for which the inclusion is affine and dense.

\begin{defn}\label{SEP_ELEMENTARY}Let $X\rightarrow S$ be quasi-separated, and let $U/V$ be an extension problem. The morphism $X\leftarrow \tilde V\rightarrow V$ constructed above is called the \emph{canonical extension}. The inclusion $U\hookrightarrow \tilde V$ to the canonical extension is an affine and (scheme-theoretically) dense open immersion (def. \ref{MORP_EMBEDDING}).

An extension problem $U/V$ in $\mathbf{FSch}$ is called \emph{elementary} if $V$ is affine and $U\hookrightarrow V$ is affine and dense.\end{defn}

\begin{lemma}Every extension problem for $f\P$ in $\Sh\mathbf{FSch}$ (resp. $\P$ in $\Sh\mathbf{Sch}$) can be covered by elementary extension problems. That is, if $U/V$ is any extension problem, then there exist elementary extension problems $U_\bullet/V_\bullet$ and a covering $\mathrm{Sur}_{U_\bullet/V_\bullet}\twoheadrightarrow\mathrm{Sur}_{U/V}$.\end{lemma}

\subsubsection{Overconvergence for schemes and formal schemes}In cases of interest, the classes $\P,i/\P,f\P,f\!i/\P$ all define the same notion (def. \ref{SEP_DEF_PROPER}) of overconvergence.

\begin{lemma}\label{SEP_GOOD}The classes $\P$ and $i/\P$ in $\mathbf{Sch}$ and $f\P$ and $f\!i/\P$ in $\mathbf{FSch}$ obey the axioms \emph{(P1-5)} and  \emph{(SC)}.\end{lemma}
\begin{proof}Axioms (P1-3) are handled by proposition \ref{MORP_PROJ_STABILITY}, while (P4) follows from the construction of canonical solutions for formally projective morphisms (which are, in particular, quasi-separated.

The locality axiom (SC) is a consequence of corollary \ref{MORP_PROJECTIVE_LOCAL}.

For (P5), over $\F_1$, after reducing the question to an elementary extension problem, the covering condition is trivial and so there is nothing to check. For $\Z$, the result is much harder and relies on equating our approach to propriety with the classical one (thm. \ref{SEP_ME=EGA}). Assuming this, one may prove it as follows.\footnote{For the reader concerned about circular reasoning: we will not use axiom (P5) until the sequel \cite{part2}.} 

Let $U/V$ be an elementary extension problem, $U_\bullet\twoheadrightarrow U$ a covering, $\mathrm{Sur}_{U_\bullet/V}\rightarrow X$ a morphism. We may assume $X$ is quasi-separated. For each $U_i$, the extension $\tilde V_i=\mathrm{cl}(U_i/V\times X)$ is proper over $V$. Since closure commutes with finite unions, we have a surjective map
\[ \coprod_i\tilde V_i\rightarrow \tilde V \]
to the extension $\tilde V=\mathrm{cl}(U/V\times X)\rightarrow X$. By \cite[\textbf{II}.5.2.3.\emph{ii)}]{EGA}, $\tilde V$ is proper over $V$. Thus $\mathrm{Sur}_{U/V}\rightarrow X$.\end{proof}

\begin{remark}[Why formally projective?]\label{SEP_REMARKS}The conclusion of lemma \ref{SEP_GOOD} is false for the class $\P$ in $\mathbf{FSch}$; it is essential to allow \emph{formally} projective, rather than simply projective, modifications. For instance, the proof of axiom (P4) relies on passing to the embedded closure of an open immersion $U\hookrightarrow V$ - which in general fails to be representable by schemes over $V$.

More importantly, the ability to also take point-set-topological closures within $\P$ is a necessary condition for (SC). Indeed, this axiom requires that for sufficiently small $V$ and any open immersion $V^\prime\hookrightarrow V$ disjoint from $U$, we should be able to find a modification of $V$ whose pullback to $V^\prime$ factors through the overconvergent neighbourhood $\emptyset$ of $U\cap V^\prime/V^\prime$.\end{remark}

\begin{prop}\label{SEP_FIP}Let $f$ be a locally finite type morphism in $\Sh\mathbf{FSch}$. The following are equivalent:
\begin{enumerate}\item $f$ is $f\!i/\P$-overconvergent in $\Sh\mathbf{FSch}$;
\item $f$ is $f\P$-overconvergent in $\Sh\mathbf{FSch}$.\end{enumerate}
To establish overconvergence of $f$, it is enough to exhibit solutions to extension problems $U/V$ with $V$ a local formal scheme.

Suppose that $f$ in fact lies in the subcategory $\Sh\mathbf{Sch}$. Then the above are moreover equivalent to the following:
\begin{enumerate}\item[iii)] $f$ is $i/\P$-overconvergent in $\Sh\mathbf{Sch}$;
\item[iv)] $f$ is $\P$-overconvergent in $\Sh\mathbf{Sch}$.\end{enumerate}
To establish overconvergence of $f$, it is enough to exhibit solutions to extension problems $U/V$ with $V$ a local scheme.
\end{prop}

In light of proposition \ref{SEP_FIP}, the definitions below are unambiguous. The proof and discussion of this fact occupies the rest of this section (\ref{SEP_FSCH}). 

\begin{defns}\label{SEP_DEF_PROPER}A morphism of formal schemes is said to be \emph{overconvergent} (resp. \emph{proper}) if it is locally of finite type (resp. and quasi-compact) and satisfies the equivalent conditions \emph{i), ii)} of proposition \ref{SEP_FIP}. A morphism is \emph{locally separated} (resp. \emph{separated}) if its diagonal is overconvergent (resp. proper).

A morphism of schemes is overconvergent (resp. proper, locally separated, separated) if it is so when considered as a morphism of formal schemes; equivalently, if it satisfies the equivalent conditions \emph{iii), iv)} of \ref{SEP_FIP}.\end{defns}

\begin{proof}[Proof of \ref{SEP_FIP}.] The implications \emph{ii)}$\Rightarrow$\emph{i)} and \emph{iv)}$\Rightarrow$\emph{iii)} are the easy directions, by the inclusions $\P\subset i/\P$ and $f\P\subset f\!i/\P$. The reverse implications, as well fact that these properties may be checked on local objects, follow from a standard compactness argument; for more discussion, see \S\ref{SEP_FINITENESS} below.

Let us focus on the two possible definitions for schemes. In the language of \S\ref{SEP_COMPARE}, we want to show that the inclusion \[(\phi^*\Sh\mathbf{Sch},\P)\hookrightarrow(\Sh\mathbf{FSch},f\P)\] detects and preserves overconvergence.

To show \emph{i/ii)}$\Rightarrow$\emph{iii/iv)}, it will be enough to show that the pushforward $\phi_*:\Sh\mathbf{FSch}\rightarrow\Sh\mathbf{Sch}$ preserves overconvergence (cf. lemma \ref{cf}). For this, we will use the comparison criterion \emph{i)} of lemma \ref{SEP_THELEMMA}.

Let $U/V$ be an affine open immersion in $\mathbf{Sch}$. We want to know that the overconvergent germ $\mathrm{Sur}_{U/V}$ does not depend on whether we consider $U/V$ as objects of $\mathbf{Sch}$ or of $\mathbf{FSch}$. In other words, we must show that every $f\P$-overconvergent neighbourhood $\tilde V$ of $U/V$ can be dominated by a $\P$-overconvergent neighbourhood. It is enough to take the embedded closure of the affinisation of $U$ over $\tilde V$; this will always be a scheme and hence projective over $V$.

\begin{cor}\label{SEP_YONED}The Yoneda embedding $\mathbf{FSch}\hookrightarrow\Sh\mathbf{Sch}$ preserves overconvergence.\end{cor}

As for the converse statement, we must check that a $\P$-overconvergent scheme automatically has solutions to all extension problems with $U/V$ arbitrary \emph{formal} schemes. We will achieve this by showing that any such extension problem over a morphism $X\rightarrow S$ in $\Sh\mathbf{Sch}$ can be factored through an open immersion $U^\prime/V^\prime$ of schemes over $X/S$; this is criterion \emph{iv)} of lemma \ref{SEP_THELEMMA}.

\begin{lemma}\label{blargunsfish}Let $A\rightarrow A[f^{-1}]$ be a localisation, $B_f\rightarrow A[f^{-1}]$ a ring homomorphism, $\tilde f\in B_f$ a unit lifting $f$. Then $B:=B_f\times_{A[f^{-1}]}A\rightarrow B_f$ is a localisation at $\tilde f$, and $A[f^{-1}]\cong A\tens_BB_f$.\end{lemma}
\begin{proof}By commutativity of finite limits with filtered colimits, and of localisation with base change, respectively.\end{proof}

Let $X/S$ be a morphism in $\Sh\mathbf{Sch}$, $U/V$ an elementary extension problem for $X/S$ in $\mathbf{FSch}$. Suppose $U$ is defined by the non-vanishing of $f\in\sh O(V)$. We will construct an extension problem $U^\prime/V^\prime$ in $\mathbf{Sch}$ and a Cartesian square \[\xymatrix{ U \ar[r]\ar[d] & V \ar[d] \\ U^\prime \ar[r] & V^\prime }\] for which purpose we may replace $X$ and $S$ with affine open subsets.

Let $U^\prime:=\G_{m,X}$. The invertible function $f$ on $U$ defines a morphism $U\rightarrow U^\prime$. Define $V^\prime$ by the exactness of the pullback square
\[\xymatrix{ \sh O(V^\prime)\ar[r]\ar[d] & \sh O(X)[f^{\pm1}] \ar[d] \\ \sh O(V)\ar[r] & \sh O(U) }\]
By lemma \ref{blargunsfish} applied to the discrete quotients of $\sh O(V)$, $\sh O(V^\prime)\rightarrow\sh O(U^\prime)$ is a localisation. Therefore \emph{iii/iv)}$\Rightarrow$\emph{i/ii)}.
\end{proof}



\begin{remark}[Non-quasi-compact extension problems]For schemes, it is possible to reduce the definition of quasi-compactness (and therefore, by consideration of the diagonal, quasi-separatedness) to a certain non-quasi-compact extension problem.

Let $X=\bigcup_iU_i$ be a scheme written as a union of affine open subsets. Taking the product of $\sh O(U_i)$ in the category of discrete algebras, we obtain an open immersion
\[ \coprod_iU_i\hookrightarrow\Spec\prod_i\sh O(U_i). \]
A solution to the associated extension problem is a diagram \[\xymatrix{ & U^\prime\ar[d]\ar[dr]^f \\
\coprod_iU_i \ar@{^{(}->}[r]\ar@{^{(}->}[ur]  & \Spec\prod_i\sh O(U_i) & X}\] with the vertical arrow projective, and so in particular quasi-compact. Since the target is affine, we may assume that $U^\prime$ is covered by the pullbacks of finitely many $U_i$s along $f$. This finite list of affine sets covers each member of the cover $\{U_i\subseteq X\}$ and therefore $X$ itself.

The converse appears to require a good theory of quasi-affine morphisms, which is a little delicate in the non-quasi-compact case. Nonetheless, using \cite[\href{http://stacks.math.columbia.edu/tag/01P9}{01P9}]{stacksproject} it is possible to prove that the following are equivalent for a scheme $X/S$:
\begin{enumerate}
\item all extension problems $\iota:U\hookrightarrow V$ such that $\iota_*\sh O_U$ is $\sh O_V$-quasi-coherent have unique solutions;
\item $X\rightarrow S$ is $\P$-proper.
\end{enumerate}
Since as it stands the general criterion is a little clumsy, I omit it from the development.
\end{remark}

\subsubsection{Expansions}\label{Expanditures}
Let $U/V$ be an affine extension problem for formal schemes. We also suppose given a function $f\in\sh O(V)$ whose non-vanishing defines $U$. Let $Z\subseteq V$ be a finitely presented closed subscheme in which $Z_U:=Z\cap U$ is dense. We do not assume $U$ is dense in $V$. 

We define an operation called \emph{intermediate expansion} on the data $(U,V,Z)$, as follows: let $p:\tilde V\rightarrow V$ be the blow-up along $Z\cap(f=0)$, and write $\tilde U=\tilde V\setminus\left( p_*^{-1}(f=0)\right)$ where $p_*^{-1}$ denotes the operation of strict transform along the blow-up $p$ \cite[\href{http://stacks.math.columbia.edu/tag/080D}{080D}]{stacksproject}. Explicitly, if $T$ is the ideal defining $Z$, then
\[ \tilde U=\Spec\sh O_V\{T/f\} \]
and the ideal defining $p_*^{-1}Z\cap\tilde U$ is $T/f\trianglelefteq \sh O_V\{T/f\}$. We get a commutative diagram
\[\xymatrix{ \tilde U\ar[r] & \tilde V\ar[d] & p_*^{-1}Z\ar@{=}[d]\ar[l] \\ U\ar[u]\ar[r] & V & Z\ar[l] }\]
of formal schemes over $S$, with $U\hookrightarrow\tilde U$ an open immersion. Since the intersection of the discriminant locus with $Z$ is a Cartier divisor, the restriction $p_*^{-1}Z\rightarrow Z$ is an isomorphism. It follows that $Z_{\tilde U}:=Z\cap\tilde U$ remains dense in $Z$.

If $U$ is representable by schemes over some base formal scheme $S$, then so is the affine part $\tilde U$ of the intermediate expansion:

\begin{lemma}\label{SEP_EL_FINITE}Let $V\rightarrow S$, $U\hookrightarrow V$ an affine, dense open immersion whose complement is defined by an equation $(f=0)$. Suppose that $U$ is representable by schemes over $S$. Let $I_V\trianglelefteq\sh O_V$ be an ideal of definition. Then $\Spec\sh O_V\{I_V/f\}$ is representable by schemes over $S$.\end{lemma}
\begin{proof}The crux will be to show that $\sh O_V\{I_V/f\}$ is a Banach module over $\sh O_S$. Let $I_S$ be an ideal of definition for $\sh O_S$, and $I_{V/S}$ the image of $I_V$ in $\sh O_V/I_S$. We will show that $I_{V/S}$ is nilpotent in $\sh O_V\{I_V/f\}/I_S$ and hence that the latter is discrete.

Since by hypothesis $\sh O_V\{f^{-1}\}$ is Banach over $\sh O_S$, there is some $k\in\N$ for which $f^kI_{V/S}$ is nilpotent in $\sh O_V/I_S$. Thus \[ I_{V/S}^{k+1}=f^{k+1}(I_{V/S}/f)^{k+1}\subseteq f^{k+1}(I_{V/S}/f)=f^kI_{V/S} \]
is nilpotent in $\sh O_V\{I_V/f\}/I_S$.\end{proof}

By iterating intermediate expansions, there are some inductive constructions that are useful as technical tools in the sequel.

\begin{defn}[Expansions]\label{SEP_EXP_DEG}Write $(U_0,V_0,Z_0)=(U,V,Z)$. 

If we set $(U_{i+1},V_{i+1},Z_{i+1})=(\tilde U_i,\tilde V_i, Z_i\times_{V_i}\tilde V_i)$, then the formal $V$-scheme \[U^\mathrm{\'el}:=\colim_{i\rightarrow\infty}U_i\] is called the \emph{expanded degeneration} of $U\subseteq V$.

If instead we write $(U_{i+1},V_{i+1},Z_{i+1})=(U_0,\tilde V_i,\tilde Z_i)$, then instead we get a projective system \[\mathrm{Sur}_{U/V}^Z:= \lim_{i\rightarrow\infty}V_i\]
defined as a pro-object of $\mathbf{FSch}_S$. It is the largest quotient of $\mathrm{Sur}_{U/V}$ admitting a section over $Z$. Now writing $(U_{i+1},V_{i+1},Z_{i+1})=(U_0,\tilde U_i,\tilde Z_i)$, we get a pro-open immersion
\[\mathrm{sur}_{U/V}^Z:= \lim_{i\rightarrow\infty}U_i\hookrightarrow \mathrm{Sur}^Z_{U/V}. \]
 In fact, since the transition maps of this one are affine, $\mathrm{sur}_{U/V}^Z$ even has a limit as a formal scheme, affine and representable by schemes over $V$. It is usually not of finite type (or Noetherian), and so for our purposes it will be useful to anyway consider $\mathrm{sur}_{U/V}^Z$ as a pro-object.\end{defn}

\begin{eg}The universal cover of the Tate formal curve over a complete DVR $\sh O_K$ is an example of a $U^\mathrm{\'el}$, with initial data $V=\P^1_{\sh O_K},U=V\setminus\{0,\infty\},Z=\P^1_k$.\end{eg}

\begin{remark}Later, we will introduce certain topological objects parametrising toric formal schemes. The above constructions should be thought of in terms of the following analogies on the (cone over the) unit interval $[0,1]$:

\begin{center}\begin{tabular}{ l | l } $U/V$ & $[0,0]/[0,1]$\\ \hline 
$U^\mathrm{\'el}$ & $[0,1)$ \\ 
 $\mathrm{Sur}_{U/V}^Z$ & $[0,1]$ subdivided at points $2^{-k}, k\in\N$ \\
$\mathrm{sur}_{U/V}^Z$ & $\lim_{k\rightarrow\infty}[0,2^{-k}]$\end{tabular}\end{center}
This analogy can be made precise using rigid analytic geometry \cite{part2}.\end{remark}

\paragraph{Expanded degeneration} The purpose of the construction $U/V\mapsto U^\mathrm{\'el}/V$ is to replace a quasi-compact open immersion with an overconvergent morphism.

\begin{lemma}Let $U/V$ be an affine open immersion of formal schemes, $Z\hookrightarrow V$ a closed subscheme that set-theoretically contains $U$ (for example, one that contains an ideal of definition). Then the formally embedded closure $\mathrm{cl}(U/\tilde V)$ of $U$ in the intermediate expansion $\tilde V$ is contained in $\tilde U$.\end{lemma}
\begin{proof}This is a consequence of the elementary fact that a blow-up of the intersection of two subschemes separates the strict transforms of those subschemes.\end{proof}

\begin{prop}\label{EXPANDED_DEGENERATION}Let $U/V$ be an affine open immersion of formal schemes, $Z\hookrightarrow V$ a closed subscheme that set-theoretically contains $U$. Then $U^\mathrm{\'el}\rightarrow V$ is an overconvergent morphism.\end{prop}
\begin{proof}First note that by lemma \ref{SEP_EL_FINITE}, $U^\mathrm{\'el}\rightarrow V$ is locally of finite type. Let $U/V$ be an elementary extension problem. Then $U$ factors through some finite stage $U_i\subseteq V_i$ of the expanded degeneration. By functoriality of formally embedded closures, $V\times_{V_i}V_{i+1}$ then factors through $\mathrm{cl}(U_i/V_{i+1})\subseteq U_{i+1}$.\end{proof}

\paragraph{$Z$-rational overconvergent germ} The second and third constructions of definition \ref{SEP_EXP_DEG} provide certain canonical covers in the category of formal algebraic spaces.

\begin{lemma}\label{SEP_CLOSED_SUBSET_LEMMA}Suppose that $U$ is dense in $V$. The square
\[\xymatrix{ Z_U\ar[r]\ar[d] & Z\ar[d] \\ U\ar[r] & \lim\mathrm{sur}_{U/V}^Z }\]
is a pushout in the category of affine formal schemes.\end{lemma}
\begin{proof}Assume $V=\Spec A$ is affine. We are looking at morphisms \[\xymatrix{  & A[f^{-1}] \ar@{->>}[d] \\ A/I \ar@{^{(}->}[r] & A/I[f^{-1}]}\]and I'm claiming that the fibre product is just $\bigcup_{k\rightarrow\infty}A[I/f^k]\subseteq A[f^{-1}]$. Certainly this injects into the fibre product, so it will suffice to produce a section. If $g/f^k\in A[f^{-1}]$ has image in $A/I$, then it can be written in the form $h_1+h_2/f^k$ with $h_1\in A$ and $h_2\in I$.\end{proof}

\begin{prop}\label{SEP_CLOSED_SUBSET}Let $U/V$ be an elementary extension problem over $S$, and suppose that $V$ is local. The square
\[\xymatrix{ Z_U\ar[r]\ar[d] & Z\ar[d] \\ U\ar[r] & \mathrm{sur}_{U/V}^Z }\]
is a pushout over the category of formal algebraic spaces locally of finite type over $S$.\end{prop}
\begin{proof}Indeed, the global statement follows from the fact that $Z$ is local, and so $U$ must factor through any affine open subset of $X$ containing the image of the closed point of $Z$.\end{proof}

\subsubsection{Finiteness}\label{SEP_FINITENESS} In this section we'll take a closer look at the statements in proposition \ref{SEP_FIP} that depend essentially on the finiteness of $f$.

\begin{prop}\label{SEP_FP}Let $f:X\rightarrow S$ be a morphism in $\Sh\mathbf{FSch}$ or $\Sh\mathbf{Sch}$, locally of finite type (resp. presentation). To establish overconvergence of $f$, it is enough to exhibit solutions to extension problems $U/V$ either
\begin{enumerate}\item for all $V$ of finite type (resp. presentation) over $S$; or
\item for all local $V$ essentially of finite type (resp. presentation) over $S$.\end{enumerate}Over $\F_1$, one can check with $V$ local and of finite type (resp. presentation) over $S$.\end{prop}
\begin{proof}For part \emph{i)}, it is equivalent to check that for any $S$, the inclusions
\begin{align}\label{first} \Sh\mathbf{FSch}_S^\mathrm{lpf} \hookrightarrow &\Sh\mathbf{FSch}_S^\mathrm{ltf}\hookrightarrow \Sh\mathbf{FSch}_S \\ \label{second} \Sh\mathbf{Sch}_S^\mathrm{lpf} \hookrightarrow &\Sh\mathbf{Sch}_S^\mathrm{ltf}\hookrightarrow \Sh\mathbf{Sch}_S \end{align}
of the topoi of (formal) schemes locally of finite presentation, resp. type, over $S$ preserves overconvergence. If we consider the right-hand topoi to be endowed with classes $f\!i/\P,i/\P$, respectively, then this argument will also fill in the details of the proof of proposition \ref{SEP_FIP}. 

It is also worth mentioning that the topoi of finitely presented objects do not obviously satisfy (P4), as a canonical solution (def. \ref{SEP_ELEMENTARY}) need not be finitely presented. The r\^ole of these topoi is sufficiently auxiliary that this will not cause us any serious problems.

\begin{lemma}If $X/S$ is locally of finite type (resp. presentation), then for any quasi-compact extension problem $U/V$ there exists a finite type (resp presentation) extension problem $U^\prime/V^\prime$ over $S$ such that $U^\prime\rightarrow X$ factors $U\rightarrow X$.\end{lemma}
\begin{proof}I provide the argument for finite type. Let $U/V$ be an elementary extension problem, and write $\sh O_S(V)$ as a filtered union $\sh O_S(V_i)$ of $\sh O_S$-algebras of finite type over which $U$ is defined. Then $\sh O_S(U)=\bigcup_i\sh O_S(U_i)$ is also a filtered union.

Replacing $X$ with an affine subset through which $U$ factors, finiteness implies that \[\sh O_S(X)\rightarrow \sh O_S(U)\cong\bigcup_i\sh O_S(U_i)\] factors through $\sh O_S(U_i)$ for some $i$.\end{proof}

This gives us the first part of the result by part \emph{iv}) of lemma \ref{SEP_THELEMMA}.

The second part is essentially a repetition of the same finiteness argument. Suppose that $f$ has unique solutions to all extension problems for local objects, and let $U/V$ be an extension problem. For each $p\in V$, let $V_p=\lim_{f(p)\neq 0}V_f$ denote the local scheme of $V$ at $p$, $U_p:=U\times_VV_p$. By hypothesis, there exists a formally projective modification $\tilde V_p\rightarrow V_p$ such that $\tilde V_p\rightarrow X$ under $U$.

By definition and by lemma \ref{MORP_FINITE}, this modification can be written as a finite morphism, followed by a formal completion (along a finitely presented subscheme), followed by a projective morphism. By an affine embedding, it can be replaced with one whose finite and projective parts are finitely \emph{presented}. This replacement $\tilde V_p$ has a formally projective model over $V_f$ for a coinitial family of $V_f$. Moreover, one can take a standard affine atlas \[ \tilde V_f=\bigcup_{j=1}^n\tilde V_f^j, \quad  \tilde V_p^j\cong\lim_{f(p)\neq0}\tilde V^j_f\] indexed by a finite set independent of $f$. Write $U_f^j:=U\times_V\tilde V_f^j$.

Let $X^j\subseteq X$ be an affine open subset containing the image of $\tilde V^j_p$, and let $\sh O_S\{\vec{x}\}\twoheadrightarrow\sh O_S(X^j)$ be a presentation by finitely many generators. By cocompactness, there is a homomorphism $\sh O_S\{\vec{x}\}\rightarrow \sh O_S(\tilde V^j_f)$ making the square
\[\xymatrix{ \sh O_S\{\vec x\}\ar[d]\ar@{->>}[r] & \sh O_S(X^j)\ar[d] \\
 \sh O_S(\tilde V^j_f)\ar@{^{(}->}[r] & \sh O_S(U^j_f) }\]
commutative. It follows that $\sh O_S(X^j)\rightarrow \sh O_S(\tilde V^j_f)$. By varying $j$ from $1$ to $n$, we obtain an $f$ such that $\tilde V_p\rightarrow X$ extends to $\tilde V_f$.

Finally, by varying $p$ we obtain a covering of $V$ and hence a solution $\mathrm{Sur}_{U/V}\rightarrow X$.
\end{proof}

There is also the question of whether the inclusions \ref{first}, \ref{second} \emph{detect} overconvergence, and as such whether the theory of overconvergence can be formulated entirely in terms of the category of objects of finite type (resp. presentation) over $S$. For the inclusion 
\[ \Sh\mathbf{Sch}_S^\mathrm{ltf}\rightarrow \Sh\mathbf{Sch}_S, \]
the affirmative answer is a consequence of the fact that every $\P$ modification of a scheme $V/S$ of finite type remains of finite type.

However, as remarked above in \ref{SEP_REMARKS}, the inclusion
\[ \Sh\mathbf{FSch}_S^\mathrm{ltf}\rightarrow \Sh\mathbf{FSch}_S \]
certainly does not detect overconvergence: not every extension problem $U/V$ for an overconvergent morphism $X/S$ of formal schemes can be solved by a modification representable by schemes over $V$. One way to fix this would be to weaken the finiteness condition to \emph{formally} of finite type over $S$, which is satisfied by morphisms in $f\P$. We will not pursue this approach here.

\begin{remark}Let $\tilde V$ be any overconvergent neighbourhood of $U/V$. Applying the intermediate expansion (\S\ref{Expanditures}) to the data $U,\tilde V,Z$ with $Z\subseteq \tilde V$ the closed subscheme cut out by an ideal of definition does provide an overconvergent neighbourhood $\tilde U$ of $U/V$ which is affine and, by lemma \ref{SEP_EL_FINITE},  of finite type over $V$. Indeed, this blow-up separates $Z$ from $V\setminus U$, and so the formal completion of $\tilde U$ along the strict transform $p_*^{-1}Z$ is formally projective over $V$.

Extending this to the construction $U^\mathrm{\'el}$ of def. \ref{SEP_EXP_DEG}, one can even arrange that $U^\mathrm{\'el}\rightarrow V$ is an \emph{overconvergent morphism}, as can be seen, at least in the Noetherian case, from the valuative criterion (\ref{SEP_VAL}). I omit the details. This is analogous (and, as we shall see in \cite{part2}, directly related) to the ability to refine a neighbourhood to an \emph{open} neighbourhood in general topology. 

We may therefore replace the system $f\P$ with the set of overconvergent morphisms (lemma \ref{SEP_ENLARGE}), thereby defining - somewhat circularly - a reasonable theory of overconvergent neighbourhoods in the topos $\Sh\mathbf{FSch}_S^\mathrm{ltf}$.\end{remark}

\subsection{Reduction to the underlying space}

As the usual definitions of separated and proper morphisms from algebraic geometry lead us to expect (cf. 
\cite[\textbf{I}.5.5.1.vi, \textbf{II}.5.4.6]{EGA}), extensional properties of a scheme can be understood at the level of its underlying reduced scheme. 

\begin{prop}\label{SEP_REDUCE}A morphism $f:X\rightarrow S$ in $\mathbf{FSch}$ is overconvergent if and only if its reduction $X^\mathrm{red}\rightarrow S^\mathrm{red}$ is overconvergent. 

To establish overconvergence of $f$, it is enough to exhibit solutions to extension problems $U/V$ with $V$ a reduced scheme - which may be taken of finite type over $S$, or local and essentially of finite type over $S$.\end{prop}

\begin{proof}Let us denote by
\[  \mathrm{dR}_!:\mathbf{Sch}^\mathrm{red}\leftrightarrows\mathbf{FSch}: \mathrm{dR}^* \]
the inclusion and reduction functors betweeen the category of formal schemes (having locally an ideal of definition) and its full, coreflective subcategory $\mathbf{Sch}^\mathrm{red}$ of reduced schemes. They extend to an essential geometric morphism
\[ \mathrm{dR}:\Sh\mathbf{Sch}^\mathrm{red}\rightarrow\Sh\mathbf{FSch}.\footnote{So named because $\mathrm{dR}_*\mathrm{dR}^*X$ on a scheme $X$ is often called the `de Rham sheaf' and denoted $X^\mathrm{dR}$.} \]
The theorem has two statements. The first is that if $X\in \mathbf{FSch}_S$ is overconvergent, then its reduction $\mathrm{dR}_!\mathrm{dR}^*X$ is overconvergent. The second is that $X$ is overconvergent as soon as $\mathrm{dR}^*X$ is overconvergent. Both are subsumed, in the language of \S\ref{SEP_COMPARE}, by the lemma:

\begin{lemma}The functors \[\mathrm{dR}^*:\mathbf{FSch}\rightarrow\mathbf{Sch}^\mathrm{red},\quad \mathrm{dR}_!:\Sh\mathbf{Sch}^\mathrm{red}\rightarrow\Sh\mathbf{FSch} \] detect overconvergence. Their adjoints 
\[\mathrm{dR}_!:\mathbf{Sch}^\mathrm{red}\rightarrow\mathbf{FSch}, \quad \mathrm{dR}^*:\Sh\mathbf{FSch}\rightarrow\Sh\mathbf{Sch}^\mathrm{red}  \] preserve overconvergence.\end{lemma}
\begin{proof}In light of the fact that $\mathrm{dR}^*$ is a left inverse to $\mathrm{dR}_!$, it will suffice to prove the statements for the former; cf. lemma \ref{cf}.

Let us begin with the second statement. Let $U\hookrightarrow V$ be a quasi-compact open immersion in $\mathbf{Sch}^\mathrm{red}$ and take an overconvergent neighbourhood $\tilde V$ of $U/V$ in $\mathbf{FSch}$. The reduction $\tilde V^\mathrm{red}\rightarrow \tilde V$ is embedded and contains $U$, thus in particular an overconvergent neighbourhood of $U/V$. Thus \[\mathrm{Sur}_{\mathrm{dR}_!(U/V)}\cong\mathrm{dR}_!\mathrm{Sur}_{U/V},\] and $\mathrm{dR}^*$ preserves overconvergence.

It remains only to show that $\mathrm{dR}^*$ detects overconvergence. At this point, I should clarify that, more precisely, we are investigating
\[ \mathrm{dR}^*:(\mathbf{FSch},f\!i/\P)\rightarrow(\mathbf{Sch}^\mathrm{red},i/\P). \]
By criterion \emph{iii)} of lemma \ref{SEP_THELEMMA}, we have to show that for $X\rightarrow S$ and a quasi-compact open immersion $U\hookrightarrow V$ in $\mathbf{FSch}$, the square
\[\xymatrix{ \Hom_S(\mathrm{Sur}_{U/V},-)\ar[r]\ar[d] & \Hom_{S^\mathrm{red}}(\mathrm{Sur}_{(U/V)^\mathrm{red}},X^\mathrm{red})\ar[d]\ar@{=}[r] & \Hom_S(\mathrm{Sur}_{(U/V)^\mathrm{red}},X)\ar[d] \\
 \Hom_S(U,-)\ar[r] & \Hom_{S^\mathrm{red}}(U^\mathrm{red},X^\mathrm{red}) \ar@{=}[r] & \Hom_S(U^\mathrm{red},X) }\]
is Cartesian. This is handled by lemma \ref{SEP_NILPOTENTS}.\end{proof}

\begin{lemma}[Independence of nilpotents]\label{SEP_NILPOTENTS}If $U/V$ is quasi-compact and $V_0\rightarrow V$ is a nilpotent embedding of schemes, then \[\xymatrix{\Hom_S(\mathrm{Sur}_{U/V},X )\ar[r]\ar[d] &\Hom_S(\mathrm{Sur}_{U_0/V_0},X )\ar[d] \\ \Hom_S(U,X)\ar[r] & \Hom_S(U_0,X)}\] is Cartesian. That is, solutions of an extension problem $U\rightarrow X$ correspond to those of the restricted problem $U_0\rightarrow X$.\end{lemma}
\begin{proof}Apply lemma \ref{SEP_CLOSED_SUBSET_LEMMA} to the data $U,V,Z=V_0$ to see that \[\xymatrix{ U_0\ar[r]\ar[d] & V_0\ar[d] \\ U\ar[r] & \lim\mathrm{sur}_{U/V}^{V_0} }\] is a pushout in the category of algebraic spaces. The blow-up of the intersection of a nilpotent ideal with a Cartier divisor is finite; therefore $\mathrm{Sur}_{U/V}^Z=\mathrm{sur}_{U/V}^Z$ is pro-finite over $V$. Thus $\lim\mathrm{sur}_{U/V}^Z\rightarrow V$ is integral and so $\mathrm{Sur}_{U/V}\rightarrow\lim\mathrm{sur}_{U/V}^Z$ with respect to $i/\P$.\end{proof}

Finally, to see that overconvergence can still be checked using reduced \emph{and} finite type (or local and essentially finite type) test spaces, it is enough to observe that if a formal scheme is of finite type, resp. local, then so is its reduction.\end{proof}

\begin{cor}\label{SEP_YONEDA}The Yoneda embedding $\mathbf{FSch}\rightarrow\Sh\mathbf{Sch}$ detects overconvergence.\end{cor}

Beware that strictly speaking this Yoneda functor does not detect \emph{quasi-compactness}; fortunately, this will not usually cause confusion.

\subsection{Base change $\F_1\rightarrow \Z$}For our definition of properness of $\F_1$-schemes to say anything useful about ordinary algebraic geometry, we must have two things: first, that it is preserved by the base change functor to ordinary schemes over $\Z$, and second, that over $\Z$ our definitions are equivalent to the ones in \cite{EGA} that all know and love.

\begin{prop}\label{SEP_Z}The base change $p^*:\Sh\mathbf{Sch}_{\F_1}\rightarrow\Sh\mathbf{Sch}_\Z$ preserves overconvergence.\end{prop}
\begin{proof}We will use criterion \emph{i')} of lemma \ref{SEP_THELEMMA}, which we note does not require the forgetful functor $p_!$ to preserve $\P$ or even to be defined on the whole of $\Sh\mathbf{Sch}_\Z$. Indeed, one can still define $p_!\mathrm{Sur}_{U/V}$ for affine $U/V$ as a left pro-adjoint
\[ \Hom(p_!\mathrm{Sur}_{U/V},-):=\Hom(\mathrm{Sur}_{U/V},p^*(-)) \]
to $p^*$, whence our objective is simply to find a section to the natural map
\[\Hom(\mathrm{Sur}_{U/V},p^*(-))\rightarrow \Hom(\mathrm{Sur}_{p_!U/p_!V},-).\]
 
Let $U/V$ be an elementary extension problem for schemes over $\Z$, and let $p_!U/p_!V$ be the open immersion of affine $\F_1$-schemes obtained by forgetting the additive structure. (The assumption that $U/V$ are affine is essential here for $p_!$ to be defined.) By theorem \ref{MORP_PROJECTIVE=BLOWUP}, after refinement, any overconvergent neighbourhood $\tilde V_{\F_1}$ of $p_!U/p_!V$ is a blow-up along some finitely generated ideal $T\trianglelefteq\sh O(p_!V)$.

Let $\Z T\trianglelefteq\sh O(V)$ denote the additive closure of $T$, $\tilde V_\Z$ the blow-up of $V$ along $\Z T$. Since $T$ generates $\Z T$, the morphism of Rees algebras $\bigoplus_iT^i\tens_{\F_1}\Z\rightarrow\bigoplus_i\Z T^i$ induces a morphism $p_!\tilde V_\Z\rightarrow\tilde V_{\F_1}$ over $p_!V$. 
\[\xymatrix{ &\tilde V_\Z\ar[d]^{\mathrm{Bl}_{\Z T}\phantom{bl}}="name" & & p_!\tilde V_\Z\ar[r] &\tilde V_{\F_1}\ar[d]^{\mathrm{Bl}_T}  \\ 
U\ar[r]\ar[ur] & V  && p_!U\ar[u]^{\phantom{blah}}="nume" \ar[r] & p_!V  \ar@{..>}^-{p_!} "name";"nume" }\]
Thus $p_!\mathrm{Sur}_{U/V}\rightarrow\mathrm{Sur}_{p_!U/p_!V}$ over $p_!V$.\end{proof}

It is not possible to repeat this argument directly for formal schemes due to the absence of the forgetful functor $p_!$ even for affine objects (as remarked at the end of \S\ref{RIG_PROD}). However, we can to deduce its conclusion from the fact that overconvergence depends only on the underlying reduced scheme.

\begin{cor}\label{SEP_FZ}The base change $p^*:\mathbf{FSch}_{\F_1}\rightarrow\mathbf{FSch}_\Z$ preserves overconvergence.\end{cor}
\begin{proof}Follows from \ref{SEP_YONED}, \ref{SEP_YONEDA}, and \ref{SEP_Z}.\end{proof}

\begin{prop}\label{SEP_ME=EGA}A finite type, separated morphism in $\mathbf{FSch}_\Z$ is proper if and only if it is universally closed.\end{prop}
\begin{proof}The reductions of corollary \ref{SEP_REDUCE} allow us to assume that $X/S$ are (reduced) schemes. Suppose that $X/S$ is proper. The Chow property then implies that $X$ is dominated by a projective $S$-scheme, and is therefore universally closed by \cite[\textbf{II}.5.2.3.\emph{ii)}]{EGA}.

The difficulty lies in showing that if $X\rightarrow S$ is proper in the sense of \cite[\textbf{II}.5]{EGA}, then it is proper in the sense of definition \ref{SEP_DEF_PROPER}. Let $U/V$ be a quasi-compact extension problem, and let $\tilde V\rightarrow X$ be the canonical extension (def. \ref{SEP_ELEMENTARY}). Then $\tilde V\rightarrow V$ is also of finite type, separated, and universally closed. By \cite[\href{http://stacks.math.columbia.edu/tag/081T}{081T}]{stacksproject}, there is a $V\setminus U$-admissible blow-up of $V$ that dominates $\tilde V$. Therefore $X\rightarrow S$ is proper.\end{proof}

\subsection{Embeddings are proper}

Recall (def. \ref{MORP_IMMERSION}) that an open immersion followed by an embedding is called an \emph{immersion}. Every immersion is separated. Since open immersions can be understood, by definition, at the level of point-set topology, so too can the difference between immersions and embeddings. In particular:

\begin{lemma}\label{the difference}A surjective immersion is an embedding.\end{lemma}

\begin{lemma}\label{SEP_EMBED_PROPER}Let $X\hookrightarrow S$ be an immersion, $U\hookrightarrow V$ an affine open immersion of $S$-schemes. Then $\Hom_S(\mathrm{cl}(U/V),X)\tilde\rightarrow \Hom_S(\mathrm{Sur}_{U/V},X)$.\end{lemma}
\begin{proof}Without loss of generality, assume $V=\mathrm{cl}(U/V)$. Let $\tilde V$ be an overconvergent neighbourhood of $U/V$, $\tilde V\rightarrow X$ an extension.

The base change $X\times_SV\rightarrow V$ is an immersion containing $U$. It is also surjective, since $\tilde V\rightarrow V$ is surjective. Therefore it is an isomorphism by lemma \ref{the difference}.\end{proof}

It follows, more or less tautologically:

\begin{prop}\label{SEP_EMBEDDING}An immersion is proper if and only if it is an embedding.

A formal immersion is $f\P$-proper if and only if it is a formal embedding.\end{prop}
\begin{proof}Let $X\hookrightarrow S$ be a formal embedding, and let $U/V$ be an elementary extension problem. Then in particular, $X/S$ has affine diagonal, and so $U\rightarrow X$ is affine. Since taking the formally embedded image is order-preserving, \[\mathrm{cl}(V/S)=\mathrm{cl}(U/S)=\mathrm{cl}(\mathrm{cl}(U/X)/S)\subseteq X.\] Thus $V\rightarrow X$ and $X$ is $f\P$-proper.

Conversely, let $X\hookrightarrow S$ be a formally proper formal immersion, and let $U\rightarrow S$ be an affine formal immersion factoring through $X$. Let $V=\mathrm{cl}(U/S)$ be the formally embedded closure; then $U$ is open in $V$. By the definition of propriety, there is a projective modification of $V$ that factors through $X$. Since $U/V$ is dense, such a modification is necessarily surjective. Therefore $V\subseteq X$.

The first statement follows immediately from the second and the definitions.\end{proof}

We will return to a discussion on this theme in \cite[\S4.5]{part2}.

\subsection{Alternative characterisations}\label{SEP_SEP}
In this section, we gather a few alternative characterisations of separated, proper, and overconvergent morphisms couched in more traditional terms.

\subsubsection{Diagonal criterion for separation}
\begin{cor}[of prop. \ref{SEP_EMBEDDING}]\label{SEP_CRITERIA}Let $X\rightarrow S$ be a morphism of formal schemes. The following are equivalent:
\begin{enumerate}\item $X/S$ is separated;
\item the diagonal $X\rightarrow X\times_SX$ is an embedding.
\end{enumerate}\end{cor}
\begin{proof}Indeed, the proof of \cite[\href{http://stacks.math.columbia.edu/tag/01KJ}{01KJ}]{stacksproject} runs without modification in the $\F_1$ setting to establish that the diagonal is always an immersion.\end{proof}

\begin{remark}Of course, we cannot ask that the diagonal be closed: this fails even for the affine line over $\F_1$, cf. \cite[\S6.5.20]{Durov}. 

The diagonal of a separated morphism can also fail to be affine, as the following pathological example shows: Let $V=\Spec(\F_1\times\F_1)$, $U\subset V$ the complement of the closed point. Then $U\hookrightarrow V$ is projective, but not affine. Thus the scheme obtained by glueing two copies of $V$ along $U$ is separated with non-affine diagonal.\end{remark}

\subsubsection{Chow criterion for propriety}
The Chow lemma depends on the fact that $\mathbf{FSch}^\mathrm{ltf}_S$ has a site consisting of open subobjects of objects $\P$ over the base.

\begin{thm}[Strong Chow lemma]\label{SEP_CHOW}A finite type morphism $X\rightarrow S$ is proper if and only if it is quasi-separated and, for any $U\subseteq X$ quasi-projective over $S$, there locally on $S$ exists an extension 
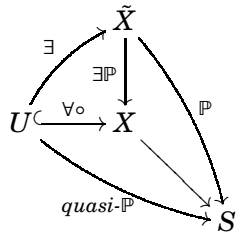
\[\xymatrix{ & \tilde X\ar[d]_{\exists \P}\ar@/^/[ddr]^{\P} \\ U\ar@{^{(}->}[r]^{\forall\circ}\ar@/^/[ur]^\exists\ar@/_/[drr]_-{\text{quasi-}\P} & X\ar[dr] \\ && S}\] of the inclusion of $U$ in $X$ to a $\P$ morphism $\tilde X\rightarrow X$ such that $\tilde X$ is $\P$ over $S$, and moreover this property persists after any base change.\end{thm}
\begin{proof}Indeed, this Chow property states that for any quasi-projective $U$ and quasi-compact open immersion $U\hookrightarrow V$, the canonical solution $X\times V\rightarrow V$ can be dominated by a $\P$ morphism with a section over $U$ and is therefore an overconvergent neighbourhood.\end{proof}



\subsubsection{Local criterion for overconvergence}Overconvergence is a bit more difficult to characterise in classical terms, since there is little literature on the subject.

\begin{prop}[Overconvergence is local propriety]\label{SEP_SUR_MODELS}Let $X\rightarrow S$ be a quasi-separated morphism of formal schemes. The following are equivalent:
\begin{enumerate}\item $f$ is overconvergent;
\item every proper formal $X$-scheme qcqs over $S$ is proper over $S$;
\item every formally embedded formal subscheme of $X$ qcqs over $S$ is proper over $S$.
\end{enumerate}\end{prop}
\begin{proof}The implications \emph{i)}$\Rightarrow$\emph{ii)}$\Rightarrow$\emph{iii)} being clear, suppose \emph{iii)}, and let $U/V$ be an elementary extension problem. Since $X$ is quasi-separated over $S$, $U\rightarrow X$ is quasi-compact. There is therefore a formally embedded subscheme $Z$ of $X$ qcqs over $S$ that contains the image of $U$. By hypothesis, it is proper and so $\mathrm{Sur}_{U/V}\rightarrow Z\hookrightarrow X$.\end{proof}

A morphism between locally integral over Noetherian formal schemes is always quasi-separated. Hence, this criterion for overconvergence applies in most cases one encounters in practice in formal geometry.

\subsubsection{Valuative criteria}
Here, we will establish a version of the \emph{valuative criteria} of \cite[\textbf{II}.7]{EGA} for Noetherian formal $\F_1$-schemes.

\begin{lemma}Let $V$ be an affine $\F_1$-scheme, $U=\Spec A\hookrightarrow V$ a dense affine open immersion. Let $U_{\F_1}=\Spec A/(A^\times=1)$, $V_{\F_1}$ the localisation. Then \[\xymatrix{U_{\F_1}\ar[r]\ar[d] & V_{\F_1}\ar[d] \\ U\ar[r] & V}\] is a pushout in the category of algebraic spaces.\end{lemma}
\begin{proof}Since the vertical maps are homeomorphisms, it will be enough to treat the affine case; that is, we must check that the square \[\xymatrix{A\ar[r]\ar[d] & A[f^{-1}]\ar[d] \\ A/A^\times\ar[r] & A/A^\times[f^{-1}]}\] is Cartesian. This follows from the fact that the inclusion $A\subseteq A[f^{-1}]$ is preserved by the group action of $A^\times$.
\end{proof}

\begin{thm}[Valuative criterion]\label{SEP_VAL}Let $S$ be a locally Noetherian formal $\F_1$-scheme, $f:X\rightarrow S$ locally of finite type and paracompact. Then $f$ is overconvergent if and only if every commuting square
\[\xymatrix{  \A^1_{\F_1}\setminus 0\ar[r]\ar[d] & \A^1_{\F_1}\ar[d] \\ X\ar[r] & S }\]
admits a unique lift $\A^1_{\F_1}\rightarrow X$.

If $S$ is more generally locally integral/Noetherian, then the same holds with $\A^1_{\F_1}$ replaced with $\Spec\F_1[z^\Q]$.\end{thm}
\begin{proof}Indeed, by corollary \ref{INT_DECOMP_DETECT}, we may assume that $U$ and $V$ are quasi-integral and, replacing $X$ with the closure of the image of $U$, that $X$ is quasi-integral and Noetherian. The result then follows from proposition \ref{FAN_SUR}.\end{proof}

I had initially hoped to prove this result directly by an algorithmic construction based on proposition \ref{SEP_CLOSED_SUBSET}, but my attempts were ultimately confounded, and so the proof given rests on the classification of $\F_1$-schemes and criterion for propriety found in \S\ref{FAN}. In particular, the methods manifestly do not apply over $\Z$. We do not use this criterion except in the proof of theorem \ref{SEP_GROTHENDIECK_OVERCONVERGENT}, which is independent of the rest of the paper.

\begin{remark}This does not imply that overconvergence can be detected by \emph{non-boundary} morphisms from $\A^1_{\F_1}$. Indeed, the fan in $\Q^2$ whose non-zero cones are every rational ray provides a counterexample to that statement.

The valuative criterion also fails without the paracompactness assumption: indeed, the construction $U^\mathrm{\'el}$ (def. \ref{SEP_EXP_DEG}) applied to the affine plane provides a counterexample.\end{remark}

\subsubsection{Overconvergence via images}\label{SEP_IMAGES}We have already seen (prop. \ref{SEP_ME=EGA}) that proper morphisms over $\Z$ can be characterised, as in \cite{EGA}, by the behaviour of closed subsets under images. Over $\F_1$, it is of course not true that proper morphisms are universally closed, and since embeddings are not characterised by their underlying sets, one must be a bit more careful when taking the image.

In this section I present a possible analogue to the approach of Grothendieck.

\begin{defn}Write $\hat{\lie Z}(X/S)$ for the poset of formally embedded subschemes of $X$ qcqs over $S$.
Let us call a morphism $f:X\rightarrow S$ \emph{grounded} if for every $Z\in\hat{\lie Z}(X/S)$, \[Z\rightarrow\mathrm{cl}(Z/S)\] is surjective on points. It is \emph{universally grounded} if it is grounded after any base change.\end{defn}

The intuition behind this definition is that if it were possible to define the formally \emph{immersed} image of morphisms, then by lemma \ref{the difference}, the difference between this and the embedded image would be detected by the underlying set. Indeed, one can show that grounding implies that the square
\[\xymatrix{\hat{\lie Z}(X/S) \ar[d]_{\mathrm{cl}(-/S)}\ar[r]&  \mathrm{Imm}(X/S)\ar[d]^{f_!}   \\\Pro\hat{\lie Z}(S)\ar[r] & \Pro\mathrm{Imm}(S) }\]
commutes, where $\mathrm{Imm}(X/S)$ denotes the poset of formally immersed subschemes of $X$ qcqs over $S$, and $f_!$ the left pro-adjoint to pullback.

With this definition, it is easy to establish an analogue of \cite[\textbf{II}.5.2.3.\emph{ii)}]{EGA} - which would be useful to have for our definition of propriety via extension problems:

\begin{lemma}\label{phlegm}Let $X\rightarrow Y\rightarrow S$ be a pair of morphisms such that
\begin{enumerate}\item $X\rightarrow S$ is universally grounded;
\item $X\rightarrow Y$ is \emph{strongly surjective}, meaning that for any embedded subscheme $Z\hookrightarrow Y$, \[ Z=\mathrm{cl}(Z\times_YX/Y).\] \end{enumerate} Then $Y\rightarrow S$ is universally grounded.\end{lemma}

\begin{eg}Note that surjectivity of $X/Y$ is clearly not enough for this lemma to run in the $\F_1$ case, as the example $\{1\}\rightarrow \A^1_{\F_1}\setminus\{0\}\rightarrow\A^1_{\F_1}$ demonstrates.\end{eg}

\begin{thm}\label{SEP_GROTHENDIECK_OVERCONVERGENT}Let $S$ be locally integral/Noetherian, $f:X\rightarrow S$ separated, of finite type, and universally grounded. Then $f$ is proper.\end{thm}
\begin{proof}We apply the valuative criterion. For ease of notation, we treat the Noetherian case, though the general case follows from the same argument. Consider a commuting square
\[ \xymatrix{ \A_{\F_1}^1\setminus\{0\}\ar[r]\ar[d] & \A_{\F_1}^1\ar[d] \\ X\ar[r]&S } \]
and replace $X$ resp. $S$ with the embedded closures of $\A_{\F_1}^1\setminus\{0\}$ resp. $\A_{\F_1}^1$; this is possible because every open subset of $\A_{\F_1}^1$ is affine, and so the morphisms to $X$ and $S$ are also affine. By hypothesis, $X\rightarrow S$ is now surjective.

We therefore obtain a dual square \[ \xymatrix{  \sh O(S) \ar[r]\ar[d] & \sh O(X) \ar[d] \\ \F_1[t] \ar[r] & \F_1[t^{\pm1}]  } \] in which all arrows are injective. Thus $\sh O(X)\setminus 0$ is a submonoid of $\Z$ and $\sh O(S)\setminus 0$ of $\N$. By surjectivity of $f$, $\sh O(X)\setminus 0$ must in fact be contained in $\N$. Therefore $\A_{\F_1}^1\rightarrow X$.\end{proof}

\section{Integrality and normalisation}\label{INT}

In this section we discuss some absolute properties of $\F_1$-scheme that bring us towards the realm of toric geometry:
\begin{itemize}
\item quasi-integral;
\item integral;
\item normal.
\end{itemize}
The \emph{quasi-integral} condition is an intermediate notion that does not appear for ordinary rings; it exists because a monoid can fail to be cancellative without having zero-divisors. This is another manifestation of the fact that embeddings can fail to be closed. The geometric provenance of such monoids appears to be quite mysterious.

On the other hand, with \emph{non-boundary morphisms} (morphisms with no kernel) one can define an interesting subcategory of quasi-integral $\F_1$-schemes that enjoys stability under fibre products and localisations; no analogous construction can work over $\Z$, or with the more restrictive hypotheses of integrality or normality.

\subsection{Quasi-integral}\label{QUASI_INT}

\begin{defn}An $\F_1$-algebra $A$ is said to be \emph{quasi-integral} if the zero ideal is prime, or equivalently, if $A\setminus 0$ is a submonoid. A homomorphism of $\F_1$-algebras is said to be \emph{non-boundary} if its kernel is zero.

An $\F_1$-scheme is quasi-integral if the stalks of its structure sheaf are quasi-integral $\F_1$-algebras. A morphism $f:X\rightarrow Y$ of $\F_1$-schemes is non-boundary if $f^{-1}\sh O_Y\rightarrow\sh O_X$ has zero kernel. The category of quasi-integral schemes and non-boundary homomorphisms is denoted $\mathbf{Sch}_{\F_1}^\mathrm{qi/nb}$.\end{defn}

The subcategory $\mathrm{Alg}_{\F_1}^\mathrm{qi/nb}$ of $\mathrm{Alg}_{\F_1}$ consisting of quasi-integral algebras and non-boundary morphisms is exactly the image of the category of ordinary monoids under adjoining zero: the restriction of this functor
\[  \F_1[z^{-}]:\mathbf{Mon} \rightarrow \mathrm{Alg}^\mathrm{qi/nb}_{\F_1} \]
is an equivalence with inverse $A\mapsto A\setminus 0$. It follows that any colimit of quasi-integral $\F_1$-algebras along non-boundary homomorphisms is quasi-integral.

In particular, the inclusion of opposite categories is left exact, and so generates a subtopos $\Sh\mathbf{Sch}^\mathrm{qi/nb}_{\F_1}$ under colimits. The inclusion is the pullback functor for an essential geometric morphism
\[ \Sh\mathbf{Sch}_{\F_1}\rightarrow\Sh\mathbf{Sch}^\mathrm{qi/nb}_{\F_1} \]
whose pullback preserves open immersions. The topos on the right is nothing more than the presheaf topos on the opposite category to $\mathrm{Alg}_{\F_1}^\mathrm{qi/nb}$. The associated embedding of the category of locally representable objects into $\mathbf{Sch}_{\F_1}$ identifies it with $\mathbf{Sch}^\mathrm{qi/nb}_{\F_1}$, since being quasi-integral or non-boundary is a local property.

\begin{remark}\label{INT_EMPTYSET}Note that an affine scheme is quasi-integral if and only if it is the spectrum of a quasi-integral $\F_1$-algebra, \emph{except for the empty set}, which is locally representable, but not representable, in $\mathrm{Alg}_{\F_1}^\mathrm{qi/nb}$.\end{remark}

\begin{lemma}A scheme is quasi-integral if and only if every open subset is non-boundary.\end{lemma}
\begin{proof}It suffices to check the affine case. The kernel of a localisation $A\rightarrow A[f^{-1}]$ is the annihilator of $f$. Quasi-integrality of $A$ means that this is zero whenever $f$ is non-zero. If $f=0$, then this is dual to the inclusion of $\emptyset\in\mathbf{Sch}_{\F_1}^\mathrm{qi/nb}$.\end{proof}

In other words, the small site of a quasi-integral scheme does not depend on whether we consider it as an object of $\Sh\mathbf{Sch}_{\F_1}^\mathrm{qi/nb}$ or of $\Sh\mathbf{Sch}_{\F_1}$.

\paragraph{Detecting overconvergence}

\begin{lemma}\label{INT_PROJ}\begin{enumerate}\item The total space of a blow-up of a quasi-integral scheme is quasi-integral, and the projection is non-boundary.
\item Let $U/V$ be an open immersion of quasi-integral schemes. The relative normalisation $\nu_{V\setminus U}V$ is quasi-integral, and the projection is non-boundary.\end{enumerate}\end{lemma}

It follows that for $U\hookrightarrow V$ in $\mathbf{Sch}_{\F_1}^\mathrm{qi/nb}$, the overconvergent germ $\mathrm{Sur}_{U/V}$, computed in $\Sh\mathbf{Sch}_{\F_1}$, is also a pro-object of $\Sh\mathbf{Sch}_{\F_1}^\mathrm{qi/nb}$.

\begin{prop}The pullback $\Sh\mathbf{Sch}_{\F_1}^\mathrm{qi/nb}\hookrightarrow \Sh\mathbf{Sch}_{\F_1}$ preserves overconvergence.\end{prop}
\begin{proof}Let $U/V$ be an extension problem for a non-boundary morphism $X\rightarrow S$ of quasi-integral schemes. We will produce an extension problem $U_0/V_0$ in $\mathbf{Sch}^\mathrm{qi/nb}_{\F_1}$ factoring $U/V$ and such that $U\cong U_0\times_{V_0}V$. This will get us criterion \emph{iv)} of lemma \ref{SEP_COMPARE}.

Suppose that $V$ is affine and that $U\hookrightarrow V$ is affine and dense, and replace $X$ with a quasi-integral affine scheme through which the morphism from $U$ factors. Then we are in the situation of a commuting square
\[\xymatrix{ \sh O_S(V)[f^{-1}] & \sh O_S(V)\ar@{_{(}->}[l] \\ \sh O_S(X)\ar[u] & \sh O_S\ar[l]\ar[u] }\]
with the $\F_1$-algebras in the bottom row integral and $f\in\sh O_S(V)$ a cancellable element. 

Let us define $\sh O_S(U_0)$ to be the subalgebra of $\sh O_S(U)$ generated by $\sh O_S(X)$ and $f^{\pm 1}$, and $\sh O_S(V_0)=\sh O_S(U_0)\cap\sh O_S(V)$. These subalgebras are quasi-integral because $f$ is a non-zero-divisor and $\sh O_S(X)$ is quasi-integral. The homomorphisms between them are injective, and so the dual square
\[\xymatrix{ U_0\ar[r]\ar[d] & V_0\ar[d] \\ X\ar[r] & S}\]
now lives in $\mathbf{Sch}_{\F_1}^\mathrm{qi/nb}$. By lemma \ref{blargunsfish}, $U_0\hookrightarrow V_0$ is an open immersion.
\end{proof}

\subsection{Irreducible components}

If $X$ is any $\F_1$-scheme and $I\trianglelefteq\sh O_X$ an ideal sheaf, then the closed subscheme cut out by $I$ is quasi-integral if and only if $I$ is prime. This defines a one-to-one, inclusion-reversing correspondence
\[ \{\text{quasi-integral closed subschemes of $X$}\} \quad \leftrightarrow\quad \{\text{primes of }\sh O_X\}. \]
Now suppose the underlying topological space of $X$ is Noetherian and let $X=\bigcup_{i=1}^kX_i$ be its decomposition into irreducible components. Each of these components is the closure of a unique minimal prime $\lie p_i\trianglelefteq \sh O_X$. Equipped with their reduced scheme structure they are quasi-integral closed subschemes. We will always consider the irreducible components of $X$ - when they exist - to carry this scheme structure.

\paragraph{Categorical decomposition into irreducible components}

Suppose that $\{T_i\trianglelefteq A\}_{i\in J}$ is a finite family of ideals in an algebra $A$ (over $\F_1$ or $\Z$). Let's write
\[ T_I:=\sum_{i\in I}T_I \]
for a subset $I\subset J$, and $T=\bigcap_{i\in J}T_j$. The family $A/T_I$ is then filtered.

Over $\Z$, it is possible to show that
\[ A/T\tilde\rightarrow \lim_IA/T_I \]
in the category of rings. In other words, a finite union of closed subschemes of a fixed scheme is actually a \emph{colimit} in $\mathbf{Sch}_\Z$. In particular, a scheme over $\Z$ whose underyling topological space is Noetherian has a \emph{categorical} decomposition into irreducible components. 

Unfortunately, for schemes relative to the category of monoids this cannot be true.

\begin{eg}[More stupid limits of monoids]Let $X$ be the spectrum of $\F_1[x,y]/xy$. The irreducible components of this thing are the axes, intersecting in the origin, so we are led to consider the square
\[\xymatrix{ \F_1[x,y]/xy\ar[r]^{x=0}\ar[d]_{y=0} & \F_1[y]\ar[d] \\ \F_1[x]\ar[r] & \F_1}\]
Unfortunately, this square is \emph{not} Cartesian - in fact, the pullback monoid has infinitely many transcendents.

Taking the limit instead in the category of monads improves matters slightly: we get an additional relation $w=x+y$. The prime spectrum of this monad has the same underlying space as $X$, although there is still no morphism to it from $X$. There is therefore some hope that certain schemes relative to the category of monads (or blueprints, or sesquiads) will admit a categorical decomposition into irreducible components.\end{eg}

In any case, we are able to get a weaker result in the present framework which is sufficient for understanding overconvergence:

\begin{prop}\label{INT_DECOMP}Let $U\hookrightarrow V$ be an affine, dense open immersion. Let $\{V_i\subseteq V\}_{i\in J}$ be a finite family of closed subschemes with defining ideals $T_i\trianglelefteq \sh O_V$, and suppose that $U\cap V_i$ remains scheme-theoretically dense in $V_i$ for each $i$. 

Let $\tilde V$ be the subscheme cut out by $\bigcap_{i\in J}T_i$. Then \[\xymatrix{\colim_I(U\cap V_I) \ar[r]\ar[d] &\colim_IV_I \ar[d] \\  U\cap \tilde V\ar[r] & \tilde V}\] is a pushout in the category of algebraic spaces.\end{prop}
\begin{proof}Let $X$ be an algebraic space, and take a morphism to $X$ from the pushout. After replacing $V$ with an affine cover, we will need to produce an affine open subset of $X$ through which $U\cap \tilde V$ and all the $V_i$ factor. We may then reduce to the affine case. That the resulting diagram is a pushout in the category of affine schemes is the content of lemma \ref{INT_DECOMP_AFF}.

Since affine $\F_1$-schemes are local, $V$ has a unique closed point, contained in each $V_i$. It will be enough to show that the closed points of all the $V_i$ determine the same point of $X$. But this follows from the fact that the strata $V_I$ form a cofiltered poset.\end{proof}

\begin{lemma}\label{INT_DECOMP_AFF}Let $A\hookrightarrow A[f^{-1}]$ be a localisation at a cancellable element $f$. Suppose that $f\not\in\bigvee_{i\in J}T_i$. Then \[\xymatrix{ A/\bigcap_{i\in J}T_i \ar[r]\ar[d] &A[f^{-1}]/\bigcap_{i\in J}T_i \ar[d] \\ \prod_iA/T_i \ar[r] & \prod_iA[f^{-1}]/T_i }\] is a pullback.\end{lemma}
\begin{proof}The assumptions ensure that all the arrows are injective, so we are investigating an intersection of sets. Let $g\in A[f^{-1}]/\bigcap_iT_i$ be non-zero. Then $g$ is also non-zero in some $A[f^{-1}]/T_i$. If its image there is contained in $A/T_i$, then via the set-theoretic section
\[ (A/T_i)\setminus 0\rightarrow (A/\cap_iT_i)\setminus T_i \]
it is also contained in $A/\bigcap_{i\in J}T_i$.\end{proof}

\begin{cor}Let $X$ be a reduced integral/Noetherian $\F_1$-scheme, $U\subseteq X$ a dense open subset. Then $X$ is the colimit in the category of schemes of $U$ and its quasi-integral components.\end{cor}
\begin{proof}By the proposition, it is enough to show that the intersection of the minimal primes of $\sh O_X$ is its nilradical, which follows from the usual argument in commutative algebra.\end{proof}

\begin{cor}\label{INT_DECOMP_DETECT}Let $S$ be locally integral-over-Noetherian, $X\rightarrow S$ locally of finite type. Then $X/S$ is overconvergent if and only if every extension problem with quasi-integral test spaces has a unique solution.\end{cor}
\begin{proof}Let $U/V$ be an elementary extension problem \ref{SEP_ELEMENTARY}. By propositions \ref{SEP_REDUCE} and \ref{SEP_FP}, we may assume that $V$ is a reduced scheme of finite type over $S$. Then by lemma \ref{FIN_CRUX} the underlying topological spaces of $S,U,V$, and $X$ are (locally in the latter case) Noetherian. 

Let us write $V=\bigcup_iV_i$ as a union of irreducible quasi-integral closed subschemes.Write $U_i:=U\times_VV_i$, and let $\mathrm{Sur}_{\coprod_i(U_i/V_i)}\rightarrow X$. By theorem \ref{MORP_PROJECTIVE=BLOWUP}, we may assume that the morphism is defined on a finite type blow-up $\coprod_i\tilde V_i\rightarrow \coprod_iV_i$ along some ideals $T_i\trianglelefteq\sh O_{V_i}$. 

Considering these as ideals on $V$ via the closed embeddings $V_i\hookrightarrow V$, let $T=\prod_iT_i$ and let $\tilde V$ be the ($V\setminus U$-admissible) blow-up. By lemma \ref{INT_DECOMP}, we obtain a unique morphism to $X$ from the union of the closed subschemes $\tilde V_i\times_V\tilde V$ of $\tilde V$, extending the given $\tilde V_i\rightarrow X$. Since the union is finite, the inclusion of the resulting subscheme is a closed embedding, hence in particular projective. This therefore determines a unique map $\mathrm{Sur}_{U/V}\rightarrow X$ descending from $\mathrm{Sur}_{\coprod_i(U_i/V_i)}\rightarrow X$.\end{proof}

\subsection{Integrality and normalisation}

Let $A$ be quasi-integral. The \emph{field of fractions} of $A$ is the localisation
\[ K_A:=\F_1[z^{(A\setminus0)\tens\Z}]. \]
The fraction fields glue together to yield a sheaf $K$ of $\F_1$-fields on $\mathbf{Sch}^\mathrm{aff/qi/nb}$.

If $X$ is a connected, quasi-integral $\F_1$-scheme, then non-empty open subsets are topologically dense, and so the restriction $K_X$ of $K$ to (the small site of) $X$ is constant. In this case, we will confuse $K_X$ with its algebra of sections, which is called the \emph{function field} of $X$. Its spectrum is the \emph{generic point} of $X$. The inclusion of the generic point is an affine morphism.

This allows us to define absolute versions of the notions introduced for pairs in \S\ref{MORP_NORM}.

\begin{defn}[Integrality]\label{INT_RUGGET}A quasi-integral $\F_1$-algebra $A$ is \emph{integral} if every non-zero principal divisor is Cartier; that is, if $A\setminus0$ is a cancellative monoid. 

An integral $\F_1$-algebra $A$ is \emph{normal} if $(K_A;A)$ is relatively normal (def. \ref{MORP_ADMIT_INT}); that is, if $A\setminus0$ is saturated in $(A\setminus 0)\tens\Z$.

A quasi-integral $\F_1$-scheme $X$ is said to be \emph{integral}, resp. \emph{normal}, if the stalks of $\sh O_X$ are integral, resp. normal, $\F_1$-algebras.\end{defn}

One associates to any quasi-integral $\F_1$-scheme an \emph{underlying integral subscheme}, which for connected $X$ is calculated as the affine embedding
\[ X^\mathrm{i}\cong\mathrm{cl}(\Spec K_X/X)\hookrightarrow X. \]
In particular, a quasi-integral scheme is integral if and only if every inhabited open subset is scheme-theoretically dense.

Similarly, an integral scheme may be replaced with its \emph{normalisation} $\nu X\rightarrow X$, which on connected components is the same as the relative normalisation of the pair $(\Spec K_X,X)$. Note that the normalisation and the inclusion of the underlying integral subscheme are \emph{ homeomorphisms} in the $\F_1$-setting. They are integral/projective.

These constructions yield right adjoints to the inclusions
\[ \mathbf{Sch}_{\F_1}^\mathrm{n/nb}\hookrightarrow\mathbf{Sch}_{\F_1}^\mathrm{i/nb}\hookrightarrow\mathbf{Sch}_{\F_1}^\mathrm{qi/nb} \]
of the full subcategories of $\mathbf{Sch}_{\F_1}^\mathrm{qi/nb}$ whose objects are normal, respectively integral schemes.

\begin{remark}The categories of integral and normal $\F_1$-schemes are not closed in $\mathbf{Sch}_{\F_1}^\mathrm{qi/nb}$ under fibre products, as the co-Cartesian square
\[\xymatrix{\F_1[x,y]\ar[r]^{y/x=w}\ar[d]_{y=x} & \F_1[x,w]\ar[d] \\ \F_1[x]\ar[r] & \F_1[x,w]/(xw=w)}\]
easily demonstrates. (This example is an affine patch of the pullback of the diagonal of $\A^2$ along the blow-up at $0$).\end{remark}

\begin{lemma}Passage to the underlying integral scheme $\mathbf{Sch}_{\F_1}^\mathrm{qi/nb}\rightarrow\mathbf{Sch}_{\F_1}^\mathrm{i/nb}$ detects and preserves overconvergence.\end{lemma}

\begin{lemma}Normalisation $\mathbf{Sch}_{\F_1}^\mathrm{i/nb}\rightarrow\mathbf{Sch}_{\F_1}^\mathrm{n/nb}$ detects and preserves $i/\P$-overconvergence.\end{lemma}
\begin{proof}We prove both lemmas at once. The preservation statements follow from the fact that normalisation is of class $i/\P$, and so the overconvergent germ $\mathrm{Sur}_{U/V}$ does not depend on in which category it is computed.

Now let $U/V$ be an affine, dense extension problem, and let $\nu U$, resp $\nu V$, be the normalisation of $U$, resp. $V$. Define $\tilde V$ via the pullback square
\[\xymatrix{ \sh O(\tilde V)\ar[r]\ar[d] & \sh O(U)\ar[d] \\ \sh O(\nu V)\ar[r] & \sh O(\nu U). }\]
The dual square is a pushout in $\mathbf{Sch}_{\F_1}$ because $\nu V\rightarrow \tilde V$ is a homeomorphism. Moreover, $\tilde V\rightarrow V$ is integral/projective. This shows criterion \emph{iv)} of lemma \ref{SEP_THELEMMA}.\end{proof}

\section{From schemes to fans}\label{FAN}

The starting point of the classification theorem is the observation, first codified, to the best of my knowledge, in \cite{Deitmartoric}, that integral $\F_1$-schemes are essentially the same as toric varieties: they can be packaged in terms of a \emph{fan} in a rational vector space.

Most of the statements in this section are well-known basic properties of toric geometry. I restate them here with an eye towards generalisation to the formal and rigid analytic cases.
In this and the next section, we state relative properties of morphisms as an absolute property of an $\F_1$-scheme $X$ if it holds for the structural morphism $X\rightarrow\Spec\F_1$.


The main thrust will be to explain the following statement:

\begin{thm}\label{FAN_THM}The construction
\[ X \mapsto (K_X^\times,\Sigma_X) \]
sets up an equivalence between the category of normal, connected, separated $\F_1$-schemes with enough jets, with non-boundary morphisms, and the category of pairs consisting of an Abelian group and a fan in its $\Q$-dual. Moreover, we have the following equivalences:
\begin{enumerate}\item $X$ is quasi-compact if and only if $\Sigma_X$ is finite;
\item $X$ is locally integral/Noetherian if and only if the cones of $\Sigma_X$ are rational polyhedral;
\item $X$ is locally Noetherian if and only if in addition to {ii)}, $\Sigma_X$ is spanned by its $\Z$-points;
\item $X$ is proper if and only if in addition to {iii)}, $\Sigma_X(\R)=N(\R)$.\end{enumerate}
The integral closed subschemes of $X$ are in natural, inclusion-reversing correspondence with the cones of $\Sigma_X$, with the fan data of a closed subscheme with cone $\sigma$ given by $(K_X^\times/\sigma,\Sigma_X/\sigma)$.\end{thm}

The notation will be defined shortly, though its meaning should be clear to students of toric geometry (see \cite{Fultor}).
A version of this, missing the separation hypothesis, appeared as \cite[thm. 4.1]{Deitmartoric}.

\subsection{Diagonalisable group action}\label{FAN_DIAG}

The category of monoids, and hence $\mathrm{Alg}_{\F_1}^\mathrm{qi/nb}$, is semiadditive, that is, finite products are finite coproducts. It follows that any quasi-integral $\F_1$-algebra $A$ is naturally a \emph{bialgebra}, with comultiplication given by the diagonal
\[ A\rightarrow A\tens_{\F_1}A,\quad f\mapsto f\tens f \]
and counit by the homomorphism $A\rightarrow \F_1$ that sends all non-zero elements to $1$. Non-boundary homomorphisms are automatically bialgebra homomorphisms.
If $A=K_A$ is an $\F_1$-field - so $A\setminus 0$ is an Abelian group - it is even a Hopf algebra. 

Thus any affine, quasi-integral $\F_1$-scheme $X$ carries the structure of a \emph{monoid scheme} over $\F_1$, with unit the unique $\F_1$-point sitting over the generic point of $X$, and this structure is automatically compatible with non-boundary morphisms. In particular,
\[\G_X:=\Hom(K_X^\times,\G_m)=\Spec K_X\] carries the structure of a \emph{group $\F_1$-scheme}. An $\F_1$-field is determined by its \emph{character group}
\[ K_A^\times=\Hom(\Spec K_A,\G_m)=K_A\setminus 0 \]
where of course write $\G_m$ for the group scheme $\Spec\F_1[t^{\pm1}]$. 

More generally, $\G_X$ is defined for any connected, quasi-integral $\F_1$-scheme $X$ and acts naturally, so for any non-boundary homomorphism $X\rightarrow Y$ the square
\[\xymatrix{ \G_X\times X\ar[r]\ar[d] &  X\ar[d] \\ \G_Y\times Y\ar[r] & Y }\]
commutes. We refer to $K_X^\times$ also as the \emph{character group of $X$}. 

We will also see that for irreducible closed subscheme $Z\hookrightarrow X$, there is a natural quotient homomorphism $\G_X\rightarrow \G_Z$ such that the diagram
\[\xymatrix{ \G_X\times Z\ar[r]\ar[d] & \G_Z\times Z\ar[r] &   Z\ar[d] \\ \G_X\times X\ar[rr] && X }\]
commutes.

Of course, this action is preserved - and becomes visible at the level of underlying topological spaces - after base change to $\Z$. In other words, any scheme with a quasi-integral model over $\F_1$ inherits an action by an equivariantly embedded diagonalisable group. The classification theorem in this section can therefore be thought of as a classification of ordinary schemes with this extra structure.

In particular, if $X$ is a normal, separated, and of finite type, and $K_X^\times$ is torsion free, then its base change is a \emph{toric variety} in the sense of \cite{Fultor}, and any non-boundary homomorphism yields a \emph{toric morphism} of the base changes.

\subsection{Embedded cones}\label{FAN_CONE}

Suppose again that $X=\Spec A$ is affine. Then $A\setminus0$ is a generating cone in the character group $K^\times_A$. Its \emph{polar} \[\sigma_A=\{v:K_A^\times\rightarrow\Q|v(A\setminus 0)\leq 0\}\] is a strongly convex cone in the rational vector space $N_A(\Q):=\Hom(K^\times_A,\Q)$. This sets up a functor
\[ \Sigma:\mathbf{Sch}_{\F_1}^\mathrm{aff/qi/nb} \rightarrow \mathbf{Cone}^N,\quad \Spec A \mapsto (K_A^\times,\sigma_A) \]
from the category of affine quasi-integral $\F_1$-schemes with non-boundary morphisms to the category $\mathbf{Cone}^N$ of pairs $(\sigma,K^\times)$ consisting of an Abelian group $K^\times$ and a strongly convex, reflexive cone $\sigma$ in its rational dual $N(\Q)$. 

More generally, we treat $N$ as a functor of Abelian groups, writing $N(H)=\Hom(K^\times,H)$ for a group $H$. When $H$ is totally ordered, the points of $N(H)$ can be thought of as \emph{valuations} on $K$. Each non-zero rational function $f\in K^\times$ induces a linear function $N(H)\rightarrow H$, denoted $\log f$, and this correspondence is injective as long as $K^\times$ is torsion-free.

 In  toric parlance, $N_A(\Z)$ is the group of cocharacters, or 1-parameter subgroups, of $\G_A$.

\paragraph{Open subsets}The map of cones associated to a localisation $A\rightarrow A[f^{-1}]$ is the inclusion of the face of $\sigma_A$ where $\log f=0$. 

Let us define a \emph{face} of an object $(\sigma,K^\times)$ of $\mathbf{Cone}^N$ to be a pair $(\tau,K^\times)$, where $\tau$ is a subcone of $\sigma$  defined by a collection of functions $f\in K^\times$ such that $\log f\leq 0$ on $\sigma$. Such a sub-cone is automatically strongly convex and reflexive. We call a face \emph{principal} if it can be cut out by a single function $f$. Of course, in the case $\sigma=\sigma_A$, then any element of $A$ cuts out a principal face of $\sigma_A$.

Then $\Sigma$ induces a bijection between the set of open immersions into $\Spec A$ and the set of principal faces of $\sigma_A$.

\paragraph{Closed subsets}Let $\lie p\trianglelefteq A$ be a prime ideal. Then $A/\lie p$ is in canonical bijection with $A\setminus\lie p\subseteq A\setminus 0$, and so $K_{A/\lie p}^\times\subseteq K_A^\times$. We obtain surjective maps
\[ N_A\twoheadrightarrow N_{A/\lie p}, \quad \sigma_A\twoheadrightarrow\sigma_{A/\lie p} \]
with kernel the subspace of $N_A$ of valuations \emph{centred at $\lie p$}, that is, that restrict to zero on $A\setminus\lie p$. The intersection of this kernel with $\sigma_A$ is the intersection of the faces cut out by the functions $\log f$ for $f\in\lie p$.

This correspondence is natural in $\lie p$, and hence puts the set of irreducible closed subsets of $\Spec A$ in inclusion-reversing correspondence with the set of linear quotients of $(K_A^\times,\sigma_A)$ by faces of $\sigma_A$.

Since any morphism can be written as a non-boundary morphism followed by a closed embedding, embedded cones can be used to describe all morphisms between affine quasi-integral $\F_1$-schemes.

\paragraph{Points}Let $H$ be an Abelian group. In order to make sense of the set $\sigma(H)\subseteq N(H)$ of $H$-rational points of the cone $\sigma$, we need a partial order on $H$. Thus $\sigma$ can be thought of as a functor of partially ordered groups (in the sequel, \emph{pogroups}). For concreteness, the reader may like to focus on the case that $H$ is totally ordered and Archimedean, {i.e.} of rank one. 

The free $\F_1$-algebra on $H$, when equipped with the $t$-adic topology, is called an $\F_1$-\emph{valuation field} and denoted $\F_1(\!(t^{-H})\!)$. Its \emph{ring of integers}, comprised of negative elements of $H$, or positive powers of $t$, is denoted $\F_1[\![t^{-H}]\!]$. The underlying discrete $\F_1$-algebra is what we would have called $\F_1[z^{H_{\leq 0}}]$ in the notation of \S\ref{FSCH}.

One usually thinks of the spectrum $\Delta$ of $\F_1[\![t]\!]=\F_1[\![t^\Z]\!]$ as a \emph{formal disc}; we will generalise this to arbitrary $H$ and write
\[ \Delta_H:=\Spec\F_1[\![t^{-H}]\!] \]
for the formal disc and punctured disc with exponent group $H$. Morphisms from $\Delta$ (resp. $\Delta_H$) into a formal scheme are called \emph{jets} (resp. \emph{$H$-jets}). While we are considering only schemes, the topology of $\F_1[\![t^{-H}]\!]$ plays no r\^ole.

Geometrically, the set $\sigma_A(H)$ of valuations from $K_A$ into $H$, non-positive on $A$, correspond to non-boundary $H$-jets $\Delta_H\rightarrow \Spec A$. The point sets are a left exact functor.

\paragraph{Topological realisation}
The vector space $N(\R)$ carries a natural weak topology coming from the order topology on $\R$, and this is inherited by $\sigma(\R)$, making it a contractible $\R_{\geq0}^\times$-space. It also inherits an $H$-affine structure - in particular, a subset of $H$-rational points - for any additive subgroup $H\subseteq\R$.

The topological realisation is a left exact functor, and $\sigma(\R)$ is a CW complex if $\sigma$ is finite-dimensional.


\subsection{Local finiteness conditions}

The correspondence $\mathbf{Sch}_{\F_1}^\mathrm{aff/qi/nb} \rightarrow \mathbf{Cone}^N$ has a fully faithful left adjoint
\[ (K^\times,\sigma)\mapsto \Spec \F_1[\sigma^\circ\cap K^\times]. \]
The counit of this adjunction is an isomorphism - that is, a quasi-integral $\F_1$-algebra $A$ is determined by its cone $\sigma_A$ - under the following independent conditions:
\begin{enumerate}
\item $A\setminus0$ is a saturated submonoid of $K_A^\times$;
\item the $\Q_{\geq0}$-span of $A\setminus0$ in $K_A^\times\tens\Q$ is reflexive.
\end{enumerate}
The first condition by definition \ref{INT_RUGGET} says that $\Spec A$ is \emph{normal}. The second has the following geometric meaning: let $D=(f_+/f_-)$ be a principal divisor; then if the pullback of $D$ along any $\Q$-jet $\Delta_\Q\rightarrow \Spec A$ is effective, then $D$ is effective. In other words, $\Q$-jets are enough to detect poles of rational functions.

\begin{defn}Let us say that a quasi-integral scheme $X$ has \emph{enough jets} if a locally principal divisor with effective pullback to any $\Q$-jet is effective.\end{defn}

The left adjoint to $\Sigma$ associates to an affine quasi-integral scheme $X$ a universal non-boundary morphism \[ X^\mathrm{ej}\rightarrow X \]
from a normal $\F_1$-scheme with enough jets. This morphism is integral if and only if $X$ already had enough jets, in which case it is the normalisation.

\begin{lemma}\label{FAN_AFFINE}The adjunction $\mathbf{Sch}_{\F_1}^\mathrm{aff/qi/nb} \leftrightarrows \mathbf{Cone}^N$ restricts to an equivalence on the full subcategory of $\mathbf{Sch}_{\F_1}^\mathrm{aff/qi/nb}$ whose objects are normal with enough jets.\end{lemma}

\begin{eg}In the case that $\langle\sigma_A\rangle$ is finite-dimensional, $A$ has enough jets if and only if $(A\setminus 0)\tens\Q\subset K_A^\times\tens\Q$ is closed in the order topology of $\Q$. For instance, any valuation ring $\F_1[\![t^{-H}]\!]$ with $H$ totally ordered of rank greater than one does not have enough jets. Let $-1\in H$ generate a minimal convex subgroup. Then the reflexivisation of $\F_1[\![t^{-H}]\!]$ is a localisation at $t$.

If, on the other hand, the topological boundary of $\sigma_A(\R)$ has no non-zero rational points, then $A$ automatically has enough jets. Since in this case, the only face is zero, the underlying space of $\Spec A$ is the two-point Sierpinski space: any non-zero function $f$ induces a homeomorphism $\Spec A\rightarrow \A^1_{\F_1}$.\end{eg}

\paragraph{Relative finiteness}Let us call an object $(\sigma,K^\times)$ of $\mathbf{Cone}^N$ \emph{rational polyhedral} if its span $\langle\sigma\rangle(\Q)$ in $N(\Q)$ is finite-dimensional and $\sigma\subseteq\langle\sigma\rangle$ is rational polyhedral in the sense of \cite[\S1.1]{Fultor}. More generally, we say that a morphism $f:\sigma_1\rightarrow\sigma_2$ of cones is rational polyhedral if there exists a rational polyhedral cone $\sigma_2^\prime\subseteq N_1$ such that $\sigma_1=\sigma_2^\prime\cap f^{-1}\sigma_2$.

\begin{lemma}\label{FAN_FINITE}Let $A\rightarrow B$ be a non-boundary homomorphism of quasi-integral $\F_1$-algebras.
\begin{enumerate}\item If $B$ is integral over a finite type $A$-algebra, then the kernel $\mathrm{ker}$ of $N_B(\Q)\rightarrow N_A(\Q)$ is finite-dimensional and $\sigma_B\rightarrow\sigma_A$ is rational polyhedral.
\item If $A\rightarrow B$ is of finite type, then \emph{i)} and $\mathrm{ker}(\Z)\tens\Q\tilde\rightarrow \ker(\Q)$.\end{enumerate}
If $B$ has enough jets, then the converses to these statements hold.\end{lemma}

\begin{eg}Note that $X\rightarrow S$ finitely presented does not imply even that $\nu X\rightarrow\nu S$ is of finite type. Take, for example, $K^\times=\Q^2$ with basis $x,y$, and $\sigma_A$ the positive orthant. The $\F_1$-scheme $X=\Spec A$ is normal; however, the blow-up of $X$ at the finitely generated ideal $T=(x,y)$ is not, and its normalisation is of infinite type. This can be seen from the fact that the convex hull of $T\setminus0$ in $A\setminus 0$ is not finitely generated as an $A\setminus 0$-module.\end{eg}

\begin{lemma}Let $A$ be a quasi-integral $\F_1$-algebra.
\begin{enumerate}\item If $A$ is integral-over-Noetherian, then $\sigma_A$ is rational polyhedral. 
\item If $A$ is Noetherian, then \emph{i)} and $\langle\sigma_A\rangle(\Z)\tens\Q\tilde\rightarrow \langle\sigma_A\rangle(\Q)$.\end{enumerate}
If $A$ has enough jets, then the converses to these statements hold.\end{lemma}
\begin{proof}By lemma \ref{FIN_NOETHER}, we are reduced to checking finiteness of the homomorphism $\F_1[A^\times]\rightarrow A$, which follows from lemma \ref{FAN_FINITE}.\end{proof}

\subsection{Glueing}\label{FAN_GLUE}

Since $\Sigma$ has a left adjoint, it is, in particular, left exact. It is therefore the pullback functor of an essential geometric morphism
\[\phi:\PSh\mathbf{Cone}^N\rightarrow\Sh\mathbf{Sch}_{\F_1}^\mathrm{qi/nb}\]
between the presheaf topoi. We equip $\mathbf{Cone}^N$ with `open immersions' the class of face inclusions of finite codimension. The category of \emph{embedded cone complexes} - that is, locally representable presheaves on $\mathbf{Cone}^N$ - is denoted $\mathrm{C}\mathbf{Cone}^N$.

As we have already noted, $\Sigma$ preserves open immersions, and so its extension preserves the subcategories of locally representable objects
\[ \Sigma:\mathbf{Sch}_{\F_1}^\mathrm{qi/nb}\rightarrow \mathrm{C}\mathbf{Cone}^N. \]
We thus associate to any quasi-integral scheme $X/\F_1$ a collection of cones glued together along faces in a way that respects their embedding into $N$.

\begin{prop}The pullback $\Sigma$ for the geometric morphism
\[ \phi:\PSh\mathbf{Cone}^N\rightarrow\Sh\mathbf{Sch}_{\F_1}^\mathrm{qi/nb}  \]
preserves open immersions and induces bijections on open subobject lattices. 

It has a fully faithful left adjoint with image the full subcategory generated under colimits by the normal schemes with enough non-boundary jets.\end{prop}

\begin{cor}The cone complex functor \[ \Sigma:\mathbf{Sch}_{\F_1}^\mathrm{qi/nb}\rightarrow \mathrm{C}\mathbf{Cone}^N\] induces, for every $X\in\mathbf{Sch}_{\F_1}^\mathrm{qi/nb}$, a homeomorphism $\Sh(\Sigma_X)\tilde\rightarrow\Sh(X)$.

The left adjoint to $\Sigma$ is fully faithful with image the category $\mathbf{Sch}_{\F_1}^\mathrm{ej/n/nb}$ of normal schemes with enough jets. In particular, $\Sigma$ restricts to an equivalence on this subcategory.\end{cor}

\begin{cor}The specialisation order on the topological space underlying a quasi-integral scheme $X$ with enough jets is the inclusion order on cones of $\Sigma_X$.\end{cor}

\paragraph{Developing map}
By taking colimits, a cone complex $\Sigma$ can be realised as a functor of pogroups
\[ \Sigma_X(H)=\Hom^\mathrm{nb}(\Delta_H,X) \]
and, in the case $H=\R$, as an $\R_{\geq0}^\times$-equivariant topological space. By definition, $\Sigma(\R)$ is strongly topologised with respect to cone inclusions $\sigma(\R)\hookrightarrow\Sigma(\R)$. The $\R_{\geq0}^\times$-action contracts  $\Sigma_X(\R)$ onto a discrete set, which may be identified with $\pi_0(X)$.

If $\Sigma$ is a \emph{finite} complex - that is, if it is qcqs as an object of $\Sh\mathbf{Cone}^N$ - cone inclusions are actually topological immersions. It follows that on qcqs objects of $\mathrm C\mathbf{Poly}^N$, topological realisation is left exact.
If the cones of $\Sigma$ are finite-dimensional, $\Sigma(\R)$ is a CW complex. Since every cone complex is a filtered colimit of finite subcomplexes and filtered colimits of CW complexes are exact, topological realisation is left exact on all complexes with finite-dimensional cones.

When $\Sigma\in\mathrm{C}\mathbf{Cone}^N$ is connected, $N$ is constant, and we get a \emph{developing map} $\delta:\Sigma\rightarrow N$ that is linear and injective on each cone. If the developing map is globally injective, $\Sigma$ is nothing more than a \emph{fan} in $N$.

\subsection{Criteria for separation and propriety}

Through its embedding into $\Sh\mathbf{Sch}_{\F_1}$, $\PSh\mathbf{Cone}^N$ inherits notions of overconvergence, separation, and propriety of morphisms. These correspond, as in \S\ref{SEP}, to the class $\P$ of morphisms that give rise under that embedding to integral-over-proper morphisms of schemes.

Suppose that $\Sigma$ is connected with cocharacter group $N$. We have seen that the rational (resp. integral) points of $N$ correspond to non-boundary $\tilde\A^1_{\F_1}\setminus0$-points (resp, $\A^1_{\F_1}\setminus0$-points) of the associated scheme, while rational (resp. integral) points of $\Sigma$ correspond to $\tilde\A^1_{\F_1}$-points (resp. $\A^1_{\F_1}$-points). To every point of $N$, then, we associate an extension problem which, since $\A^1$ cannot be modified, has a solution if and only if the point lifts to $\Sigma$. 

In other words, the developing map $\Sigma\rightarrow N$ induces an injection (resp. bijection) on $\Z$ or $\Q$-points whenever $\Sigma$ is separated (resp. overconvergent). More generally, an overconvergent morphism $\tilde\Sigma\rightarrow\Sigma$ induces Cartesian squares
\[\xymatrix{\Sigma_X(\Z)\ar[r]\ar[d] & N_X(\Z) \ar[d]  & \Sigma_X(\Q)\ar[r]\ar[d] & N_X(\Q) \ar[d] \\
 \Sigma_S(\Z)\ar[r] & N_S(\Z) & \Sigma_S(\Q)\ar[r] & N_S(\Q)}\]
of sets, and a separated morphism at least injects into the fibre product.

Any point $\tilde\A^1_{\F_1}\rightarrow X$ factors through some quasi-integral closed subscheme. In the spirit of \cite[\S2.4]{Fultor}, it is not unreasonable to expect converse statements. However, since the valuative criterion \ref{SEP_VAL} presented in this paper actually depends on the following theorem, the logic here is somewhat inverted.

\begin{prop}\label{FAN_SUR}Let $f:X\rightarrow S$ be a non-boundary morphism of finite type between quasi-integral, locally integral/Noetherian $\F_1$-schemes. The following are equivalent:
\begin{enumerate}\item $f$ is separated;
\item $f$ is locally separated;
\item for each cone $\sigma$ of $\Sigma_S$ and connected component $\Sigma_0$ of $\sigma\times_{\Sigma_S}\Sigma_X$, $\Sigma_0(\Q)$ is a fan in $N_0(\Q)$;\end{enumerate}
Suppose further that either $f$ is finitely presented or that $N_X(\Q)$ is locally finite-dimensional.\footnote{The $N_X(\Q)$ finite-dimensional case does not follow from the valuative criterion, but one can obtain it directly from some simple convex geometry arguments.} The following are equivalent:
\begin{enumerate} \item the restriction of $f$ to each connected component of $X$ is proper; \item $f$ is overconvergent;
\item for each cone $\sigma$ of $\Sigma_S$ and connected component $\Sigma_0$ of $\sigma\times_{\Sigma_S}\Sigma_X$, $\Sigma_0(\R)$ is a finite fan with support $f^{-1}\sigma(\R)$.\end{enumerate}\end{prop}
\begin{proof}These are special cases of theorems \ref{PFAN_SEP} and \ref{PFAN_PROP}.\end{proof}

This finishes the proof of theorem \ref{FAN_THM}.

\section{From formal schemes to punctured fans}\label{PFAN}

The straightforward nature of linear topologies on $\F_1$-algebras means that the local picture for integral formal $\F_1$-schemes is largely the same as for algebraic $\F_1$-schemes. However, the global structure is not so trivial as before, and to characterise interesting classes of formal schemes we will be forced to work at the level of cone complexes.

\subsection{Diagonalisable groupoid action}

\begin{defn}Let $\mathbf P$ be one of the properties quasi-integral, integral, normal. We say that an $\F_1$ formal scheme $X$ is $\mathbf P$ if ${}^?X$ is $\mathbf P$. A morphism $X\rightarrow Y$ of formal schemes is non-boundary if ${}^?X\rightarrow{}^?Y$ is non-boundary.

We notate the categories of formal schemes with these properties with the same superscripts as in the algebraic case.\end{defn}

The diagonalisable group action discussed in \S\ref{FAN_DIAG} persists for affine, quasi-integral formal schemes, but with no generic point there is no chance of globalising in general. 

By patching together the objects $\G_U\times U$ for affine open $U\subseteq X$, we obtain instead an action of a "local system of algebraic groups", or more precisely an $\F_1$-formal \emph{groupoid} 
\[\xymatrix{ \G_X^2/X\ar[r]^\mu & \G_X/X\ar@(ur,ul)[]_\iota \ar@<3pt>[r]^-\pi\ar@<-3pt>[r]_-\sigma & X\ar[l] }\]
whose structural morphisms are representable by $\F_1$-schemes. Here of course \[\G_X/X=\Hom_X(K_X^\times,\G_{m,X}),\] so the `character group' of this groupoid is the local system $K_X^\times$.

This structure is functorial for non-boundary morphisms, and a similar story to that of \S\ref{FAN_DIAG} also holds for quasi-integral formal subschemes of the boundary.

\subsection{Punctured cones}\label{PFAN_CONE}

Let $(\Spec A,Z)\in {}^Z\mathbf{Sch}_{\F_1}$ be an affine quasi-integral scheme marked along a single subscheme. In \S\ref{FAN_CONE}, we associated to $A$ a strongly convex, reflexive cone $\sigma_A\in\mathbf{Cone}^N$ in the rational dual $N_A$ of $K_A^\times$.

The formal completion of $A$ depends only on the complement $\Spec A\setminus Z$, which, as discussed in \S\ref{FAN_CONE}, corresponds to a union of faces $\zeta$ of $\sigma_A$. The faces $\zeta$ are cut out by the equations $\log f_\zeta=0$ where $f\neq0$ is in the ideal of $Z$. In the setting of lemma \ref{FAN_AFFINE} - that is, when $A$ is normal and has enough jets - $\hat A$ is determined by the pair $(\sigma,\zeta)$. These are the combinatorial data we can associate to $\F_1$-formal schemes.

\begin{defn}An \emph{punctured embedded cone} is a reflexive, strongly convex cone $\sigma\in\mathbf{Cone}^N$ together with a specified collection of proper principal faces $\zeta\subset\sigma$, the \emph{punctures}. If the collection is finite, we say that $\sigma\setminus\zeta$ has \emph{finite puncturing}. A morphism $\sigma_1\rightarrow\sigma_2$ of punctured cones must restrict to a map $\sigma_1\setminus\zeta_1\rightarrow\sigma_2\setminus\zeta_2$. The category of punctured embedded cones is denoted $\mathbf{Cone}_*^N$. In the sequel, I will usually omit the word `embedded', and also the notation $\zeta$, where this cannot cause confusion.

A \emph{face} of a punctured cone is a face of the underlying cone (defined by an element of $K^\times$), not entirely contained in $\zeta$, with the induced puncturing.\end{defn}

Note that the category $\mathbf{Cone}_*^N$ is not finitely complete: faces of a punctured cone can have empty intersection, and the empty set is not a punctured cone.

We have associated a punctured cone to any affine quasi-integral marked scheme \emph{not entirely contained in $Z$}; that is, any marked scheme in the essential image of the algebraisation functor
\[?:\mathbf{FSch}_{\F_1}^\mathrm{aff/qi/nb}\rightarrow {}^Z\mathbf{Sch}^\mathrm{aff/i/nb}_{\F_1}, \]
by precomposition with which we obtain a \emph{punctured cone functor}
\[ \Sigma:\mathbf{FSch}_{\F_1}^\mathrm{aff/qi/nb}\rightarrow\mathbf{Cone}^N_*, \quad \lim_{n\rightarrow\infty}A/T^n \mapsto \left(\sigma_A,\zeta_A\right). \]

In the other direction, by considering a punctured cone as a cone and applying the left adjoint construction, we obtain an affine, normal scheme, and the punctures mark it along a closed subscheme. By composing with formal completion, we obtain a fully faithful left adjoint to $\Sigma$. Its essential image has objects the normal formal schemes with enough jets.

\begin{lemma}\label{PFAN_AFFINE}The adjunction $\mathbf{FSch}_{\F_1}^\mathrm{aff/qi/nb} \leftrightarrows \mathbf{Cone}_*^N$ restricts to an equivalence on the full subcategory of $\mathbf{FSch}_{\F_1}^\mathrm{aff/qi/nb}$ whose objects are normal with enough jets.\end{lemma}

\begin{lemma}[Krull intersection]\label{KRULL}Let $A$ be a quasi-integral $\F_1$-algebra, $T\trianglelefteq A$ a finitely generated ideal. One of the following are true:
\begin{enumerate}\item $A\rightarrow\hat A$ is injective;
\item $\log f:\sigma_A(\Q)\rightarrow\Q$ is identically zero for some $0\neq f\in T$.\end{enumerate}\end{lemma}
\begin{proof}Since formal completion along $T$ can be written as a finite sequence of formal completions at the generators of $T$, we may assume $T=(t)$ is principal. A non-zero element of $\bigcap_{n\in\N}(t^n)$ is, as an additive function $\sigma_A(\Q)\rightarrow\Q$, bounded above by all multiples of $\log t$. Such exists if and only if $\log t$ is identically zero.\end{proof}

\begin{remark}If $A$ is integral and, say, Noetherian (so $\langle\sigma_A\rangle(\Q)$ is finite-dimensional), then this result implies that a formal completion $A\rightarrow\hat A$ is either bijective or \emph{zero}. If $A$ is only quasi-integral, it may have non-invertible elements that are identically zero on $\sigma_A(\Q)$.\end{remark}

\paragraph{Open immersions}The stupid feature of monoidal geometry that makes this construction possible is the following:

\begin{lemma}The restriction of the functor $?$ of forgetting the topology to the category of \emph{integral} $\F_1$-algebras preserves localisations.\end{lemma}
\begin{proof}Let $A$ be complete with respect to a finitely generated ideal $T$, and let $A\rightarrow A\{f^{-1}\}$ be a non-zero completed localisation. Then $f\not\in T$ and so, since $A$ is integral, for every $t\in T$, $\log f\neq\log t$. It follows that $\log t$ is also non-zero on the face $\log f=0$.

By the Krull intersection theorem (lemma \ref{KRULL}), the $T$-adic filtration of $A^?[f^{-1}]$ is separated. Therefore $A^?[f^{-1}]=A\{f^{-1}\}^?$.\end{proof}

This statement is false for quasi-integral algebras. However, since normalisation is a homeomorphism, it still follows that $\Sigma$ exchanges localisations with face inclusions:

\begin{lemma}\label{PFAN_IMMERS}The punctured cone functor $\Sigma$ preserves open immersions. For each $A\in\Alg_{\F_1}^\mathrm{qi/nb}$, $\Sigma$ identifies the poset of non-zero localisations of $A$ with that of faces of $\sigma_A\setminus\zeta_A$.\end{lemma}

\paragraph{Points}Let us define the $H$-points of a punctured cone $\sigma\setminus\zeta$ to be $\sigma(H)\setminus\zeta(H)$, so that $\sigma\setminus\zeta(H)$ remains the set of non-boundary morphisms $\Delta_H=\Spec\F_1[\![t^{-H}]\!]\rightarrow X_\sigma$. It is important to distinguish in the notation between $\sigma\setminus\zeta(H)$ and $\sigma(H)$, and so we will never omit $\zeta$ when talking about points.

\paragraph{Topological realisation}When $H=\R$, the set of points $\sigma\setminus\zeta(\R)$ is a $\R_{>0}^\times$-invariant open subset of a cone in a real vector space. 
The $\R^\times_{>0}$-action is free as soon as $\zeta\neq\emptyset$. If $A$ is integral-over-Noetherian, then $\sigma_A\setminus\zeta_A(\R)$ is a finite-dimensional CW complex.

\paragraph{Relative variant}In the case $A$ is defined over a valuation ring $k[\![t^{-H}]\!]$ for some pogroup $H$ and coefficient $\F_1$-field $k$, we can instead use the punctured cone inside the fibre product
\[\xymatrix{ N_X(H^\prime) \ar[r]\ar[d] & \Hom(K_X^\times,H^\prime)\ar[d] \\ H^\prime\ar[r] & \Hom(H,H^\prime)}\] which is defined as a functor of pogroups over $H$. When $H$ is not a subgroup of $\Q$, the usual cone associated to $A$ will be infinite-dimensional, and so this variant is likely to be more sensible. We use the same notation.

When $H\subseteq \R$, we can also define the topological realisation. If $A$ is integral over a finite type $k[\![t^{-H}]\!]$-algebra, then the variant $\sigma_A\setminus\zeta_A(\R)$ is a finite-dimensional CW complex.

\subsection{Glueing}\label{PFAN_GLUE}
The globalisation of punctured cone functor proceeds much as in the schemes case \S\ref{FAN_GLUE}.

\begin{lemma}\label{PFAN_FLAT}$\Sigma$ is a flat functor.\end{lemma}
\begin{proof}Since we are dealing with presheaf categories, we need to show two things: first, that a finite diagram in $\mathbf{FSch}_{\F_1}^\mathrm{qi/nb}$ that has a limit in $\mathbf{Cone}_*^N$ has a limit in $\mathbf{FSch}_{\F_1}^\mathrm{qi/nb}$, and second, that $\Sigma$ is left exact. The latter follows from the fact that $\Sigma$ has a left adjoint.

A fibre product $\sigma_X\times_{\sigma_Z}\sigma_Y$ is representable in $\mathbf{Cone}^N_*$ if and only if the images in $\sigma_Z$ of $\sigma_X\setminus\zeta_X$ and $\sigma_Y\setminus\zeta_Y$ have non-empty intersection. In particular, it has a $\Q$-point. The fibre product $X\times_ZY$ in $\mathbf{FSch}^\mathrm{aff}_{\F_1}$ is therefore non-empty and hence an object of $\mathbf{FSch}_{\F_1}^\mathrm{aff/qi/nb}$ (cf. the aside \ref{INT_EMPTYSET}).\end{proof}

As before, let us define the `open immersions' in $\mathbf{Cone}_*^N$ to be the principal face inclusions. From lemmas \ref{PFAN_AFFINE}, \ref{PFAN_IMMERS}, and \ref{PFAN_FLAT}, in reverse order, it follows:

\begin{prop}The punctured cone functor $\Sigma$ extends to the pullback along a geometric morphism
\[\PSh\mathbf{Cone}_*^N\rightarrow\Sh\mathbf{FSch}_{\F_1}^\mathrm{qi/nb}  \]
which preserves open immersions and induces bijections on open subobject lattices. 

It has a fully faithful left adjoint with image the full subcategory generated under colimits by the normal formal schemes with enough non-boundary jets.\end{prop}

We therefore obtain an adjunction
\[ \Sigma:\mathbf{FSch}_{\F_1}^\mathrm{qi/nb}\leftrightarrows\mathrm{C}\mathbf{Cone}^N_* \]
between the categories of quasi-integral formal schemes and non-boundary morphisms and that of \emph{punctured cone complexes}, that is, locally representable presheaves on $\mathbf{Cone}^N_*$.

\begin{cor}\label{PFAN_EQUIVALENCE}The cone complex functor and its left adjoint restrict to an equivalence
\[ \Sigma:\mathbf{FSch}_{\F_1}^\mathrm{ej/n/nb}\widetilde\rightarrow \mathrm{C}\mathbf{Cone}_*^N\] 
on the category of normal formal $\F_1$-schemes with enough jets.\end{cor}

\begin{cor}\label{PFAN_CONES_POINTS}The specialisation order on the topological space underlying a quasi-integral formal scheme $X$ with enough jets is the inclusion order on cones of $\Sigma_X$.\end{cor}

\paragraph{Developing map}Just as in the unpunctured case, a punctured cone complex can be realised as a functor of pogroups
\[ \Sigma_X(H)=\Hom^\mathrm{nb}(\hat\Delta_H,X) \]
and, by applying it to the case $H=\R$, as an $\R_{>0}^\times$-equivariant topological space. If $X$ is locally integral-over-Noetherian, $\Sigma_X(\R)$ is a CW complex. Unlike the scheme case, however, it is not necessarily a $0$-type.

The topological realisation commutes with fibre products of
\begin{enumerate}\item finite complexes, and
\item complexes with finite-dimensional cones.\end{enumerate}

The cocharacter groups assemble to form a local system $N_X$ on $\Sigma_X$. It can also be thought of as a conically invariant local system on $\Sigma_X(\R)$. The monodromy of this local system is the obstruction to the developing map $\Sigma_X\rightarrow N_X$ being globally defined. In particular, for each choice of cone $\sigma\subset\Sigma_X$ one gets a canonical developing map
\[ \delta:\widetilde\Sigma_X\rightarrow N_{X,\sigma} \]
from the cover $\widetilde\Sigma_X$ of $\Sigma_X$ that trivialises $N_X$ on the component of $\sigma$.

The same holds true for the relative situation $\widetilde\Sigma_{X/H}\rightarrow N_{X/H,\sigma}$.

\subsection{Overconvergent neighbourhoods}\label{PFAN_SUR}
As in the scheme case, the embedding of $\PSh\mathbf{Cone}_*^N$ into $\Sh\mathbf{FSch}_{\F_1}^{n/nb}$ equips it with a class of morphisms $\P$ that become formally proper under that embedding. The crux of the study of extension problems will be to identify what birational $\P$ morphisms look like at the level of topological realisations.

\paragraph{Refinements}A morphism $\Sigma_X\rightarrow\Sigma_Y$ from a finite punctured cone complex to a punctured cone is called a \emph{refinement} if it induces a bijection on points and on $K^\times$. Since taking points commutes with limits, this notion is stable for base change. It therefore makes sense to talk about refinements of arbitrary punctured cone complexes.

One way to induce a refinement of a cone $\sigma\subseteq N$ is by choosing a finite collection of linear functions $\{f_i\}_{i=1}^k\subseteq K^\times$. The functions may as well be assumed non-positive on $\sigma$. Thus we reach the following fact, true also in the algebrac case $\mathrm{C}\mathbf{Cone}^N$:

\begin{lemma} \label{PFAN_REFINEMENT}A refinement induced by a function is in class $\P$.\end{lemma}

Note that the inclusion of any (relatively) polyhedral subcone $\sigma_X\subseteq\sigma_Y$ can be extended to a refinement: if $\sigma_X$ is defined by equations $\log f_i\leq 0$, then the function you should take is $F_X=0\vee\bigvee_{i=1}^k\log f_i$.

We can also define a scaling invariant function $\partial F/\partial r$, where $\partial/\partial r$ generates the $\R_{>0}^\times$-action. Conical compactness of $N_X\setminus0$ ensures that this function is bounded above.

\paragraph{Puncturing along a subcomplex}Let us call an `open subset' $\Sigma_U$ of a punctured cone complex $\Sigma_X$ a \emph{subcomplex}. In other words, a subcomplex is a morphism $\Sigma_U\hookrightarrow\Sigma_X$ whose restriction to each cone of $\Sigma_X$ is the inclusion of a union of faces, with the induced puncturing.

It makes sense to further puncture $\Sigma_X$ along a subcomplex $\Sigma_U$. At the level of cones $\sigma\setminus\zeta\subseteq\Sigma_X$ it is defined by enlarging the puncturing of $\sigma$ to $(\Sigma_U\cap\sigma)\cup\zeta$. There is a natural morphism $\Sigma_X\setminus\Sigma_U\rightarrow\Sigma_X$.

Equivalently, $\Sigma_U\subseteq\Sigma_X$ gives rise to an open immersion of formal schemes $U\subseteq X$, and $\Sigma_X\setminus\Sigma_U$ is simply the complex obtained by applying $\Sigma$ to the formal completion of $X$ along the complement of $U$. In particular:

\begin{lemma}\label{PFAN_PUNCT}The puncturing of a complex along a subcomplex is of class $\P$.\end{lemma}

Note that at the level of topological realisations, a subcomplex whose inclusion is quasi-compact is always \emph{closed}. In particular, this applies to finite subcomplexes of quasi-separated complexes.

\begin{prop}[Overconvergent n'hoods are open n'hoods]\label{PFAN_MODIFICATION}Let $\tilde\sigma_V\in\mathrm{C}\mathbf{Cone}_*^N$ be a connected, finite complex, $\sigma_V\in\mathbf{Cone}_*^N$ a cone, and $f:\tilde\sigma_V\rightarrow\sigma_V$ a polyhedral morphism fixing a common cell $\tau\subseteq\sigma_V,\tilde\sigma_V$. The following are equivalent:
\begin{enumerate}\item $\tilde\sigma_V\setminus\tilde\zeta_V(\R)\rightarrow\sigma_V\setminus\zeta_V(\R)$ is the inclusion of a neighbourhood of $\tau\setminus\zeta(\R)$; \item $f$ is an overconvergent neighbourhood of $\tau$.\end{enumerate}\end{prop}
\begin{proof}The implication \emph{ii)}$\Rightarrow$\emph{i)} is the easy direction: it follows from lemmas \ref{PFAN_REFINEMENT} and \ref{PFAN_PUNCT} and the fact that a formally projective morphism is, by definition, a composite of a formal completion and an integral-over-projective morphism.

We focus on the converse. Since $\sigma_V$ and $\tilde\sigma_V$ share a face, they also have the same character group. The polyhedral morphism $f$ is therefore obtained by base change from a morphism between rational polyhedral cones in finite-dimensional vector spaces. Since passing to the topological realisation of finite complexes commutes with fibre products, it is safe to work with such a model; we may thus assume that $N_V=N_\tau$ is finite-dimensional and that $\tilde\sigma_V$ and $\sigma_V$ are rational polyhedral. It is now an exercise in elementary convex geometry.

We begin the exercise by forgetting about the punctures. This does not introduce complications because $\sigma_V$ is a cone and, by assumption, $\tilde\sigma_V\setminus\tilde\zeta_V\rightarrow\sigma_V\setminus\zeta_V$ is injective. Let $\sigma\subseteq\tilde\sigma_V$ be a cone defined by the functions $\{f_i\}_{i=1}^k$. Let $p\in N_\tau(\R)$ be in the interior of $\tau$. If $\log f_i(p)\leq 0$ for all $i$, then $p\in\sigma(\R)$. Since $\sigma$ must intersect $\tau$ in a face, in this case $\tau\subseteq\sigma$, and 
\[  F_\sigma:=\max\{ 0,\log f_i\}_{i=1}^k\]
defines a subdivision of $\sigma_V$ containing both $\sigma$ and $\tau$ as cells. 

Otherwise, $\bigvee_{i=1}^k\log f_i$ is strictly positive on the interior of $\tau$. It is therefore dominated by some positive linear combination $F$ of the $\log f_i$. Then the function
\[  F_\sigma:=\max\{ 0,\log f_i,F\}_{i=1}^k\]
does the job.

Repeating this choice for each cone of $\tilde\sigma_V$, we obtain a finite list of functions $F_\sigma$ whose join $\bigvee_\sigma F_\sigma$ defines a finite refinement of $\sigma_V$ that contains $\tau$ and contains a refinement of $\tilde\sigma_V$ as a subcomplex.

Now let us reintroduce the punctures. Since $\tilde\sigma_V\setminus\tilde\zeta_V(\R)$ is a neighbourhood of $\tau\setminus\zeta(\R)$, any cone of $\sigma_U$ whose interior does not meet $\tilde\sigma_V$ is disjoint from $\tau\setminus\zeta$. Let $\zeta_X$ be the union of all such cones. Then $\sigma_X\setminus\zeta_X\rightarrow\sigma_V\setminus\zeta_V$ is a morphism of class $\P$ that supers $\tilde\sigma_V$ and has a section over $\tau$.\end{proof}

\begin{cor}[of proof]Let $f:X\rightarrow S$ be a non-boundary morphism of finite type between quasi-integral formal $\F_1$-schemes. If $\Sigma_X\rightarrow\Sigma_S$ is a refinement, then $f$ is proper birational.\end{cor}

\subsection{Criteria for overconvergence}\label{PFAN_CRITERIUM}

Suppose that we are given a morphism $\Sigma_X\rightarrow\Sigma_S$ of punctured cone complexes, whose extensional properties we wish to investigate. By lemma \ref{SEP_ELEMENTARY}, we are checking extension problems
\[\xymatrix{ \sigma_U\ar[r]\ar[d] & \sigma_V\ar[d] \\ \Sigma_X\ar[r] & \Sigma_S }\]
with $\sigma_U\subset\sigma_V$ a face inclusion of punctured cones.

The new feature of cone complexes that we will use to understand this extension problem is the developing map $\Sigma_X\stackrel{\delta}{\rightarrow} N_X$. Since this is not everywhere defined, we must first reduce the question to studying a region of $\Sigma_X$ `close' to $\sigma_U$.

\begin{defn}Two cones of a punctured cone complex $\Sigma\in\mathbf{C}\mathbf{Cone}_*^N$ are said to be \emph{contiguous} if they intersect along a principal face. The \emph{big star} of a cone $\sigma\subset\Sigma$ is the complex $\sigma^\star$ obtained by puncturing $\Sigma$ along all cones discontiguous with $\sigma$. The \emph{small star} is the subcomplex $\sigma^*\subseteq\sigma^\star$ of cones containing $\sigma$ as a principal face.

A punctured cone complex is said to be \emph{locally finite} if the star of any cone is finite and has finite puncturing.\end{defn}

The developing map of a star-shaped complex
\[ \sigma^\star\stackrel{\delta}{\rightarrow} N_{\sigma^\star}=N_{X,\sigma} \]
is globally defined. Note that if $\Sigma$ is a connected cone complex without punctures, then all cones are pairwise contiguous, and so the star of any cone of $\Sigma$ is equal to $\Sigma$ itself.

\begin{lemma}Suppose that $\Sigma$ has finite puncture type.\begin{enumerate}\item The big star $\sigma^\star(\R)$ is a neighbourhood of $\sigma(\R)$ in $\Sigma(\R)$. \item The small star $\sigma^*(\R)$ is a neighbourhood of $\sigma(\R)$ minus its proper faces in $\Sigma(\R)$.\end{enumerate}\end{lemma}
\begin{proof}Indeed, since $\Sigma(\R)$ is strongly topologised with respect to cone inclusions, it is enough to check these statements on every cone.\end{proof}

Geometrically, the open star of $\sigma_U\subseteq\Sigma_X$ is obtained by formally completing $X$ along the closure of $U$. In the language of \S\ref{MORP_EMBED}, it is the \emph{formally embedded closure} $\mathrm{cl}(U/X)$ of $U$.

\begin{lemma}\label{riblet}$\sigma^\star\rightarrow\Sigma_X$ is an overconvergent neighbourhood of $\sigma/\Sigma_X$.\end{lemma}

In other words, to check overconvergence of a morphism $\Sigma_X\rightarrow\Sigma_S$ near $\sigma\subseteq\Sigma_X$, it is necessary and sufficient that its restriction to $\sigma^\star$ be overconvergent.

\paragraph{Separated}Suppose that we have an extension problem $\sigma_U/\sigma_V$ and that $\sigma_U$ lands in a cone $\sigma_0\subseteq\Sigma_X$. By lemma \ref{riblet}, we may, without loss of generality, replace $\Sigma_X$ with the star of $\sigma_0$. We obtain a linear map $\varphi:N_U\rightarrow N_{X,0}=N_{\sigma_0}$ and, since $N_U=N_V$, a commuting diagram
\[\xymatrix{ \sigma_U\ar[r]\ar[d] & \sigma_V\ar[r] & N_U\ar[d]^\varphi \\
 \sigma_0\ar[r] & \sigma_0^\star\ar[r]^-\delta & N_{X,0} }\]
By the characterisation \ref{PFAN_MODIFICATION} of overconvergent neighbourhoods, solutions $\mathrm{Sur}_{U/V}\rightarrow\sigma_0^\star$ correspond to lifts of $\varphi:\sigma_V(\R)\rightarrow N_{X,0}(\R)$ along $\delta$ in a neighbourhood of $\sigma_U(\R)$. Uniqueness of solutions is therefore related to local injectivity of the developing map.

\begin{thm}\label{PFAN_SEP}Let $\Sigma_X\in\mathrm{C}\mathbf{Cone}_*^N$. The following are equivalent:
\begin{enumerate}\item $\Sigma_X$ is locally separated;
\item $\Sigma_X$ is separated and any contiguous pair of cones intersect in a face;
\item the restriction of the developing map to the union of any contiguous pair of cones is injective;
\item the restriction of the developing map to the small star of any cone is injective.\end{enumerate}\end{thm}
\begin{proof}The equivalence \emph{iii)}$\Leftrightarrow$\emph{iv)} and the implication \emph{ii)}$\Rightarrow$\emph{i)} are straightforward. 

Suppose \emph{iii)}, and let $\sigma_V\rightrightarrows\sigma_0^\star$ be a pair of solutions. Since the homomorphism of character groups is predetermined, it suffices to check on points. 

We may assume $\sigma_V$ is a single cone, in which case both maps factor through the union of two cones $\sigma_1,\sigma_2\subseteq\sigma_0^\star$ contiguous along a face containing the image of $\sigma_U$. Both solutions must lift along the developing map, which by \emph{iii)} is injective. Therefore they are equal.

Criterion \emph{iii)} also implies that contiguous cones must intersect in at most one face. In particular, $\Sigma_X$ is quasi-separated, and \emph{ii)} follows.

It remains to prove \emph{i)}$\Rightarrow$\emph{iii)}. Write
\[ \sigma_U:=\sigma_1\cap\sigma_2\subseteq\Sigma_X,\quad \sigma_V:=\delta(\sigma_1)\cap\delta(\sigma_2)\subseteq N_U \]
and consider the extension problem $\sigma_U/\sigma_V$. By proposition \ref{PFAN_MODIFICATION}, local separatedness implies that the two sections $\sigma_V(\R)\rightarrow\sigma_i(\R)$ must be equal in a neighbourhood of $\sigma_U(\R)$. Therefore either $\sigma_1=\sigma_2$ or $\sigma_V=\sigma_U$. It follows that the restriction of $\delta$ to $\sigma_1\cup\sigma_2$ is injective.\end{proof}

Intuitively, separatedness of $\Sigma_X$ means that the developing map cannot `double back' when passing through from one cone into a neighbouring cone.

Suppose that $\Sigma_X$ is separated and locally finite-dimensional. For any point $p$ of $\Sigma_X(\R)$, there is a unique minimal cone $\sigma$ containing $p$. The small star $\sigma^*(\R)$ of $\sigma$ is then a neighbourhood of $p$ in $\Sigma_X(\R)$. By proposition \ref{PFAN_SEP}, the restriction of $\delta$ to $\sigma^*(\R)$ is injective. In other words, the developing map is a \emph{local immersion}.

\begin{cor}Suppose that every cone of $\Sigma_X$ is finite-dimensional. Then $\Sigma_X$ is separated if and only if the developing map is a local immersion.

If $\Sigma_X$ is also locally finite, then in fact every cone has a neighbourhood in $\Sigma_X(\R)$ on which the developing map is an immersion.\end{cor}
\begin{proof}The last statement is, \emph{a priori}, a bit stronger than what we discussed. Since we won't use that result, I only provide a sketch of the proof. 

Suppose $\Sigma_X$ is locally finite and separated, and let $\sigma$ be a cone. We will construct a set $U\subseteq N_\sigma(\R)$ whose preimage in $\sigma^\star(\R)$ is a connected neighbourhood of $\sigma\setminus\zeta(\R)$ and such that $\delta$ induces a bijection
\[ \pi_0(\sigma^\star\setminus\sigma)\rightarrow\pi_0(U\setminus\sigma(\R)) \]
and is therefore injective for topological reasons. First, by finiteness of the punctures there exists a (not necessarily strongly) convex polyhedral cone $C\subseteq N_\sigma(\R)$ containing $\sigma$ and such that $\partial C\cap\sigma=\zeta$. Choose a deformation retract $r$ of the interior of $C$ onto $\sigma\setminus\zeta(\R)$. 

Let $T$ denote the union of all punctured faces of $\sigma$ contiguous with another cone. The connected components of $T$ are indexed by $\pi_0(\sigma^\star\setminus\sigma)$. Setting $U=r^{-1}T\cup\sigma\setminus\zeta(\R)$, we have $U\cap\sigma=T$ and so $r:U\setminus\sigma\rightarrow T$ is a weak equivalence.\end{proof}

\begin{cor}\label{PFAN_SEPARATED}Let $f:X\rightarrow S$ be a non-boundary morphism of quasi-integral formal schemes with enough jets. The following are equivalent:
\begin{enumerate}\item $f$ is locally separated;
\item $f$ is separated and has affine diagonal;
\item for each cone $\sigma_S$ of $\Sigma_S$ and $\sigma_X$ of $\Sigma_X$, the restriction of the developing map to $f^{-1}\sigma_S\cap\sigma_X^*$ is injective;\end{enumerate}
and in the case that $\Sigma_X(\R)$ (or $\Sigma_{X/H}(\R)$) is locally finite dimensional, we may add
\begin{enumerate}\item[iv)] over each cone of $\Sigma_S$, the developing map is a local immersion. \end{enumerate}\end{cor}

\paragraph{Overconvergent}We return to the situation of an arbitrary morphism $f:\Sigma_X\rightarrow\Sigma_S$ of cone complexes and extension problem $\sigma_U/\sigma_V$. We are trying to complete the square
\[\xymatrix{\sigma_V\ar[r]\ar@{-->}[d] & N_U\ar[d]^\varphi \\ \sigma_0^\star\ar[r] & N_{X,0}}\]
for a cone $\sigma_0\subseteq\Sigma_X$ containing the image of $\sigma_U$. Evidently, any solution must factor through the star of $\sigma_U$ in the `fibre product' 
\[ \tilde\sigma_V\subseteq\sigma_V\times_{N_{X,0}}\sigma_0^\star:=\colim_{\sigma\subseteq\sigma_0^\star}(\varphi^{-1}\sigma\cap\sigma_V). \]
In other words, the extension problem has a solution if and only if $\tilde\sigma_V\rightarrow\sigma_V$ is an overconvergent neighbourhood of $\sigma_U$.

Applying this to the case $\sigma_U=\sigma_0$, we find:

\begin{lemma}Let $f:\Sigma_X\rightarrow\Sigma_S$ be overconvergent, and suppose that for every cone $\sigma_S$ of $\Sigma_S$ and $\sigma_X$ of $\Sigma_X$, $f^{-1}\sigma_S\cap\sigma_X^\star$ is finite-dimensional. Then $f$ is locally finite.\end{lemma}
\begin{proof}In this case every strongly convex cone in $N_U$ containing $\sigma_U$ as a face must intersect $\sigma_U^\star$ in finitely many cones. If $\delta(\sigma_U^\star)$ is contained in a finite-dimensional subspace, then this is enough to conclude that it is finite.\end{proof}

Let us call the finiteness condition of the lemma \emph{locally finite-dimensional}. If $f$ is locally finite, then it is equivalent to the condition that every cone of $\Sigma_X$ be finite-dimensional. We henceforth assume both of these conditions.\footnote{In fact, for certain questions the locally finite case can be reduced to the finite-dimensional case by picking a model of the star; however, for want of applications, we do not ask these.}

\begin{thm}\label{PFAN_PROP}Suppose that $\Sigma_X(\R)$ (or $\Sigma_{X/H}(\R))$ is locally finite-dimensional, and let $f:\Sigma_X\rightarrow\Sigma_S$ be a locally polyhedral morphism. The following are equivalent:
\begin{enumerate}\item $f$ is overconvergent;
\item $f$ is locally finite and the developing map is a local homeomorphism.\end{enumerate}
In this case, $N_{X/S}$ is everywhere spanned by cones of $\Sigma_X$ and in particular, finite-dimensional.\end{thm}
\begin{proof}Let $\sigma_U$ be a cone of $\Sigma_X$. Overconvergence says that for any cone $\sigma_V$ in $f^{-1}\sigma_S$ containing $\sigma_U$ as a principal face, $\sigma_V\cap\sigma_U^\star$ is a neighbourhood of $\sigma_U(\R)$. Thus in any finite-dimensional subspace $H$ of $N_X(\R)$, $\sigma_U^\star(\R)$ is a neighbourhood of $\sigma_U(\R)$ in $H\cap f^{-1}\sigma_S$. The finiteness of $\sigma_U^\star$ implies therefore that $f^{-1}\sigma_S$ is finite-dimensional.

Assume $\Sigma_S=\sigma_S$ consists of a single cone. Let $\sigma_0$ be a cone of $\Sigma_X$. By local finiteness, we may replace $\sigma_0^\star\rightarrow\sigma_S$ with a polyhedral model. We will prove that this is overconvergent near $\sigma$.

Let $\sigma_U/\sigma_V$ be an extension problem, and let $\tilde\sigma_V$ be the `canonical solution'. By hypothesis, there is a polyhedral open neighbourhood of $\sigma_0$ in $\sigma_0^\star$ the restriction to which of $\delta$ is an open immersion. The inclusion of this neighbourhood is overconvergent at $\sigma_0$. Therefore its pullback to $\tilde\sigma_V$ is an overconvergent neighbourhood of $\sigma_U$ in $\sigma_V$.\end{proof}

\begin{cor}\label{PFAN_PROPER}Let $f:X\rightarrow S$ be a non-boundary morphism locally of finite type between quasi-integral formal schemes with enough jets. Suppose that $S$ is locally integral-over-Noetherian (resp. integral over finite type over a valuation ring with value group $H\subseteq\R$). 

The following are equivalent:
\begin{enumerate}\item $f$ is overconvergent;
\item $f$ is paracompact and over each cone of $\Sigma_S$, the developing map of $\Sigma_X$ (resp. $\Sigma_{X/H}$) is a local homeomorphism.\end{enumerate}\end{cor}

In this case, the developing maps equip $\Sigma_X(\R)$ with the structure of a $\R^\times_{>0}$-equivariant smooth manifold - in fact, an \emph{affine manifold} - whose boundary is vertical over $\Sigma_S(\R)$. In particular, the fibres of $f$ are manifolds without boundary. If $f$ is proper, then the fibres are compact.

\begin{cor}The fibres of an overconvergent (resp. proper) morphism $\Sigma_X(\R)\rightarrow\Sigma_S(\R)$ are affine manifolds (resp. conically compact affine manifolds) without boundary in a unique way such that the relative developing map $f^{-1}(p)\stackrel{\delta}{\rightarrow}N_{X/S}(\R)$ is a local affine diffeomorphism.\end{cor}

\subsection{Meromorphic functions as obstruction to algebraisation}\label{PFAN_MERO}

The meromorphic function sheaf of any connected, quasi-integral $\F_1$-scheme is constant. A necessary condition for algebraisability of a connected, quasi-integral formal $\F_1$-scheme $X$ is therefore that $K_X^\times$ be constant. 

If this is the case, then $N_X$ is defined as a rational vector space, and we get a global developing map
\[ \delta:\Sigma_X\rightarrow N_X. \]
We can use $\delta$ to `put back in' the intersections of open sets in an algebraisation of $X$ that became empty upon formal completion. Informally, the prescription for two punctured cones $\sigma_i\setminus\zeta_i$ in $\Sigma_X$ is:
\[  \bigcap_{i=1}^k\sigma_i:=\left\{\begin{array}{cl} \{0\} & \text{if }\bigcap_{i=1}^k(\sigma_i\setminus\zeta_i)=\emptyset \\ \delta\left(\bigcap_{i=1}^k\sigma_i\right) & \text{otherwise}\end{array}\right. \]
Note that unless $X$ has affine diagonal, $\bigcap_{i=1}^k\sigma_i$ might consist of several faces. This defines an atlas for an object ${}^?\Sigma_X$ of $\PSh\mathbf{Cone}^N$ whose transition maps are principal face inclusions. Moreover, it comes equipped with a compatible collection of marked cones $\zeta$ that, when punctured, retrieve $\Sigma_X$.

To see that it is locally representable, we can apply the following elementary lemma:

\begin{lemma}\label{TOPOS_ATLAS}Let $\Sh\mathbf C$ be a spatial theory, $I$ a cofiltered poset, and $V:I\rightarrow\mathbf C$ a functor such that for every $f\in I$, $Vf$ is an open immersion. Suppose that $V$ is \emph{locally flat}, meaning that for each $i\in I$, $V:I_{/i}\rightarrow \mathbf C_{/V_i}$ is flat. Then $V$ is an atlas for a locally representable sheaf; that is, $\colim V$ is locally representable and $V$ is an open covering.\end{lemma}
\begin{proof}By proposition \ref{TOPOS_FUNCT}, the condition on the slice sets ensures that for any $i\in I$ we get a geometric morphism $\Sh\mathbf C_{V_i}\rightarrow\Sh I_i$ whose pullback takes monomorphisms to open immersions. It follows that for any subset $K$ of $I_{/i}$, the restriction $V_{/K}$ of $V$ to the slice set $I_{/K}$ is an open subset of $V_i$.

Since $\Sh\mathbf C$ is a topos, a pushout of monomorphisms is also a pullback. In particular, a pushout of two locally representable sheaves along an open subset is locally representable.

By induction on the number of elements, $V_{/J}$ is locally representable for any finite $J\subseteq I$, and an inclusion $J\subseteq K$ induces an open immersion $V_{/J}\hookrightarrow V_{/K}$. The result now follows from the fact that inductive systems of open immersions are locally representable.\end{proof}

It follows that ${}^?\Sigma_X$ is a cone complex that can be punctured to retrieve $\Sigma_X$. If $X$ is normal with enough jets, this of course yields an algebraisation of $X$ itself. 

We make these arguments precise in the proof of the following statement, which does not, in fact, depend on the classification theorem:

\begin{thm}\label{PFAN_ALG}Let $X$ be a connected, integral formal scheme. The following are equivalent:
\begin{enumerate}\item $X$ is algebraisable;
\item $K^\times_X$ is a constant sheaf.\end{enumerate}
There exists an algebraisation functor
\[ ?:\mathbf{FSch}_{\F_1}^\mathrm{alg/i/nb}\rightarrow{}^Z\mathbf{Sch}^\mathrm{i/nb}_{\F_1}, \quad X\mapsto {}^?X \]
on the category of algebraisable integral formal schemes such that ${}^?X$ has affine diagonal over $\F_1$ whenever $X$ does.\footnote{Note that the corresponding statement for formal algebraic spaces is completely trivial.}
\end{thm}
\begin{proof}The restriction of the forgetful functor to $\mathbf{FSch}_{\F_1}^\mathrm{aff/i/nb}$ is left exact; it follows that the restrictions of
\[ ?:\sh U^\mathrm{aff}_{/X}\rightarrow\mathbf{Sch}^\mathrm{aff}_{\F_1},  \]
to slice sets are flat, \emph{except} possibly for fibre products realised by $\emptyset\not\in\mathbf{FSch}_{\F_1}^\mathrm{aff/i/nb}$. We will need to tweak $\sh U^\mathrm{aff}_{/X}$ to correct for this. The tweaks may be separated into steps.

\paragraph{Step I}
Write $\sh U^\mathrm{aff/in}_{/X}\subset\sh U^\mathrm{aff}_{/X}$ for the subcategory of inhabited affine subsets of $X$. The restriction of $?$ to $\sh U^\mathrm{aff/in}_{/X}$ preserves fibre products and take all arrows to open immersions, but we will need to introduce new limits in order to get flatness.

\paragraph{Step II}The universal way to extend $?$ to a locally flat functor is to enlarge $\sh U^\mathrm{aff/in}_{/X}$ to the poset $\widetilde{\sh U}^\mathrm{aff}_{/X}$ of non-empty finite subsets of $\sh U^\mathrm{aff/in}_{/X}$ closed under fibre products. This poset has all fibre products, and the left Kan extension
\[ \mathrm{LKE} : \widetilde{\sh U}^\mathrm{aff}_{/X}\rightarrow \mathbf{Sch}_{\F_1}^\mathrm{aff},\quad \{U_i\}_{i\in I}\mapsto \lim_{i\in I}U_i \]
preserves these. Therefore $\mathrm{LKE}$ is locally flat. However, the new morphisms introduced between disjoint families $\{U_i\}$ are not open immersions.


\paragraph{Step III}
Since $K_X^\times$ is a constant sheaf, we have a map
\[ \G_X=\Spec K_X\rightarrow {}^?X \]
that factors through $\mathrm{LKE}$. Since $\G_X$ is a single point, $\G_X\rightarrow\mathrm{LKE}$ is an affine morphism. We will replace $\mathrm{LKE}$ with the embedded closure of $\G_X$, which is calculated as
\[ \{U_i\}_{i=1}^k\mapsto \Spec\left(\prod_{i=\in I}\sh O(U_i)\subseteq K_X\right). \]
The resulting functor $\widetilde{\sh U}^\mathrm{aff}_{/X}\rightarrow \mathbf{Sch}_{\F_1}^\mathrm{aff}$ preserves fibre products and open immersions. Therefore by lemma \ref{TOPOS_ATLAS}, it is an atlas for ${}^?X$, and ${}^?X$ is a scheme.\end{proof}

Let $X$ be an integral formal scheme, and let $x\in X(\F_1)$. There is an action
\[ \pi_1(X,x)\rightarrow\Aut(K_{X,x}^\times) \]
which obstructs algebraisation. Since $\Sigma_X(\R)$ has the same weak homotopy type as $X$, this can also be thought of as an action of $\pi_1(\Sigma_X(\R),\sigma_x)$ with $\sigma_x$ the cone corresponding under corollary \ref{PFAN_CONES_POINTS} to the point $x$. The kernel of the action corresponds to a covering space of $X$ on which $K_X^\times$ is trivialised. In particular, 

\begin{cor}\label{PFAN_ALG0}Every integral formal scheme admits an algebraisable covering space.\end{cor}

\begin{remark}Separatedness of a formal scheme by no means implies that its algebraisation will be separated. Indeed, any covering space of an algebraisable integral formal scheme is algebraisable, but the algebraisation will be separated if and only if the covering is trivial.\end{remark}

\bibliographystyle{alpha}
\bibliography{poly}

\end{document}